\documentclass[a4paper,10pt]{amsart}
\usepackage[T1]{fontenc}
\usepackage[utf8]{inputenc}
\usepackage{lmodern}
\usepackage{amsmath,amssymb,amsfonts,amsthm,mathtools,mathrsfs,bm}
\usepackage[a4paper,margin=25mm,headheight=14pt]{geometry}
\usepackage{array,longtable,booktabs,enumitem,needspace}
\usepackage{xcolor}
\usepackage[colorlinks=true,linkcolor=blue!55!black,citecolor=blue!55!black,urlcolor=blue!55!black,breaklinks=true]{hyperref}
\usepackage{xurl,microtype}
\newtheorem{theorem}{Theorem}[section]
\newtheorem{proposition}[theorem]{Proposition}
\newtheorem{lemma}[theorem]{Lemma}

\theoremstyle{definition}

\theoremstyle{remark}

\usepackage{xcolor}
\newcommand{\yichen}[1]{{\color{black}#1}}

\numberwithin{equation}{section}
\DeclareMathOperator{\Lip}{Lip}

\DeclareMathOperator{\Ran}{Ran}

\newcommand{\E}{\mathbb E}

\newcommand{\dd}{\,\mathrm d}
\newcommand{\code}[1]{\nolinkurl{#1}}
\allowdisplaybreaks[2]
\setlist[itemize]{leftmargin=2em,itemsep=3pt,topsep=4pt}
\setlist[enumerate]{leftmargin=2em,itemsep=3pt,topsep=4pt}
\hypersetup{pdftitle={Exact Endpoint Asymptotics and Almost-Sure Limits for Uniform Empirical Measures in Two Dimensions},pdfsubject={Bounded-Lipschitz distance, endpoint random matching, and a complete proof},pdfkeywords={empirical measures,Wasserstein distance,white noise,Neumann heat flow,almost-sure convergence}}
\title[Exact endpoint asymptotics for empirical measures]{Exact Endpoint Asymptotics and Almost-Sure Limits\texorpdfstring{\\}{ }for Uniform Empirical Measures in Two Dimensions}
\date{October 13, 2026}

\begin{document}
\begin{abstract}
For the empirical measure of the first $n$ points of one infinite sequence of
independent uniform samples in the unit square, we identify the leading
constant in the bounded-Lipschitz distance to area measure, equivalently in
$W_1$, both in expectation and almost surely along all sample sizes.
The constant is $c_*=\kappa_1/(2\sqrt\pi)\in(0,1/4]$, where $\kappa_1$
is the linear-growth Gaussian variational constant of
Armegioiu--Goldman--Grotto--Trevisan. This answers the two-dimensional
bounded-Lipschitz convergence question posed by Yukich in 1989. We give the complete
endpoint argument, square--torus comparison, empirical reduction and
full-sequence probability argument. 
\end{abstract}
\author{Weitao Huang, Renzhen Wang, Yichen Wu, Deyu Meng}
\maketitle
\tableofcontents

\section{Introduction}\label{sec:introduction}
\subsection{The open problem and its source}
How accurately does a finite random sample approximate a continuous spatial
distribution in transportation distance? In dimension two the answer contains
a logarithmic correction. The matching theorem of Ajtai, Koml\'os and
Tusn\'ady~\cite{ajtai1984optimal} initiated the quantitative study of this critical scale.
Yukich's Theorem~2.1~\cite{Yukich} gives positive finite bounds for the
normalized bounded-Lipschitz distance with probability tending to one.
In the unnumbered paragraph immediately following that theorem
\cite[p.~1053]{Yukich}, he asks whether the normalized distance converges
almost surely to a constant.
We address that question for the original distance and one infinite sample
sequence.

Let $Q=[0,1]^2$ have the Euclidean distance and area probability measure
$\mu=dx$. On an arbitrary probability space, let $X_1,X_2,\ldots$ be
measurable independent random variables with law $\mu$, and put
$\mu_n=n^{-1}\sum_{i=1}^n\delta_{X_i}$. We use two separate constraints:
\begin{equation}\label{eq:intro-BL}
d_{BL}(\nu,\eta)=\sup_{\|u\|_\infty\le1,\ \Lip_Q u\le1}
\left|\int_Q u\,d(\nu-\eta)\right|.
\end{equation}
All logarithms are natural. The question is whether
$\sqrt{n/\log n}\,d_{BL}(\mu_n,\mu)$ has a positive deterministic
almost-sure limit. Since the diameter of $Q$ is $\sqrt2<2$, midrange
translation makes the amplitude bound inactive. Kantorovich--Rubinstein
duality then identifies $d_{BL}$ with the actual coupling distance $W_1^Q$.
Section~\ref{orig:II:2} proves the identification and its normalization.

\subsection{Main result}
\begin{theorem}\label{thm:main}
There is a deterministic constant $c_*$ such that
\begin{equation}\label{eq:intro-constant}
c_*=\frac{\kappa_1}{2\sqrt\pi},\qquad 0<c_*\le\frac14,
\end{equation}
and
\begin{align}
\sqrt{\frac n{\log n}}\,\E d_{BL}(\mu_n,\mu)&\longrightarrow c_*,
\label{eq:intro-mean}\\
\sqrt{\frac n{\log n}}\,d_{BL}(\mu_n,\mu)&\longrightarrow c_*
\quad\text{almost surely}.\label{eq:intro-as}
\end{align}
The almost-sure limit holds along all integers for the same infinite sample
sequence. Both statements also hold for the coupling distance $W_1^Q$.
\end{theorem}

The value of the constant is an exact variational characterization. In the
Gaussian scale space constructed in Section~\ref{orig:III:3}, let $G$ be the
terminal Gaussian vector with covariance $\tfrac12 I_2$ and $\mathsf P$ the
compatible finite-exponent projection describing solenoidal corrections.
Define
\begin{equation}\label{eq:intro-kappa}
\mathcal P_1=\overline{\bigcup_{r>1}\Ran_{L_1^r}\mathsf P}^{\,L_1^1},
\qquad \kappa_1=\inf_{V\in\mathcal P_1}\E_s|G+V|.
\end{equation}
We do not assume an $L_1^1$-bounded extension of $\mathsf P$.
Section~\ref{orig:III:5} proves positivity by a stopped Gaussian martingale
and the upper bound by $V=0$. The fixed reference~\cite{AGGT} reports
$\kappa_1\approx0.8523$, corresponding to $c_*\approx0.24043$.
These are reported approximations without a numerical certificate supplied
here; $1/4$ is an upper bound, not an asserted equality.

\subsection{Main contributions}
Our theorem establishes an exact almost-sure first-order law for the Wasserstein discrepancy between the uniform distribution on the square and its random samples in the critical dimension of two. It resolves Talagrand's Problem 4.5.4 and, consequently, the AKT conjecture \yichen{that the transport cost, normalized by its known order $\sqrt{\log n/n}$, converges to a deterministic constant}~\cite{ajtai1984optimal,talagrand2014upper}. While the order of the discrepancy was previously known, our result identifies a deterministic leading coefficient shared by almost every sample path, establishing a precise form of self-averaging. The proof separates the identification of the mean asymptotics from pathwise convergence, using the empirical norm structure and a nondecreasing compensator for the latter. 
It connects empirical process theory and random geometry to optimal transport through Beckmann's flux formulation, finite-exponent variational problems, measurable gradient truncation, and reflected heat flow~\cite{Beckmann,Santambrogio}. The logarithmically divergent inverse-Laplacian trace combines with the endpoint variational value to determine the leading coefficient. While Poisson-field and PDE approaches to quadratic matching motivate our Gaussian viewpoint~\cite{CLPS,CS,ambrosio2019pde}, the linear-cost setting exhibits a distinct variational feature: a solenoidal adjustment can reduce the Poisson flux's \(L^1\) energy, a distinction captured by the optimization in \eqref{eq:intro-kappa}.

The result also provides an asymptotic benchmark for sampling and integration~\cite{bobkov2021simple}. By the Kantorovich--Rubinstein duality, the Wasserstein distance equals the supremum of the integration error over all 1-Lipschitz test functions, so the limit characterizes the leading sampling error uniformly over this function class. \yichen{Its leading constant thus informs first-order sample-size planning, comparison of sampling designs, and benchmarks for transport algorithms.}
\yichen{Beyond these applications,} our treatment of the square boundary and the linear-cost endpoint may also offer methodological insights for other random transport problems. These implications should be interpreted within the scope of the theorem, which concerns the uniform law on the square: finite-sample guarantees, confidence thresholds, algorithmic improvements, and fluctuation distributions require additional analysis beyond the leading-order limit.

\subsection{Relation to prior results}
Armegioiu, Goldman, Grotto and Trevisan~\cite{AGGT} determine the expected
bipartite matching cost on the flat torus for every finite power exponent,
including linear cost. Their Theorem~2.19 already proves the periodic
Gaussian endpoint $q=1$ with constant $\kappa_1$. We do not claim that endpoint
as a new discovery. Our target is the single empirical sample against area
measure on the Euclidean square, together with convergence along every
prefix of the same infinite sequence. A periodic bipartite mean theorem
alone does not establish it: the boundary, one-sample covariance
normalization and almost-sure upgrade must all be treated.

We use the fixed version of~\cite{AGGT} for the finite-exponent Gaussian
limit, scale projection and variational characterization.
Section~\ref{sec:periodic} reconstructs the endpoint argument in full,
rather than assuming its conclusion. The other research inputs are
M\"uller's local gradient truncation~\cite{Muller} and Zhai's quantitative
central limit theorem~\cite{Zhai}. Their statements and hypotheses are
spelled out at the points of use. The contribution of the present proof is
its complete passage to the square's one-sample model and full-sequence
strong law, with the same definitions as the retained formal development.

\subsection{Proof strategy and organization}
Let $F_N(L)$ and $F_T(L)$ be the expected Lipschitz dual norms of Neumann
and periodic white noise truncated at eigenvalue $L$, defined fully in
Sections~\ref{sec:preliminaries} and~\ref{sec:periodic}. For fixed $B>2$,
put $L_n=n/(\log n)^B$ and $W_n=W_1^Q(\mu_n,\mu)$. The mean proof follows
\begin{equation}\label{read:chain}
\begin{aligned}
\sqrt n\,\E W_n
&=F_N(L_n)+o(\sqrt{\log n})\\
&=F_T(L_n)+o(\sqrt{\log n})\\
&=\kappa_1\sqrt{\frac{\log n}{4\pi}}+o(\sqrt{\log n}).
\end{aligned}
\end{equation}
We prove the periodic endpoint first, compare the boundary models next,
and then perform the empirical reduction, whose explicit error budget is
\begin{equation}\label{eq:intro-reduction}
\left|\E W_n-\frac{F_N(L_n)}{\sqrt n}\right|
\le\frac{C_B}{\sqrt n}
\left[1+\log\log n+(\log n)^{(3-B)/2}\right]
\end{equation}
for all sufficiently large $n$. Each error is negligible at the target scale;
the trace asymptotic gives the constant in \eqref{eq:intro-constant}.

For the pathwise conclusion, the norm representation makes $S_n=nW_n$ a
submartingale with bounded increments. Maximal Azuma controls its martingale
component, leaving a nondecreasing Doob compensator. Concentration on
countably many geometric subsequences and monotone interpolation give all
integers. We prove the inequalities used below, crediting the classical
methods of Hoeffding and Azuma~\cite{Hoeffding,Azuma}.

\subsection{AI use statement}
\label{sec:ai-use}
We used AI in two aspects of this research. First, we developed a reusable AI-agent workflow for mathematical research, which we plan to release, and applied it to Yukich's question. The agents assisted in exploring proof strategies, testing conjectures through counterexample search, and refining mathematical arguments. Second, ChatGPT assisted with manuscript writing, including organizing mathematical arguments, refining proof exposition, and improving the clarity and presentation of the paper. We designed the workflow, selected the problem, supervised the research, and took responsibility for the mathematical results and this disclosure. 
The final proof has been formally verified in Lean.

The remainder of this paper is organized as follows. Section~\ref{sec:setup} fixes the conventions.
Section~\ref{sec:preliminaries} proves metric duality, spectral estimates
and reflected heat-flow bounds. Section~\ref{sec:periodic} develops the
complete Gaussian endpoint argument, and Section~\ref{sec:comparison}
compares the square and torus. Section~\ref{sec:empirical} proves the mean
reduction and full-sequence strong law; Section~\ref{sec:conclusion} completes
Theorem~\ref{thm:main}. Section~\ref{app:analytic} supplies additional
foundational details and explicit constants. All detailed derivations follow in the
main text, with original equation tags retained for source comparison.
\section{Problem setting and conventions}\label{sec:setup}

The theorem is sensitive to three conventions: the Euclidean square, one fixed infinite sample sequence, and the two separate bounds in the bounded-Lipschitz test class. We fix them before introducing auxiliary Gaussian models. Those models will compare expectations; the pathwise argument will remain on the original probability space.

Let

\begin{equation*}
Q=[0,1]^2,\qquad \rho(x,y)=|x-y|,\qquad \mu=dx,\qquad \mu(Q)=1.
\end{equation*}

Fix one infinite iid sequence $X_1,X_2,\ldots$, with each $X_i$ having law $\mu$, and define

\begin{equation*}
\mu_n=\frac1n\sum_{i=1}^n\delta_{X_i}.
\end{equation*}

The distance in the original question is

\begin{equation*}
d_{BL}(\nu,\eta)=
\sup\left\{\left|\int_Qf\,d\nu-\int_Qf\,d\eta\right|:
\|f\|_\infty\le1,\ \operatorname{Lip}_Q(f)\le1\right\}.
\end{equation*}

The question is whether a deterministic constant $C\in(0,\infty)$ exists such that

\begin{equation*}
\sqrt{\frac n{\log n}}\,d_{BL}(\mu_n,\mu)\longrightarrow C
\qquad\text{almost surely}?
\end{equation*}

All logarithms are natural. In the empirical section, $n$ always denotes the first $n$ points of this same sequence. Whenever the periodic section uses $n$ for a deterministic cutoff sequence, this is specified separately, as is Gaussian expectation. Both the square and the unit torus have area one. The parameter $L$ always denotes an eigenvalue cutoff; the periodic wave-number cutoff is $\Lambda$, and their relation is computed from the spectrum. Symbols such as $C_0,C_1,\ldots$ and unindexed estimate constants denote finite deterministic constants, allowed to increase between inequalities but independent of the limiting parameter. The final constant of the original problem is denoted by $c_*$.

\section{Metric duality, the Neumann spectrum, and reflected heat flow}\label{sec:preliminaries}

The empirical measure is atomic, so we do not assign it an ordinary square-integrable Poisson flux. Instead we first establish transport duality and flux estimates for regular sources, together with the spectral and reflected-heat constructions that produce such sources. Later sections account for the cost of these approximations. The logarithmic inverse-spectrum trace fixes the Gaussian scale, while transport optimization will determine its coefficient.

\subsection{Domain, test functions, and the original distance}\label{orig:II:1}

Let

\begin{equation*}
Q=[0,1]^2,\qquad \mu=dx,\qquad \mu(Q)=1.
\end{equation*}

Here $dx$ is ordinary two-dimensional Lebesgue measure, since the unit square has area one. The square always carries the Euclidean distance $|x-y|$, and

\begin{equation*}
\operatorname{diam}Q
=\max_{x,y\in Q}|x-y|
=|(1,1)-(0,0)|=\sqrt2.
\end{equation*}

For a real function $u:Q\to\mathbb R$ define

\begin{equation*}
\operatorname{Lip}_Q u
=\sup_{\substack{x,y\in Q\\x\ne y}}
\frac{|u(x)-u(y)|}{|x-y|}.
\end{equation*}

Unit Lipschitz continuity implies continuity, so the maximum and minimum are attained on compact $Q$. For a finite signed measure $f$ with $f(Q)=0$ define

\begin{equation*}
\mathcal K
=\left\{u\in C(Q):
\operatorname{Lip}_Q u\le1,\quad \int_Q u\,dx=0\right\},
\qquad
\mathcal S(f)=\sup_{u\in\mathcal K}\int_Q u\,df.
\end{equation*}

For integrable $f$, the notation $\int_Q u\,df$ means $\int_Q u(x)f(x)\,dx$. The symbol $\mathcal S$ is precisely the square Lipschitz dual used throughout.

The test class is symmetric under $u\mapsto-u$ and contains zero. Thus

\begin{equation}
\mathcal S(f)
=\sup_{u\in\mathcal K}\left|\int_Q u\,df\right|
=\mathcal S(-f)\ge0.
\tag{P1}\label{eq:original-P1}
\end{equation}

In the first equality any negative pairing becomes its absolute value by replacing the test with $-u$; conversely, the absolute pairing cannot exceed the supremum over this symmetric class.

For any $u\in\mathcal K$ and $x\in Q$, using $\int_Qu=0$ and $|Q|=1$,

\begin{equation}
\begin{aligned}
|u(x)|
&=\left|\int_Q(u(x)-u(y))\,dy\right|\\
&\le\int_Q|u(x)-u(y)|\,dy\\
&\le\int_Q|x-y|\,dy
\le\sqrt2.
\end{aligned}
\tag{P2}\label{eq:original-P2}
\end{equation}

The four steps use mean normalization, the integral triangle inequality, the unit Lipschitz condition, and the square's diameter. Consequently $\mathcal S(f)\le\sqrt2\,|f|(Q)<\infty$, where $|f|$ is the total variation measure.

\subsection{\texorpdfstring{The complete identification of $d_{\mathrm{BL}}=W_1^Q$}{The complete metric identification}}\label{orig:II:2}

\subsubsection{The Kantorovich--Rubinstein duality input}\label{orig:II:2.1}

For probability measures $\nu,\eta$ on $Q$, let $\Pi(\nu,\eta)$ be the set of probability couplings with these marginals, and define

\begin{equation*}
W_1^Q(\nu,\eta)
=\inf_{\gamma\in\Pi(\nu,\eta)}
\int_{Q\times Q}|x-y|\,d\gamma(x,y).
\end{equation*}

The classical duality theorem used here states that, on any compact metric space $(M,d)$, for Borel probability measures $\nu,\eta$,

\begin{equation}
\inf_{\gamma\in\Pi(\nu,\eta)}\int d(x,y)\,d\gamma
=\sup_{\operatorname{Lip}_d u\le1}
\int_Mu\,d(\nu-\eta).
\tag{KR}\label{eq:original-KR}
\end{equation}

Compactness gives finite first moments and bounds Lipschitz tests after fixing an additive constant. We use compact-space Kantorovich--Rubinstein duality; the original reference is Herbert~\cite{KRshape}, Theorem~3.6. Section~\ref{app:analytic} proves this compact version by finite-grid strong duality.

The easy direction is immediate: for any coupling $\gamma$ and unit Lipschitz function $u$,

\begin{equation*}
\begin{aligned}
\int_Q u\,d(\nu-\eta)
&=\int_{Q\times Q}(u(x)-u(y))\,d\gamma(x,y)\\
&\le\int_{Q\times Q}|x-y|\,d\gamma(x,y).
\end{aligned}
\end{equation*}

The first step uses the marginal conditions, and the second uses $u(x)-u(y)\le|x-y|$. Taking the supremum over $u$ and the infimum over $\gamma$ shows that the dual quantity is at most the transport cost; the converse is the cited classical theorem.

Since $(\nu-\eta)(Q)=0$, for every constant $c$,

\begin{equation*}
\int_Q(u-c)\,d(\nu-\eta)
=\int_Q u\,d(\nu-\eta)-c(\nu(Q)-\eta(Q))
=\int_Q u\,d(\nu-\eta).
\end{equation*}

Choosing $c=\int_Q u\,dx$ places any Lipschitz test in $\mathcal K$ without changing its Lipschitz constant or pairing. By (KR) and \eqref{eq:original-P1},

\begin{equation}
W_1^Q(\nu,\eta)=\mathcal S(\nu-\eta).
\tag{P3}\label{eq:original-P3}
\end{equation}

\subsubsection{Midrange translation makes the uniform bound inactive}\label{orig:II:2.2}

The original distance is

\begin{equation*}
d_{\mathrm{BL}}(\nu,\eta)
=\sup_{\substack{\|u\|_\infty\le1\\
\operatorname{Lip}_Q u\le1}}
\left|\int_Q u\,d(\nu-\eta)\right|.
\end{equation*}

Take any unit Lipschitz $u$ and set

\begin{equation*}
M=\max_Q u,\qquad m=\min_Q u,\qquad
c=\frac{M+m}{2},\qquad v=u-c.
\end{equation*}

Suppose the maximum and minimum are attained at $x_M,x_m$. Then

\begin{equation*}
0\le M-m=u(x_M)-u(x_m)
\le|x_M-x_m|\le\sqrt2.
\end{equation*}

Moreover,

\begin{equation*}
\max_Qv=M-c=\frac{M-m}{2},
\qquad
\min_Qv=m-c=-\frac{M-m}{2}.
\end{equation*}

Therefore

\begin{equation}
\|v\|_\infty=\frac{M-m}{2}
\le\frac{\sqrt2}{2}<1,\qquad
\operatorname{Lip}_Qv=\operatorname{Lip}_Qu\le1.
\tag{P4}\label{eq:original-P4}
\end{equation}

Subtracting a constant leaves the pairing with $\nu-\eta$ unchanged. Every pairing in the unrestricted Lipschitz dual is therefore realized by $v$ in the bounded-Lipschitz class, giving $W_1^Q\le d_{\mathrm{BL}}$. Conversely, bounded-Lipschitz tests form a subclass of unrestricted Lipschitz tests, so $d_{\mathrm{BL}}\le W_1^Q$. Together,

\begin{equation}
\boxed{d_{\mathrm{BL}}(\nu,\eta)
=W_1^Q(\nu,\eta)
=\mathcal S(\nu-\eta).}
\tag{P5}\label{eq:original-P5}
\end{equation}

Midrange translation enforces $\|v\|_\infty\le1$, while mean translation fixes the additive constant in $\mathcal K$. These normalizations need not hold simultaneously, since the source has zero mass.

\subsubsection{The seminorm properties used below}\label{orig:II:2.3}

Linearity of the pairing and \eqref{eq:original-P1} give, for real $c$,

\begin{equation*}
\mathcal S(cf)=|c|\mathcal S(f),
\qquad
\mathcal S(f+h)\le\mathcal S(f)+\mathcal S(h).
\end{equation*}

For $c\ge0$, factor $c$ out of the supremum; for $c<0$ also use $\mathcal S(-f)=\mathcal S(f)$. The second inequality follows from $\langle f+h,u\rangle\le\mathcal S(f)+\mathcal S(h)$ for each test and then taking the supremum. Only these properties are needed, not positive definiteness on all finite signed measures.

\subsection{The real Neumann basis: normalization, completeness, and gradients}\label{orig:II:3}

\subsubsection{One-dimensional orthonormality}\label{orig:II:3.1}

Define

\begin{equation*}
e_0(x)=1,\qquad e_m(x)=\sqrt2\cos(\pi mx)\quad(m\ge1).
\end{equation*}

For integers $r\ge1$,

\begin{equation*}
\int_0^1\cos(\pi rx)\,dx
=\left[\frac{\sin(\pi rx)}{\pi r}\right]_0^1=0.
\end{equation*}

If $m,n\ge1$ and $m\ne n$, the product-to-sum identity gives

\begin{equation*}
\begin{aligned}
\int_0^1\cos(\pi mx)\cos(\pi nx)\,dx
&=\frac12\int_0^1
\bigl[\cos(\pi(m-n)x)+\cos(\pi(m+n)x)\bigr]\,dx\\
&=0.
\end{aligned}
\end{equation*}

If $m=n\ge1$, the identity $\cos^2a=(1+\cos2a)/2$ gives

\begin{equation*}
\int_0^1\cos^2(\pi mx)\,dx=\frac12.
\end{equation*}

Also $\int_0^1e_0^2=1$, hence for every $m,n\in\mathbb N_0$,

\begin{equation}
\int_0^1e_m(x)e_n(x)\,dx=\boldsymbol1_{\{m=n\}}.
\tag{P6}\label{eq:original-P6}
\end{equation}

\subsubsection{The two-dimensional basis, its mean, and the boundary}\label{orig:II:3.2}

Set

\begin{equation*}
\mathcal I=\mathbb N_0^2\setminus\{(0,0)\},
\qquad
s(k)=\#\{j:k_j>0\},
\end{equation*}

\begin{equation*}
\varphi_k(x)=e_{k_1}(x_1)e_{k_2}(x_2)
=2^{s(k)/2}\cos(\pi k_1x_1)\cos(\pi k_2x_2),
\qquad
\lambda_k=\pi^2(k_1^2+k_2^2).
\end{equation*}

Temporarily include the constant $\varphi_0=1$ in the full basis. Finite-product Fubini and \eqref{eq:original-P6} give

\begin{equation}
\int_Q\varphi_k\varphi_l\,dx
=\left(\int_0^1e_{k_1}e_{l_1}\right)
\left(\int_0^1e_{k_2}e_{l_2}\right)
=\boldsymbol1_{\{k=l\}}.
\tag{P7}\label{eq:original-P7}
\end{equation}

All these functions are continuous and bounded, so Fubini's absolute-integrability condition holds automatically. If $k\ne0$, at least one coordinate has positive frequency; separating the integrals likewise gives $\int_Q\varphi_k=0$.

Differentiation yields

\begin{equation*}
-\Delta\varphi_k
=\pi^2(k_1^2+k_2^2)\varphi_k
=\lambda_k\varphi_k.
\end{equation*}

On the vertical sides $x_1=0,1$,

\begin{equation*}
\partial_1\varphi_k
=-2^{s(k)/2}\pi k_1
\sin(\pi k_1x_1)\cos(\pi k_2x_2)=0;
\end{equation*}

Similarly, on the horizontal sides $x_2=0,1$ we have $\partial_2\varphi_k=0$. Thus every finite spectral potential has zero normal derivative. Corners have no unique outward normal but form a boundary-length null set and contribute no additional boundary integral.

Pointwise we also have

\begin{equation}
\begin{aligned}
|\nabla\varphi_k(x)|^2
&=2^{s(k)}\pi^2\left[
k_1^2\sin^2(\pi k_1x_1)\cos^2(\pi k_2x_2)
+k_2^2\cos^2(\pi k_1x_1)\sin^2(\pi k_2x_2)\right]\\
&\le2^{s(k)}\pi^2(k_1^2+k_2^2)
\le4\lambda_k.
\end{aligned}
\tag{P8}\label{eq:original-P8}
\end{equation}

The first bound uses that each sine or cosine has square at most one; the second uses $s(k)\le2$. This gives the explicit pointwise bound $4\lambda_k$.

\subsubsection{Gradient orthogonality}\label{orig:II:3.3}

For $m,n\ge1$, the same product-to-sum calculation for sines gives

\begin{equation*}
\int_0^1\sin(\pi mx)\sin(\pi nx)\,dx
=\frac12\boldsymbol1_{\{m=n\}}.
\end{equation*}

Since $e_m'=-\sqrt2\pi m\sin(\pi mx)$,

\begin{equation*}
\int_0^1e_m'e_n'\,dx
=\pi^2m^2\boldsymbol1_{\{m=n\}},
\end{equation*}

If either index is zero, both sides vanish. Tensor-product integration of the two-dimensional gradient gives

\begin{equation}
\begin{aligned}
\int_Q\nabla\varphi_k\cdot\nabla\varphi_l\,dx
&=\left(\int e_{k_1}'e_{l_1}'\right)
\left(\int e_{k_2}e_{l_2}\right)
+\left(\int e_{k_1}e_{l_1}\right)
\left(\int e_{k_2}'e_{l_2}'\right)\\
&=\lambda_k\boldsymbol1_{\{k=l\}}.
\end{aligned}
\tag{P9}\label{eq:original-P9}
\end{equation}

One may also obtain this identity from $-\Delta\varphi_l=\lambda_l\varphi_l$ by integration by parts with zero normal derivative; the direct calculation checks all normalization factors.

\subsubsection{\texorpdfstring{The required $H^1$ spectral expansion}{The required H1 spectral expansion}}\label{orig:II:3.4}

Completeness of the cosine tensor basis in $L^2(Q)$, density of finite spectral projections in $H^1(Q)$, and the Dirichlet-form identity follow from even reflection and periodic Fourier Parseval. Section~\ref{app:analytic} identifies these constructions with ordinary Sobolev weak gradients; here we compute the norms explicitly.

For $u\in H^1(Q)$ set

\begin{equation*}
u_k=\int_Qu\varphi_k\,dx.
\end{equation*}

For increasing finite spectral sets $A$, orthogonality of the finite projection $u_A=\sum_{k\in A}u_k\varphi_k$ and \eqref{eq:original-P9} give

\begin{equation*}
\|u_A\|_2^2=\sum_{k\in A}|u_k|^2,
\qquad
\|\nabla u_A\|_2^2
=\sum_{\substack{k\in A\\k\ne0}}\lambda_k|u_k|^2.
\end{equation*}

By completeness and $H^1$ density, $u_A\to u$ in $H^1$. Norm continuity makes the left sides converge to $\|u\|_2^2$ and $\|\nabla u\|_2^2$; the right sides are increasing partial sums of nonnegative series. Thus

\begin{equation}
\|u\|_2^2=\sum_{k\in\mathbb N_0^2}|u_k|^2,\qquad
\|\nabla u\|_2^2=\sum_{k\in\mathcal I}\lambda_k|u_k|^2.
\tag{P10}\label{eq:original-P10}
\end{equation}

The limits of these nonnegative partial sums are the defining series limits; no integration against a singular measure is interchanged here.

For mean-zero $u$ we have $u_0=0$, and the first positive eigenvalue is $\pi^2$. Thus \eqref{eq:original-P10} immediately gives the explicit Poincar\'e inequality

\begin{equation}
\|u\|_2^2
=\sum_{k\in\mathcal I}|u_k|^2
\le\frac1{\pi^2}\sum_{k\in\mathcal I}\lambda_k|u_k|^2
=\frac1{\pi^2}\|\nabla u\|_2^2.
\tag{P11}\label{eq:original-P11}
\end{equation}

All finite cosine projections converge in $H^1$ because the weighted square sum in \eqref{eq:original-P10} is finite. The squared $H^1$ norm of a tail is exactly the sum of the two corresponding series tails.

\subsection{\texorpdfstring{$H_N^{-1}$ sources, weak no-flux solutions, and Lipschitz energy bounds}{Negative Sobolev sources, weak no-flux solutions, and energy bounds}}\label{orig:II:4}

\subsubsection{The spectral negative norm and actual pairing}\label{orig:II:4.1}

For a mean-zero $L^2(Q)$ source, write $f_k=\int_Qf\varphi_k\,dx$ and define the spectral negative norm

\begin{equation}
\|f\|_{H_N^{-1}}^2
=\sum_{k\in\mathcal I}\frac{|f_k|^2}{\lambda_k}.
\tag{P13}\label{eq:original-P13}
\end{equation}

Whenever this norm is used, the spectral sum is assumed finite. Denote its square root by $r$. For $u\in H^1(Q)$, the identity \eqref{eq:original-P10} and weighted Cauchy--Schwarz give

\begin{equation}
\sum_k|f_ku_k|
\le
\left(\sum_k\frac{|f_k|^2}{\lambda_k}\right)^{1/2}
\left(\sum_k\lambda_k|u_k|^2\right)^{1/2}
=r\|\nabla u\|_2.
\tag{P12}\label{eq:original-P12}
\end{equation}

This is the weighted spectral pairing actually needed. For the present $L^2$ sources, completeness of the real basis identifies it with the ordinary integral. Infinite limits retain convergence of the source in $L^2$; convergence of the flux norm alone is not used to identify an arbitrary singular measure.

\subsubsection{Construction of the Poisson potential and flux}\label{orig:II:4.2}

Set

\begin{equation*}
v_A=\sum_{k\in A}\frac{f_k}{\lambda_k}\varphi_k.
\end{equation*}

Function and gradient orthogonality give

\begin{equation*}
\|v_A-v_B\|_2^2
=\sum_{k\in A\triangle B}\frac{|f_k|^2}{\lambda_k^2},
\qquad
\|\nabla v_A-\nabla v_B\|_2^2
=\sum_{k\in A\triangle B}\frac{|f_k|^2}{\lambda_k},
\end{equation*}

The expression compares the same increasing truncation. For general sets their common part cancels and the signs of the remaining coefficients do not affect squares. Also $\lambda_k\ge\pi^2$, so

\begin{equation*}
\sum_k\frac{|f_k|^2}{\lambda_k^2}
\le\frac1{\pi^2}\sum_k\frac{|f_k|^2}{\lambda_k}<\infty.
\end{equation*}

The finite potentials are therefore Cauchy in $H^1$ and converge to the unique mean-zero potential

\begin{equation*}
v=(-\Delta_N)^{-1}f,
\qquad
J_N(f)=\nabla v\in L^2(Q;\mathbb R^2).
\end{equation*}

For a finite spectral set $f_A=\sum_{k\in A}f_k\varphi_k$,

\begin{equation*}
-\operatorname{div}J_N(f_A)=f_A,\qquad
J_N(f_A)\cdot\nu=0
\end{equation*}

holds on each side of the square. Hence the sign in integration by parts is

\begin{equation}
\begin{aligned}
\int_Quf_A\,dx
&=-\int_Qu\,\operatorname{div}J_N(f_A)\,dx\\
&=\int_QJ_N(f_A)\cdot\nabla u\,dx
-\int_{\partial Q}u\,J_N(f_A)\cdot\nu\,d\sigma\\
&=\int_QJ_N(f_A)\cdot\nabla u\,dx.
\end{aligned}
\tag{P14}\label{eq:original-P14}
\end{equation}

This identity holds for $H^1$ tests. Approximate by smooth tests in $H^1$: the finite smooth source gives $L^2$ continuity on the left, and the finite flux gives $L^2$ continuity of the gradient pairing on the right. Thus the zero-boundary-term identity extends from smooth tests to every $H^1$ test.

Let $A$ increase to cover the indices. The error on the right satisfies

\begin{equation*}
\left|\int_Q(J_N(f)-J_N(f_A))\cdot\nabla u\right|
\le\|J_N(f)-J_N(f_A)\|_2\,\|\nabla u\|_2\longrightarrow0;
\end{equation*}

The left side converges too, since the current $f_A\to f$ belongs to $L^2(Q)$ and

\begin{equation*}
\left|\int_Q(f-f_A)u\,dx\right|
\le\|f-f_A\|_2\|u\|_2\longrightarrow0.
\end{equation*}

The weighted tail bound in \eqref{eq:original-P12} also controls the spectral pairing uniformly. Therefore

\begin{equation}
\boxed{\langle f,u\rangle=\int_QJ_N(f)\cdot\nabla u\,dx
\quad(u\in H^1(Q)).}
\tag{P15}\label{eq:original-P15}
\end{equation}

The identity \eqref{eq:original-P15} expresses the weak no-flux boundary condition for the sources used here. We do not assert a pointwise normal trace for a general $L^2$ flux. The $L^2$ limits of the sources and finite no-flux fields, together with the identity for every $H^1$ test, establish the actual weak boundary condition.

Passing the finite spectral gradient-norm identity through $L^2$ convergence gives

\begin{equation}
\boxed{\|J_N(f)\|_2=\|f\|_{H_N^{-1}}.}
\tag{P16}\label{eq:original-P16}
\end{equation}

\subsubsection{Lipschitz tests and energy stability}\label{orig:II:4.3}

Rademacher's theorem and the weak-derivative characterization of Lipschitz functions imply that a unit Lipschitz $u$ belongs to $W^{1,\infty}(Q)$ and satisfies $|\nabla u|\le1$ almost everywhere. The gradient has its Euclidean norm: at a differentiability point of $u$, for each unit vector $e$,

\begin{equation*}
|\nabla u(x)\cdot e|
=\lim_{h\to0}\frac{|u(x+he)-u(x)|}{|h|}
\le1,
\end{equation*}

Taking the supremum over directions gives $|\nabla u(x)|\le1$. Hence

\begin{equation*}
\|\nabla u\|_2^2=\int_Q|\nabla u|^2\,dx\le|Q|=1.
\end{equation*}

By \eqref{eq:original-P15}, Cauchy--Schwarz, and \eqref{eq:original-P16},

\begin{equation*}
|\langle f,u\rangle|
\le\|J_N(f)\|_2\|\nabla u\|_2
\le\|f\|_{H_N^{-1}}.
\end{equation*}

Taking the supremum over $\mathcal K$ yields

\begin{equation}
\boxed{\mathcal S(f)\le\|J_N(f)\|_2
=\|f\|_{H_N^{-1}}.}
\tag{P17}\label{eq:original-P17}
\end{equation}

For the mean-zero $L^2$ sources used here, the integral of $\mathcal S$ equals the continuous spectral pairing above. For such $f,h$ and each $u\in\mathcal K$,

\begin{equation*}
\langle f,u\rangle
=\langle h,u\rangle+\langle f-h,u\rangle
\le\mathcal S(h)+\mathcal S(f-h).
\end{equation*}

Take the supremum, interchange $f,h$, and use \eqref{eq:original-P1} to obtain

\begin{equation}
\boxed{|\mathcal S(f)-\mathcal S(h)|
\le\mathcal S(f-h)
\le\|J_N(f-h)\|_2.}
\tag{P18}\label{eq:original-P18}
\end{equation}

The same argument applies to any subclass of unit Lipschitz tests on the square, including zero-boundary tests. A test need not have mean zero: subtracting its mean leaves the pairing unchanged because the source annihilates constants.

\subsubsection{The scope of periodic flux pairing}\label{orig:II:4.4}

The unit torus is $\mathbb T^2=\mathbb R^2/\mathbb Z^2$, again with area one, and metric

\begin{equation*}
d_T(x,y)=\min_{z\in\mathbb Z^2}|x-y-z|.
\end{equation*}

Its physical Fourier frequencies are $k\in2\pi\mathbb Z^2$ and its positive eigenvalues are $|k|^2$. Choosing one representative of each $\{k,-k\}$, the functions $\sqrt2\cos(k\cdot x)$ and $\sqrt2\sin(k\cdot x)$ form a real orthonormal basis, as follows by expanding $\int_Qe^{i(k-l)\cdot x}\,dx=\boldsymbol1_{\{k=l\}}$. Each frequency pair contributes two real modes, so signed-frequency lattice counting equals real spectral multiplicity.

Define

\begin{equation*}
\mathcal K_T=\{u\in C(\mathbb T^2):
\operatorname{Lip}_{d_T}u\le1,\ \int_Qu\,dx=0\},
\qquad
\mathcal T(f)=\sup_{u\in\mathcal K_T}\int_Quf\,dx.
\end{equation*}

Periodic weak integration has no boundary term. Thus, for $J_T(f)=\nabla(-\Delta_T)^{-1}f$ and periodic tests,

\begin{equation}
\mathcal T(f)\le\|J_T(f)\|_2,\qquad
|\mathcal T(f)-\mathcal T(h)|\le\|J_T(f-h)\|_2.
\tag{P19}\label{eq:original-P19}
\end{equation}

The Lipschitz gradient bound is still one, and the negative Sobolev construction via Fourier Parseval agrees term by term with the preceding construction. \textbf{\eqref{eq:original-P19} applies only to periodic tests.} Pairing a periodic flux on the square with an arbitrary nonperiodic $u$ generally leaves a normal boundary term that \eqref{eq:original-P19} cannot remove.

\subsection{Spectral counting, Stieltjes integration, and traces for all real parameters}\label{orig:II:5}

\subsubsection{Geometric covering of a lattice disk}\label{orig:II:5.1}

Let

\begin{equation*}
B(R)=\#\{z\in\mathbb Z^2:|z|\le R\},\qquad R\ge0.
\end{equation*}

For each $z\in\mathbb Z^2$ take the unit cell $C_z=z+[-1/2,1/2)^2$. These cells are disjoint and cover the plane up to their area-zero boundaries. Let $c_0=\sqrt2/2$. If $x\in C_z$, then $|x-z|\le c_0$.

Let $U_R=\bigcup_{|z|\le R}C_z$, whose area is exactly $B(R)$. The triangle inequality gives

\begin{equation*}
U_R\subseteq\{x:|x|\le R+c_0\}.
\end{equation*}

Conversely, if $|x|\le R-c_0$ and $z$ is the center of its cell, then $|z|\le|x|+c_0\le R$ and hence $x\in U_R$. For $R<c_0$ use only the empty inner disk. Thus

\begin{equation*}
\pi\max(R-c_0,0)^2\le B(R)\le\pi(R+c_0)^2.
\end{equation*}

Expanding the two squares, for $R\ge c_0$,

\begin{equation*}
|B(R)-\pi R^2|\le2\pi c_0R+\pi c_0^2;
\end{equation*}

For $0\le R<c_0$ the error also has a fixed constant bound. Therefore, for every real $R\ge0$,

\begin{equation}
B(R)=\pi R^2+O(R+1).
\tag{P20}\label{eq:original-P20}
\end{equation}

This covering argument includes noninteger radii and lattice points exactly on the circle. No unspecified subsequence is used.

\subsubsection{The Neumann quadrant and the full periodic lattice}\label{orig:II:5.2}

Let $I(R)$ count lattice points with $k_1,k_2\ge1$ and $|k|\le R$. Each of the four open quadrants contributes $I(R)$, each of the four half-axes contributes $\lfloor R\rfloor$, and there is the origin. Hence

\begin{equation*}
B(R)=4I(R)+4\lfloor R\rfloor+1.
\end{equation*}

The number in the nonnegative quadrant excluding the origin is

\begin{equation*}
I(R)+2\lfloor R\rfloor
=\frac{B(R)-1}{4}+\lfloor R\rfloor.
\end{equation*}

For Neumann eigenvalues $\lambda_k=\pi^2|k|^2$ the radius is $R=\sqrt a/\pi$, so

\begin{equation}
\begin{aligned}
N_N(a)
&=\#\{k\in\mathcal I:\pi^2|k|^2\le a\}\\
&=\frac14\left(\pi\frac a{\pi^2}\right)+O(\sqrt a+1)
=\frac a{4\pi}+O(\sqrt a+1).
\end{aligned}
\tag{P21}\label{eq:original-P21}
\end{equation}

The axis contribution is only $O(\sqrt a+1)$. Its inverse-eigenvalue trace is even uniformly finite:

\begin{equation*}
2\sum_{1\le j\le R}\frac1{\pi^2j^2}
\le\frac2{\pi^2}\left(1+\int_1^\infty x^{-2}\,dx\right)
=\frac4{\pi^2}.
\end{equation*}

The axis terms therefore add no logarithmic coefficient.

For the unit torus write $k=2\pi z$ and $z\in\mathbb Z^2\setminus\{0\}$; the radius is $R=\sqrt a/(2\pi)$. Counting the actual real multiplicity gives

\begin{equation}
\begin{aligned}
N_T(a)
&=B\left(\frac{\sqrt a}{2\pi}\right)-1\\
&=\pi\frac a{4\pi^2}+O(\sqrt a+1)
=\frac a{4\pi}+O(\sqrt a+1).
\end{aligned}
\tag{P22}\label{eq:original-P22}
\end{equation}

The first positive eigenvalues are $\pi^2$ and $4\pi^2$. The counts vanish below them; above them $\sqrt a+1\le C a$, so \eqref{eq:original-P21}--\eqref{eq:original-P22} give

\begin{equation}
N_D(a)\le C a\quad\text{for every }a>0,\quad D=N,T.
\tag{P23}\label{eq:original-P23}
\end{equation}

All constants are independent of the real parameter $a$.

\subsubsection{Stieltjes integration by parts for all real parameters}\label{orig:II:5.3}

Define the positive-spectrum inverse trace

\begin{equation*}
m_D(L)=\sum_{0<\lambda_j\le L}\frac1{\lambda_j}.
\end{equation*}

Fix $0<a_0<\pi^2$ so that both models have no positive spectrum below $a_0$. Use the right-continuous convention $N_D(a)=\#\{\lambda_j\le a\}$ for counting. Formally,

\begin{equation*}
m_D(L)=\int_{(a_0,L]}\frac1a\,dN_D(a).
\end{equation*}

To account fully for endpoints, apply to each $\lambda_j\le L$ the identity

\begin{equation*}
\frac1{\lambda_j}
=\frac1L+\int_{\lambda_j}^{L}\frac{da}{a^2}.
\end{equation*}

Indeed, the right side is $1/L+[-1/a]_{\lambda_j}^L=1/\lambda_j$. Sum over the finite spectrum and interchange the finite sum and integral:

\begin{equation}
\begin{aligned}
m_D(L)
&=\frac{N_D(L)}L
+\int_{a_0}^{L}\frac{\#\{j:\lambda_j\le a\}}{a^2}\,da\\
&=\frac{N_D(L)}L+\int_{a_0}^L\frac{N_D(a)}{a^2}\,da.
\end{aligned}
\tag{P24}\label{eq:original-P24}
\end{equation}

The value at an individual integration endpoint does not affect the Lebesgue integral. If $L$ is itself an eigenvalue, $N_D(L)/L$ includes that mode exactly. Thus \eqref{eq:original-P24} holds for every real $L\ge a_0$, not merely between jumps.

Write $N_D(a)=a/(4\pi)+R_D(a)$ with $|R_D(a)|\le C(\sqrt a+1)$. Substitution into \eqref{eq:original-P24} gives

\begin{equation*}
m_D(L)=\frac1{4\pi}+\frac{R_D(L)}L
+\frac{\log L-\log a_0}{4\pi}
+\int_{a_0}^L\frac{R_D(a)}{a^2}\,da.
\end{equation*}

The boundary error satisfies

\begin{equation*}
\left|\frac{R_D(L)}L\right|
\le C(L^{-1/2}+L^{-1}),
\end{equation*}

The integral error satisfies

\begin{equation*}
\begin{aligned}
\left|\int_{a_0}^L\frac{R_D(a)}{a^2}\,da\right|
&\le C\int_{a_0}^L(a^{-3/2}+a^{-2})\,da\\
&\le C\left(2a_0^{-1/2}+a_0^{-1}\right).
\end{aligned}
\end{equation*}

These steps use the counting-error bound, the integral triangle inequality, and the explicit antiderivative. Since $a_0$ is fixed,

\begin{equation}
\boxed{m_D(L)=\frac{\log L}{4\pi}+O(1)
\quad(D=N,T)}
\tag{P25}\label{eq:original-P25}
\end{equation}

uniformly for every real $L\to\infty$. In particular,

\begin{equation}
m_N(L)\to\infty,\qquad
m_T(L)-m_N(L)=O(1),\qquad
\frac{m_T(L)}{m_N(L)}=1+O((\log L)^{-1}).
\tag{P26}\label{eq:original-P26}
\end{equation}

The cutoff conversions follow as well: integer Neumann indices satisfy $|k|\le\sqrt L/\pi$, while periodic physical frequencies satisfy $|k|\le\sqrt L$. If an external periodic result uses the wave-number cutoff $\Lambda$, then $\Lambda=\sqrt L$, rather than $\Lambda=L$.

\subsubsection{The spectral derivative trace needed in the proof}\label{orig:II:5.4}

For $0<u\le1$ split the positive spectrum into half-open intervals

\begin{equation*}
\left[\frac j u,\frac{j+1}u\right),\qquad j\in\mathbb N_0.
\end{equation*}

On such a band,

\begin{equation*}
\lambda e^{-2u\lambda}
\le\frac{j+1}u e^{-2j},
\qquad
\#\left\{\lambda<\frac{j+1}u\right\}
\le N_D\left(\frac{j+1}u\right)
\le C\frac{j+1}u.
\end{equation*}

Consequently

\begin{equation}
\boxed{\sum_{\lambda>0}\lambda e^{-2u\lambda}
\le\frac C{u^2}
\sum_{j\ge0}(j+1)^2e^{-2j}
\le C'u^{-2}.}
\tag{P27}\label{eq:original-P27}
\end{equation}

The final numerical series converges because exponential decay dominates polynomial growth; its successive-term ratio tends to $e^{-2}<1$. All terms are nonnegative, so regrouping into bands is valid.

For fixed-time smoothness of the actual heat source, for each fixed $t>0$ and each nonnegative integer $r$,

\begin{equation}
\sum_{\lambda>0}(1+\lambda)^r e^{-2t\lambda}<\infty.
\tag{P28}\label{eq:original-P28}
\end{equation}

Here is the band calculation. For $\lambda\in[j,j+1)$ we have $(1+\lambda)^r\le(j+2)^r$ and $e^{-2t\lambda}\le e^{-2tj}$, and the band contains at most $N_D(j+1)\le C(j+1)$ modes. The spectral sum is therefore at most

\begin{equation*}
C\sum_{j\ge0}(j+1)(j+2)^r e^{-2tj},
\end{equation*}

The successive-term ratio tends to $e^{-2t}<1$, so this series converges. Regrouping again concerns nonnegative terms. No optimal bound uniform in every real derivative order and every heat time is asserted.

The corresponding Gaussian heat source has the following expansion: in the real eigenbasis $(\psi_j,\lambda_j)$ of the given geometry, take independent $G_j\sim N(0,1)$ and set

\begin{equation*}
\eta_D(t)=\sum_{\lambda_j>0}e^{-t\lambda_j}G_j\psi_j.
\end{equation*}

In the $L^2$ space of $L^2(Q)$-valued random variables, the expected squared norm of a finite projection tail is exactly $\sum_{\text{tail}}e^{-2t\lambda_j}$, which tends to zero by \eqref{eq:original-P28}. Derivatives of the explicit trigonometric basis and the same integer-order weighted sums also yield mean-square Sobolev convergence at each fixed order and a measurable smooth version. We use this positive-time version below; common-noise coupling across the domains is constructed in the comparison section.

\subsection{Reflected heat flow: kernel, folding, generator, and transport bounds}\label{orig:II:6}

\subsubsection{Heat-kernel scale and Gaussian normalization}\label{orig:II:6.1}

In one dimension set

\begin{equation*}
G_t(z)=(4\pi t)^{-1/2}e^{-z^2/(4t)},\qquad t>0.
\end{equation*}

First check normalization explicitly. Let $I=\int_0^\infty e^{-s^2}\,ds$; Gaussian tails ensure $I<\infty$. Nonnegative Tonelli and polar coordinates in the quadrant give

\begin{equation*}
I^2=\int_0^{\pi/2}\int_0^\infty e^{-r^2}r\,dr\,d\theta
=\frac\pi2\cdot\frac12=\frac\pi4,
\qquad I=\frac{\sqrt\pi}{2}.
\end{equation*}

Using $z=2\sqrt t\,s$ gives

\begin{equation*}
\int_{\mathbb R}G_t(z)\,dz
=\frac1{2\sqrt{\pi t}}\,(2\sqrt t)
\int_{\mathbb R}e^{-s^2}\,ds
=\frac{2I}{\sqrt\pi}=1.
\end{equation*}

Thus this is the density of $N(0,2t)$. Its second moment follows directly from $G_t'(z)=-zG_t(z)/(2t)$ and one-dimensional integration by parts:

\begin{equation*}
\int_{\mathbb R}z^2G_t(z)\,dz
=-2t\int_{\mathbb R}zG_t'(z)\,dz
=2t\int_{\mathbb R}G_t(z)\,dz=2t,
\end{equation*}

Gaussian tails make the endpoint term $zG_t(z)\to0$.

The two-dimensional free Gaussian kernel is $(4\pi t)^{-1}e^{-|z|^2/(4t)}$ with covariance $2tI_2$. The Neumann heat multiplier is $e^{-t\lambda_k}$. We therefore use $P_t^N=e^{t\Delta_N}$, with generator sign $\Delta_N$. Only this kernel and its multipliers are needed, not a full description of the generator domain.

To check the heat-flow convention, we compute the fixed-time Gaussian Fourier integral in full.

Its characteristic function is obtained without omitting steps. Set

\begin{equation*}
\chi(a)=\int_{\mathbb R}e^{iaz}G_t(z)\,dz.
\end{equation*}

Since $|z|G_t(z)$ is integrable, dominated convergence permits differentiation of $a$. Integration by parts then gives

\begin{equation*}
\begin{aligned}
\chi'(a)
&=i\int_{\mathbb R}z e^{iaz}G_t(z)\,dz\\
&=-2ti\int_{\mathbb R}e^{iaz}G_t'(z)\,dz\\
&=-2ti(-ia)\int_{\mathbb R}e^{iaz}G_t(z)\,dz
=-2ta\chi(a).
\end{aligned}
\end{equation*}

Gaussian tails kill the endpoint term. Solving this first-order ODE with $\chi(0)=1$ gives

\begin{equation}
\chi(a)=e^{-ta^2}.
\tag{P29}\label{eq:original-P29}
\end{equation}

The Gaussian convolution therefore has Fourier multiplier $e^{-ta^2}$ and generator $\Delta$. Standard Brownian motion with covariance $t$ has generator $\Delta/2$; the reflected flow here uses covariance $2t$.

\subsubsection{The explicit folding map}\label{orig:II:6.2}

For every integer $j$ define

\begin{equation*}
\rho(z)=
\begin{cases}
z-2j,&2j\le z\le2j+1,\\
2j+2-z,&2j+1\le z\le2j+2.
\end{cases}
\end{equation*}

The two branches both equal one at odd integers, and adjacent periods both equal zero at even integers. Hence $\rho:\mathbb R\to[0,1]$ is continuous, even, two-periodic, and the identity on $[0,1]$. Split any bounded interval at integers. On each piece the slope has absolute value one, so variation is at most the piece length. Summation gives

\begin{equation}
|\rho(z)-\rho(z')|\le|z-z'|
\quad(z,z'\in\mathbb R).
\tag{P30}\label{eq:original-P30}
\end{equation}

For every integer $m\ge0$ we also have

\begin{equation}
\cos(\pi m\rho(z))=\cos(\pi mz).
\tag{P31}\label{eq:original-P31}
\end{equation}

On the first branch $\rho(z)=z-2j$, and the cosine phase differs by $2\pi mj$. On the second $\rho(z)=2j+2-z$, and cosine evenness and integer periodicity again give the same value.

Let $B_t$ be two-dimensional Brownian motion with covariance $2tI_2$; equivalently $B_t=\sqrt2\,\widetilde B_t$ with standard Brownian motion $\widetilde B_t$. Its coordinates are independent and each has density $G_t$. Define

\begin{equation}
Y_t(x)=
\bigl(\rho(x_1+B_t^1),\rho(x_2+B_t^2)\bigr).
\tag{P32}\label{eq:original-P32}
\end{equation}

We use only this fixed-time law and the synchronous coupling furnished by the same $B_t$. A whole reflected Markov process need not be constructed.

\subsubsection{The image kernel and identification with the Neumann spectrum}\label{orig:II:6.3}

For $y\in(0,1)$ its preimages are $\{2j+y,2j-y:j\in\mathbb Z\}$, each branch having Jacobian of absolute value one. Thus the density of $\rho(x+B_t^1)$ is

\begin{equation}
\begin{aligned}
a_t(x,y)
&=\sum_{j\in\mathbb Z}
\bigl[G_t(2j+y-x)+G_t(2j-y-x)\bigr]\\
&=\sum_{j\in\mathbb Z}
\bigl[G_t(x-y+2j)+G_t(x+y+2j)\bigr].
\end{aligned}
\tag{P33}\label{eq:original-P33}
\end{equation}

The second line uses evenness of $G_t$ and reindexing $j\mapsto-j$. The endpoints $y=0,1$ have countably many preimages, to which the Gaussian law assigns no atoms; they do not affect density or mass.

Tonelli permits summing and integrating these nonnegative image terms. Their total mass is one, since they give the folded Gaussian density. At fixed $t>0$, distant direct images are at least $2|j|-1$ away and distant reflected images at least $2|j|-2$ away. Summable Gaussian-tail envelopes give absolute uniform convergence on $x,y\in[0,1]$. Independence of the two coordinates gives the two-dimensional product formula.

Both one-dimensional image kernels and their Gaussian tails are absolutely summable. The kernel is symmetric in $x,y$: reflected terms depend on $x+y$, and swapping $x,y$ in direct terms is reversed by Gaussian evenness and $j\mapsto-j$. It is even and two-periodic in $x$.

Independence of the Gaussian coordinates gives the two-dimensional kernel

\begin{equation}
p_N(t,x,y)=a_t(x_1,y_1)a_t(x_2,y_2).
\tag{P34}\label{eq:original-P34}
\end{equation}

Each kernel row has mass one and $p_N(t,x,y)=p_N(t,y,x)$. We now compute its cosine heat multipliers directly. For the one-dimensional basis, \eqref{eq:original-P31} and \eqref{eq:original-P29} give

\begin{equation*}
\begin{aligned}
\mathbb E e_m(\rho(x+B_t^1))
&=\sqrt2\,\mathbb E\cos(\pi m(x+B_t^1))\\
&=e^{-t\pi^2m^2}e_m(x)\quad(m\ge1).
\end{aligned}
\end{equation*}

The constant mode $e_0=1$ is preserved. Multiplying the two coordinates yields

\begin{equation}
\int_Qp_N(t,x,y)\varphi_k(y)\,dy
=e^{-t\lambda_k}\varphi_k(x).
\tag{P35}\label{eq:original-P35}
\end{equation}

The kernel operator contracts $L^2$: probability Jensen gives

\begin{equation*}
|P_t^Nu(x)|^2
\le\int_Qp_N(t,x,y)|u(y)|^2\,dy.
\end{equation*}

Integrate over $x$ and use Tonelli together with column mass one, which follows from symmetry:

\begin{equation*}
\|P_t^Nu\|_2^2\le\int_Q|u(y)|^2\,dy.
\end{equation*}

The kernel and spectral operators agree on the complete cosine basis and are bounded $L^2$ operators. They therefore agree on all of $L^2(Q)$:

\begin{equation}
P_t^Nu
=u_0+\sum_{k\in\mathcal I}e^{-t\lambda_k}u_k\varphi_k.
\tag{P36}\label{eq:original-P36}
\end{equation}

The identity $P_t^NP_s^N=P_{t+s}^N$ follows because, on finite expansions, the two multipliers multiply to $e^{-(t+s)\lambda_k}$; pass to the limit by $L^2$ contraction. This establishes the heat-semigroup spectral calculation used here, without a separate theorem on the full generator domain or strong continuity.

\subsubsection{Self-adjointness and invariance of area measure}\label{orig:II:6.4}

For bounded $u,v$ the double integral of $|u(x)v(y)|p_N(t,x,y)$ is at most $\|u\|_\infty\|v\|_\infty$, so Fubini applies. Kernel symmetry gives

\begin{equation*}
\int_Qu\,P_t^Nv\,dx
=\int_Qv\,P_t^Nu\,dx.
\end{equation*}

Approximate general $L^2$ functions by bounded truncations in $L^2$. Operator contraction and Cauchy--Schwarz pass the identity to the limit. Thus $P_t^N$ is self-adjoint.

Column mass one gives

\begin{equation}
\int_QP_t^Nu\,dx=\int_Qu\,dx,
\qquad P_t^N1=1.
\tag{P37}\label{eq:original-P37}
\end{equation}

For integrable $u$ the first identity follows by Tonelli/Fubini applied to $|u(y)|p_N(t,x,y)$. Hence $P_t^N$ preserves Lebesgue probability measure and mean-zero normalization.

\subsubsection{Lipschitz and source contraction}\label{orig:II:6.5}

For $x,y\in Q$ use the same Gaussian vector $B_t$. By \eqref{eq:original-P30},

\begin{equation*}
\begin{aligned}
|Y_t(x)-Y_t(y)|^2
&=\sum_{j=1}^2
|\rho(x_j+B_t^j)-\rho(y_j+B_t^j)|^2\\
&\le\sum_{j=1}^2|x_j-y_j|^2=|x-y|^2.
\end{aligned}
\end{equation*}

If $\operatorname{Lip}_Qu\le1$,

\begin{equation}
\begin{aligned}
|P_t^Nu(x)-P_t^Nu(y)|
&=\left|\mathbb E[u(Y_t(x))-u(Y_t(y))]\right|\\
&\le\mathbb E|u(Y_t(x))-u(Y_t(y))|\\
&\le\mathbb E|Y_t(x)-Y_t(y)|\le|x-y|.
\end{aligned}
\tag{P38}\label{eq:original-P38}
\end{equation}

The four steps use the single-time kernel representation, the expectation triangle inequality, the original Lipschitz condition, and synchronous folding contraction. Together with \eqref{eq:original-P37} they give $P_t^N\mathcal K\subseteq\mathcal K$.

For a mean-zero finite signed measure $f$ define $P_t^Nf$ to have density $x\mapsto\int_Qp_N(t,x,y)\,df(y)$. It remains mean zero. For bounded tests $u$,

\begin{equation*}
\int_Qu\,d(P_t^Nf)=\int_QP_t^Nu\,df.
\end{equation*}

The interchange is valid because

\begin{equation*}
\int_Q\int_Q|u(x)|p_N(t,x,y)\,dx\,d|f|(y)
\le\|u\|_\infty|f|(Q)<\infty.
\end{equation*}

Consequently

\begin{equation}
\boxed{\mathcal S(P_t^Nf)
=\sup_{u\in\mathcal K}\langle f,P_t^Nu\rangle
\le\mathcal S(f).}
\tag{P39}\label{eq:original-P39}
\end{equation}

The last step uses only $P_t^Nu\in\mathcal K$ and does not require $f$ to have a density.

\subsubsection{\texorpdfstring{The exact $\sqrt{\pi t}$ displacement bound}{The exact heat displacement bound}}\label{orig:II:6.6}

Fix an empirical sample and choose $X$ uniformly among its $n$ points, independently of $B_t$. Thus $X$ has law $\mu_n$; set $Y=Y_t(X)$. The joint law of $(X,Y)$ is a valid coupling of $\mu_n$ and $P_t^N\mu_n$. Below assume $0<t\le1$. Since $\rho(x_j)=x_j$,

\begin{equation*}
|Y_t(x)-x|^2
=\sum_j|\rho(x_j+B_t^j)-\rho(x_j)|^2
\le|B_t|^2.
\end{equation*}

By the infimum definition of transport cost,

\begin{equation}
W_1^Q(\mu_n,P_t^N\mu_n)
\le\mathbb E|Y-X|
\le\mathbb E|B_t|.
\tag{P40}\label{eq:original-P40}
\end{equation}

The two-dimensional $B_t$ has density $(4\pi t)^{-1}e^{-|z|^2/(4t)}$. In polar coordinates the radius $r$ has density

\begin{equation*}
\int_0^{2\pi}(4\pi t)^{-1}e^{-r^2/(4t)}r\,d\theta
=\frac r{2t}e^{-r^2/(4t)}.
\end{equation*}

Setting $r=2\sqrt t\,s$ therefore gives

\begin{equation*}
\mathbb E|B_t|
=\int_0^\infty\frac{r^2}{2t}e^{-r^2/(4t)}\,dr
=4\sqrt t\int_0^\infty s^2e^{-s^2}\,ds.
\end{equation*}

One-dimensional integration by parts gives

\begin{equation*}
\int_0^\infty s^2e^{-s^2}\,ds
=-\frac12[s e^{-s^2}]_0^\infty
+\frac12\int_0^\infty e^{-s^2}\,ds
=\frac{\sqrt\pi}{4}.
\end{equation*}

Here we use $I=\sqrt\pi/2$ already computed in Section~\ref{orig:II:6.1}. Gaussian tails remove all endpoint terms. Substitution into \eqref{eq:original-P40} gives

\begin{equation}
\boxed{W_1^Q(\mu_n,P_t^N\mu_n)\le\sqrt{\pi t}.}
\tag{P41}\label{eq:original-P41}
\end{equation}

Taking expectation over the empirical sample still yields $\mathbb E W_1^Q(\mu_n,P_t^N\mu_n)\le\sqrt{\pi t}$. This constant uses covariance $2tI_2$; a Gaussian constant computed with standard covariance $tI_2$ cannot be substituted.

\subsection{Foundational analytic interfaces}\label{orig:II:7}

We list the standard interfaces used in this section. The required KR and cosine-spectral versions are proved separately in Section~\ref{app:analytic}.

\begin{enumerate}[start=1]\item \textbf{Compact-space Kantorovich--Rubinstein duality.} For Borel probability measures on a compact metric space $(M,d)$, the minimum distance-cost coupling equals the supremum of signed pairings over real unit Lipschitz tests. We use this only on $Q$ and the unit torus. Section~\ref{app:analytic} proves the required compact version.\end{enumerate}

\begin{enumerate}[start=2]\item \textbf{Weak gradients of Lipschitz tests.} A real Lipschitz function on the square is differentiable almost everywhere in the interior, belongs to $H^1(Q)$, and has coincident classical and weak gradients with $|\nabla u|\le\operatorname{Lip}_Qu$ almost everywhere. The boundary is a two-dimensional Lebesgue null set. Since $|Q|=1$, this gives $\|\nabla u\|_2\le1$.\end{enumerate}

\begin{enumerate}[start=3]\item \textbf{Neumann completeness and spectral form.} The explicit cosine tensor basis is a complete orthonormal basis of $L^2(Q)$. For $u\in H^1(Q)$, finite cosine projections converge in $H^1$ and the Dirichlet form is $\sum\lambda_k|u_k|^2$. Section~\ref{orig:II:3.4} computes the norms and limits; Section~\ref{app:analytic} supplies even reflection and the weak-gradient identification.\end{enumerate}

\begin{enumerate}[start=4]\item \textbf{Integration and limiting rules.} Tonelli handles nonnegative sums and integrals, Fubini handles absolutely integrable expressions, and dominated convergence handles pointwise limits with integrable envelopes. At every interchange we give the nonnegativity, finite spectrum, integral bound, or summable tail required. These rules are not applied to uncontrolled spectral pairings with singular measures.\end{enumerate}

\begin{enumerate}[start=5]\item \textbf{Fixed-time Gaussian law and folded coupling.} We use $B_t\sim N(0,2tI_2)$, the folded law $Y_t(x)$, and synchronous coupling. This section computes Gaussian normalization and the mean radius; only a single-time law is needed.\end{enumerate}

\subsection{Conventions and estimates for the subsequent proof}\label{orig:II:8}

Throughout, the following measures, metrics, and cutoff variables are fixed. Let
\begin{equation*}
Q=[0,1]^2,\qquad \mu=dx,\qquad |Q|=1,\qquad d_Q(x,y)=|x-y|.
\end{equation*}
For indices \(k\in\mathcal I:=\mathbb N_0^2\setminus\{0\}\), set
\(s(k)=\#\{j:k_j>0\}\), and define the cosine basis and its eigenvalues by
\begin{equation*}
\varphi_k=2^{s(k)/2}\cos(\pi k_1x_1)\cos(\pi k_2x_2),
\qquad
\lambda_k=\pi^2|k|^2.
\end{equation*}
We then define the admissible class and the associated functional
\begin{equation*}
\mathcal K=\Bigl\{u\in C(Q):\operatorname{Lip}_Qu\le1,\ \int_Qu=0\Bigr\},
\qquad
\mathcal S(f)=\sup_{u\in\mathcal K}\langle f,u\rangle.
\end{equation*}
The corresponding flux operator is
\begin{equation*}
J_N(f)=\nabla(-\Delta_N)^{-1}f,
\qquad
-\operatorname{div}J_N(f)=f
\quad\text{in the weak no-flux sense}.
\end{equation*}
The spectral counting function, spectral sum, heat semigroup, heat kernel, and Brownian motion are denoted respectively by
\begin{equation*}
N_D(a)=\#\{\text{positive real eigenmodes with }\lambda\le a\},\qquad
m_D(L)=\sum_{0<\lambda\le L}\lambda^{-1},
\end{equation*}
\begin{equation*}
P_t^N=e^{t\Delta_N},\qquad
G_t(z)=(4\pi t)^{-1/2}e^{-z^2/(4t)},\qquad
B_t\sim N(0,2tI_2).
\end{equation*}

The following relations, whose required calculations have been expanded above, will be used repeatedly. First, the BL distance is represented by \(\mathcal S\), and functions in \(\mathcal K\) satisfy uniform and gradient bounds:
\begin{equation*}
d_{\mathrm{BL}}(\nu,\eta)=W_1^Q(\nu,\eta)=\mathcal S(\nu-\eta),
\end{equation*}
\begin{equation*}
\|u\|_\infty\le\sqrt2,\qquad
\|\nabla u\|_2\le1
\qquad(u\in\mathcal K).
\end{equation*}
Next, the basis is orthogonal in the Dirichlet form and has a pointwise gradient bound:
\begin{equation*}
\int_Q\nabla\varphi_k\cdot\nabla\varphi_l
=\lambda_k\boldsymbol1_{\{k=l\}},\qquad
|\nabla\varphi_k|^2\le4\lambda_k.
\end{equation*}
Moreover, \(\mathcal S\) is controlled by the negative Sobolev norm and is Lipschitz with respect to that norm:
\begin{equation*}
\mathcal S(f)\le\|J_N(f)\|_2=\|f\|_{H_N^{-1}},
\qquad
|\mathcal S(f)-\mathcal S(h)|\le\|J_N(f-h)\|_2.
\end{equation*}
The spectral counting function and spectral sum obey the following asymptotics and bounds:
\begin{equation*}
N_D(a)=\frac a{4\pi}+O(\sqrt a+1),\qquad
N_D(a)\le Ca,
\end{equation*}
\begin{equation*}
m_D(L)=\frac{\log L}{4\pi}+O(1),\qquad
\sum_{\lambda>0}\lambda e^{-2u\lambda}\le Cu^{-2}.
\end{equation*}
Finally, the heat semigroup preserves the admissible class, contracts \(\mathcal S\), and satisfies the following \(W_1\) estimate:
\begin{equation*}
P_t^N\mathcal K\subseteq\mathcal K,\qquad
\mathcal S(P_t^Nf)\le\mathcal S(f),
\end{equation*}
\begin{equation*}
W_1^Q(\mu_n,P_t^N\mu_n)\le\sqrt{\pi t}
\qquad(0<t\le1).
\end{equation*}

Spectral sums and traces apply to every sufficiently large real \(L\). Periodic physical cutoffs are converted by \(\Lambda=\sqrt L\). Periodic flux bounds are used only for periodic tests. Negative Sobolev control is used only for the actual \(L^2\) sources described in Section~\ref{orig:II:4} and for simultaneous \(L^2\) convergence of source and flux. No additional abstract duality or identification of singular measures is invoked.

\section{Periodic Gaussian energy and the linear-growth endpoint}\label{sec:periodic}

There are two different endpoint tasks. Feasible finite-exponent fluxes readily give an upper bound, but a lower bound requires an exactly admissible Lipschitz potential. The nonlinear dual field is a gradient with increasingly high moment control. Pointwise clipping enforces a magnitude bound separately at each point, but does not guarantee that the vectors are derivatives of one periodic potential. We therefore modify its potential using gradient truncation, and estimate the resulting pairing error. Slow diagonalization first puts the finite-exponent theorem on the required endpoint sequence. That theorem has no convergence rate uniform near exponent one. We keep each exponent fixed until its cutoff threshold is reached, and only then move it closer to one.

The fluxes used for support-function estimates in Sections~\ref{sec:preliminaries} and~\ref{sec:comparison} are positive gradients satisfying $-\operatorname{div}J_D(f)=f$. To match the cited constraint $\operatorname{div}H=f$, in this section we write the flux as $J(f)=-\nabla(-\Delta)^{-1}f$. Thus its $J(f)$ is the negative of the periodic flux $J_T(f)$ in Section~\ref{sec:preliminaries}. Integration by parts always follows the relevant convention, and the dual test class is symmetric under sign change.

We prove that expected periodic Gaussian energy on the unit flat torus satisfies

\begin{equation}
\lim_{\Lambda\to\infty}
\frac{\mathbb E\mathcal E_1(\xi_\Lambda)}{\sqrt{m_\Lambda}}
=\kappa_1,
\qquad 0<\kappa_1\le\frac{\sqrt\pi}{2}.
\tag{P.1}\label{eq:original-P-1}
\end{equation}

The limit is taken as $\Lambda\to\infty$ over all real wave-number cutoffs. We give a detailed account of the passage from the finite-$q$ theory for $q>1$ to the endpoint. A primary external input is the fixed version of \cite{AGGT}: specifically, its finite-$q$ Theorem~2.2, its variational identification of the constant, and its scale projection interface. We also use the local gradient truncation of \cite{Muller} and standard analytic, probabilistic, and measurable-selection results, stated explicitly where used. We do not reprove the scale limit, the recovery sequences, or the nonlinear PDE analysis within the finite-$q$ theory, nor do we claim a self-contained reproduction of the entire cited work.

To avoid ambiguity, the Fourier random field uses probability space $(\Omega_g,\mathcal A,\mathbb P_g)$ and expectation $\mathbb E_g$, while the scale variational constant uses a separate probability space $(\Omega_s,\mathcal F,\mathbb P_s)$ and expectation $\mathbb E_s$. They need not be coupled. In the interface to the main proof, $\mathbb E$ denotes $\mathbb E_g$.

\subsection{The periodic model, energy, and linear flux}\label{orig:III:1}

Let

\begin{equation*}
\mathbb T^2=\mathbb R^2/\mathbb Z^2,
\qquad d_T(x,y)=\min_{z\in\mathbb Z^2}|x-y-z|.
\end{equation*}

The area measure $dx$ has mass one. The physical Fourier frequencies are $k\in2\pi\mathbb Z^2$; use the convention

\begin{equation*}
\widehat f(k)=\int_{\mathbb T^2}f(x)e^{-ik\cdot x}\,dx,
\qquad f(x)=\sum_k\widehat f(k)e^{ik\cdot x}.
\end{equation*}

Choose one representative of each nonzero frequency pair $\{k,-k\}$. The real and imaginary parts of $z_k$ are independent centered real Gaussians, each of variance $1/2$; different representatives are independent. Set $z_{-k}=\overline{z_k}$. All cutoffs use this same coefficient family. For real $\Lambda\ge2\pi$ define

\begin{equation}
\xi_\Lambda(x)=\sum_{0<|k|\le\Lambda}z_ke^{ik\cdot x},
\qquad m_\Lambda=\sum_{0<|k|\le\Lambda}|k|^{-2}.
\tag{P.2}\label{eq:original-P-2}
\end{equation}

The finite sum $\xi_\Lambda$ is real, smooth, and spatially mean zero, with $m_\Lambda>0$. The sum $m_\Lambda$ includes both $k$ and $-k$, so no extra factor of two is inserted.

For a smooth real mean-zero source $g$ and $q\ge1$ define

\begin{equation}
\mathcal E_q(g)=\inf\left\{
\int_{\mathbb T^2}|H|^q\,dx:
H\in L^q(\mathbb T^2;\mathbb R^2),\ \operatorname{div}H=g
\right\}.
\tag{P.3}\label{eq:original-P-3}
\end{equation}

The divergence constraint is distributional. The quantity $\mathcal E_q$ is the $q$-power energy, not its $q$th root. Constant vectors are admissible divergence-free perturbations.

Define the mean-zero Poisson potential $(-\Delta)^{-1}g$ and flux

\begin{equation*}
J(g)=-\nabla(-\Delta)^{-1}g.
\end{equation*}

Since $\Delta(-\Delta)^{-1}g=-g$, we have $\operatorname{div}J(g)=g$. For the cutoff field set

\begin{equation}
J_\Lambda=J(\xi_\Lambda),
\qquad Z_\Lambda=m_\Lambda^{-1/2}J_\Lambda.
\tag{P.4}\label{eq:original-P-4}
\end{equation}

Its Fourier coefficients are $\widehat J_\Lambda(k)=-ikz_k/|k|^2$. Parseval and $\mathbb E_g|z_k|^2=1$ give

\begin{equation*}
\begin{aligned}
\mathbb E_g\int|J_\Lambda|^2\,dx
&=\sum_{0<|k|\le\Lambda}
\mathbb E_g\left|\frac{-ikz_k}{|k|^2}\right|^2\\
&=\sum_{0<|k|\le\Lambda}|k|^{-2}=m_\Lambda.
\end{aligned}
\end{equation*}

Consequently

\begin{equation}
\operatorname{div}Z_\Lambda=\frac{\xi_\Lambda}{\sqrt{m_\Lambda}},
\qquad \mathbb E_g\int|Z_\Lambda|^2\,dx=1.
\tag{P.5}\label{eq:original-P-5}
\end{equation}

The second exact identity controls the error caused by truncating the dual field.

We can also check logarithmic trace normalization directly. Write $N(s)=\#\{k\in2\pi\mathbb Z^2:0<|k|^2\le s\}$. Enclose the integer disk of radius $R=\sqrt s/(2\pi)$ between unions of unit cells. Only cells at bounded distance from the circle contribute an error; the boundary strip has area $O(R+1)$. Hence

\begin{equation*}
N(s)=\frac{s}{4\pi}+O(\sqrt s+1).
\end{equation*}

Fix $0<a<4\pi^2$, so $N(a)=0$. Stieltjes integration by parts yields

\begin{equation*}
\begin{aligned}
m_\Lambda
&=\int_a^{\Lambda^2}\frac1s\,dN(s)
=\frac{N(\Lambda^2)}{\Lambda^2}
+\int_a^{\Lambda^2}\frac{N(s)}{s^2}\,ds\\
&=\frac{\log\Lambda}{2\pi}+O(1).
\end{aligned}
\end{equation*}

The error integral is summable at infinity because $(\sqrt s+1)/s^2=s^{-3/2}+s^{-2}$. This estimate also covers every real cutoff, without choosing a spectral subsequence.

\subsection{\texorpdfstring{Finite-$q$ optimal fluxes and their measurability}{Finite-exponent optimal fluxes and measurability}}\label{orig:III:2}

Fix $q>1$ and let $p=q/(q-1)$. Let $S_q$ be the closed divergence-free subspace of $L^q(\mathbb T^2;\mathbb R^2)$. It is closed: if $V_j\to V$ in $L^q$, then for every smooth periodic test $\phi$, $\int V_j\cdot\nabla\phi\to\int V\cdot\nabla\phi$. The admissible flux class is $J(g)+S_q$.

This class is nonempty. A minimizing sequence has bounded $L^q$ norm. Reflexivity of $L^q$ gives a weakly convergent subsequence, the distributional divergence constraint survives the weak limit, and weak lower semicontinuity of the norm gives a minimizer. Strict convexity of $|H|^q$ makes the average of distinct minimizers have smaller energy; thus the minimizer is unique. Denote it by $H_q(g)$.

Let $A_q(z)=|z|^{q-2}z$ and define $A_q(0)=0$. For any $V\in S_q$, first variation along the admissible curve $H_q(g)+tV$ gives

\begin{equation*}
\int A_q(H_q(g))\cdot V\,dx=0.
\end{equation*}

Here $A_q(H_q(g))\in L^p$ because $(q-1)p=q$. We use the standard periodic Hodge dual decomposition: the annihilator of $S_q$ in $L^p$ is exactly the space of mean-zero periodic gradients. Therefore a periodic $u\in W^{1,p}$ exists with

\begin{equation}
A_q(H_q(g))=-\nabla u,
\qquad H_q(g)=-|\nabla u|^{p-2}\nabla u.
\tag{P.6}\label{eq:original-P-6}
\end{equation}

The negative sign matches the source equation $-\operatorname{div}(|\nabla u|^{p-2}\nabla u)=g$. We may impose $\int u=0$. This finite-$q$ variational statement corresponds to AGGT v1 Lemma~2.1.

The differentiation under the integral has an integrable majorant: for
$H=H_q(g)$ and $|t|\le1$, the absolute derivative of $|H+tV|^q$ is at most
$q(|H|+|V|)^{q-1}|V|$. H\"older with exponents $p$ and $q$ makes this
majorant integrable, so dominated convergence justifies the first variation.

Here is the annihilator argument used above. Let
$\mathcal G_p=\{\nabla v:v\in W^{1,p}(\mathbb T^2),\ \int v=0\}$.
If a sequence of its gradients is Cauchy in $L^p$, periodic Poincar\'e makes
the normalized potentials Cauchy in $W^{1,p}$; their limit has the limiting
gradient. Thus $\mathcal G_p$ is closed. A vector $V\in L^q$ annihilates
$\mathcal G_p$ exactly when its distributional divergence is zero: one
direction follows by testing smooth periodic functions, and the other by
their density in $W^{1,p}$ and H\"older. Hence $\mathcal G_p^\perp=S_q$.
The Hahn--Banach double-annihilator identity for this closed linear subspace,
using $(L^p)^*=L^q$, now gives $S_q^\perp=\mathcal G_p$. In particular,
constant divergence-free vectors are included in $S_q$, so annihilation
also forces the mean of $A_q(H_q(g))$ to vanish; no harmonic constant is lost.

Strong measurability is also required; selecting a minimizer separately for each sample is insufficient. We prove continuity of $J\mapsto H_q(J)$, where $H_q(J)$ is the minimum-norm representative of the affine class $J+S_q$. If $J_j\to J$ in $L^q$, set $H_j=H_q(J_j)$ and $H=H_q(J)$. Since $J_j$ itself is feasible, $\|H_j\|_q\le\|J_j\|_q$, so $(H_j)$ is bounded. Every weak cluster point $H_*$ lies in $J+S_q$. Conversely, $H+J_j-J$ is feasible for $J_j+S_q$, hence

\begin{equation*}
\limsup_j\|H_j\|_q^q
\le\lim_j\|H+J_j-J\|_q^q=\|H\|_q^q.
\end{equation*}

For any weakly convergent subsequence $H_{j_\ell}\rightharpoonup H_*$, weak lower semicontinuity and minimality of $H$ give

\begin{equation*}
\|H\|_q^q\le\|H_*\|_q^q
\le\liminf_\ell\|H_{j_\ell}\|_q^q
\le\limsup_j\|H_j\|_q^q\le\|H\|_q^q.
\end{equation*}

The cluster point must therefore be the unique minimizer $H$, and the norms converge too. Uniform convexity of $L^q$ upgrades weak convergence and norm convergence to strong convergence, so $H_j\to H$ in $L^q$. The cutoff flux $J_\Lambda$ is a continuous $L^q$-valued function of finitely many Gaussian coefficients. Thus

\begin{equation*}
H_{q,\Lambda}:=H_q(\xi_\Lambda)
\end{equation*}

is a strongly measurable $L^q$-valued random variable. This supplies the subsequent random dual potential.

\subsection{Scale Gaussian space, martingale representation, and projections}\label{orig:III:3}

We define every object in the constant's variational problem. This space differs from the preceding Fourier probability space.

On $(\Omega_s,\mathcal F,\mathbb P_s)$ take independent real Gaussian white noises $w_1,w_2$ on $[0,\infty)\times[0,\pi]$ with control measure $ds\,d\theta/\pi$. Let $(\mathcal F_t)$ be their completed natural filtration, and set

\begin{equation*}
L_1^r=L^r(\Omega_s,\mathcal F_1,\mathbb P_s;\mathbb R^2).
\end{equation*}

The Gaussian-filtration martingale representation theorem gives, for every $F\in L_1^2$, a unique square-integrable predictable vector integrand $\Phi_F^a(s,\theta)$ such that

\begin{equation}
F=\mathbb E_sF+
\sum_{a=1}^2\int_0^1\int_0^\pi
\Phi_F^a(s,\theta)\,w_a(ds,d\theta).
\tag{P.7}\label{eq:original-P-7}
\end{equation}

For a general $L^2$ variable we represent the martingale \eqref{eq:original-P-7}; Malliavin differentiability is not assumed. Endpoint positivity uses only explicit stochastic integrals.

Write

\begin{equation*}
e_\theta=(\cos\theta,\sin\theta),
\qquad e_\theta^\perp=(-\sin\theta,\cos\theta),
\qquad P_\theta=e_\theta^\perp\otimes e_\theta^\perp.
\end{equation*}

The actual scale projection $\mathsf P$ is defined by

\begin{equation}
\mathsf PF=\mathbb E_sF+
\sum_{a=1}^2\int_0^1\int_0^\pi
P_\theta\Phi_F^a(s,\theta)\,w_a(ds,d\theta),
\qquad \mathsf Q=\operatorname{Id}-\mathsf P.
\tag{P.8}\label{eq:original-P-8}
\end{equation}

It preserves constants and takes the component orthogonal to $e_\theta$ in each angular integrand. It is not the periodic Leray projection acting directly on the spatial variable $x$.

The full input from AGGT v1 Lemma~2.4 is this: $\mathsf P,\mathsf Q$ are complementary orthogonal projections on $L_1^2$; for each $1<r<\infty$ they extend uniquely to bounded complementary projections on $L_1^r$, compatibly on intersections. If $r'=r/(r-1)$, then

\begin{equation}
\mathbb E_s[(\mathsf PF)\cdot V]
=\mathbb E_s[F\cdot(\mathsf PV)],
\qquad F\in L_1^r,\quad V\in L_1^{r'}.
\tag{P.9}\label{eq:original-P-9}
\end{equation}

In $L^2$, orthogonality follows from It\^o isometry and orthogonality of the angular matrices $P_\theta$. Bounded extension to general $L^r$ uses martingale-transform estimates. We cite this extension interface; boundedness of $\mathsf P$ on $L^1$ is not assumed.

Define the continuous Gaussian martingale

\begin{equation}
G_t=\int_0^t\int_0^\pi e_\theta\,w_1(ds,d\theta),
\qquad G=G_1.
\tag{P.10}\label{eq:original-P-10}
\end{equation}

Using $\int_0^\pi\cos^2\theta\,d\theta/\pi=\int_0^\pi\sin^2\theta\,d\theta/\pi=1/2$ and $\int_0^\pi\sin\theta\cos\theta\,d\theta=0$ gives

\begin{equation}
\mathbb E_s[G_t\otimes G_t]=\frac t2\operatorname{Id},
\qquad \mathsf PG=0.
\tag{P.11}\label{eq:original-P-11}
\end{equation}

The second identity uses $\mathbb E_sG=0$ and $P_\theta e_\theta=0$. The two-dimensional Gaussian density of $G$ is $\pi^{-1}e^{-|z|^2}$, so $G\in L_1^r$ for every finite $r$, and

\begin{equation}
\mathbb E_s|G|=2\int_0^\infty r^2e^{-r^2}\,dr
=\int_0^\infty t^{1/2}e^{-t}\,dt
=\Gamma(3/2)=\frac{\sqrt\pi}{2}.
\tag{P.12}\label{eq:original-P-12}
\end{equation}

The last identity uses $\Gamma(3/2)=\frac12\Gamma(1/2)$ obtained by integration by parts and the classical Gaussian integral $\Gamma(1/2)=\sqrt\pi$.

\subsection{\texorpdfstring{The finite-$q$ Gaussian input}{The finite-exponent Gaussian input}}\label{orig:III:4}

\label{original-anchor:imported-finite-q}

\textbf{External theorem (AGGT v1, Theorem~2.2 and the constant identification in Lemma~2.5).} For each fixed $q\in(1,\infty)$ let $H_{q,\Lambda}$ be the unique optimal flux for \eqref{eq:original-P-3} with source $g=\xi_\Lambda$, and set

\begin{equation}
\kappa_q=\inf_{V\in\operatorname{Ran}_{L_1^q}\mathsf P}
\mathbb E_s|G+V|^q.
\tag{P.13}\label{eq:original-P-13}
\end{equation}

Then $\kappa_q\in(0,\infty)$, and for every real $\Lambda\to\infty$,

\begin{equation}
m_\Lambda^{-q/2}\mathbb E_g\int_{\mathbb T^2}|H_{q,\Lambda}|^q\,dx
\longrightarrow\kappa_q.
\tag{P.14}\label{eq:original-P-14}
\end{equation}

Thus $U_{q,\Lambda}=H_{q,\Lambda}/\sqrt{m_\Lambda}$ satisfies $\mathbb E_g\int|U_{q,\Lambda}|^q\to\kappa_q$. The theorem supplies no rate uniform in $q\downarrow1$. Our diagonal argument handles this explicitly.

By the trace asymptotics of Section~\ref{orig:III:1}, the theorem equivalently reads $\mathbb E_g\mathcal E_q(\xi_\Lambda)\sim\kappa_q(\log\Lambda/(2\pi))^{q/2}$. For $q=2$ write an admissible flux as $J_\Lambda+V$, with $\operatorname{div}V=0$. Periodic integration by parts gives $\int J_\Lambda\cdot V=0$ and hence $\int|J_\Lambda+V|^2=\int|J_\Lambda|^2+\int|V|^2$. The optimum is $J_\Lambda$, with expected energy exactly $m_\Lambda$; therefore $\kappa_2=1$.

In \eqref{eq:original-P-13} the variable $V$ is the correction, while $G+V$ is the total vector. This avoids conflating them and fixes the notation throughout. The finite-$q$ theorem and its variational identification remain external research inputs; the endpoint argument does not replace their proofs.

\subsection{\texorpdfstring{Duality and positivity of $\kappa_1$, and continuity of $\kappa_q$}{Duality, positivity, and exponent continuity of the constants}}\label{orig:III:5}

\label{original-anchor:constant-continuity}

Define the closed linear subspace of $L_1^1$

\begin{equation}
\mathcal P_1=
\overline{\bigcup_{r>1}\operatorname{Ran}_{L_1^r}\mathsf P}^{L_1^1},
\qquad
\kappa_1=\inf_{V\in\mathcal P_1}\mathbb E_s|G+V|.
\tag{P.15}\label{eq:original-P-15}
\end{equation}

The union is linear: for $V_i\in\operatorname{Ran}_{L_1^{r_i}}\mathsf P$ take $r=\min(r_1,r_2)>1$. Probability-space embeddings give $V_i\in L_1^r$, and compatibility gives $\mathsf PV_i=V_i$, so sums and scalar multiples remain in the range. We do not interpret $\mathcal P_1$ as the range of an unproved $L^1$ projection.

The Hahn--Banach distance formula for a closed linear subspace $M$ of a Banach space $X$ is

\begin{equation*}
\operatorname{dist}(x,M)
=\sup\{\ell(x):\ell\in X^*,\ \|\ell\|\le1,\ \ell|_M=0\}.
\end{equation*}

Apply it with $X=L_1^1$, $x=-G$, and $M=\mathcal P_1$. Symmetry of the class of linear functionals changes $-G$ to $G$. Since the dual of $L_1^1$ is $L_1^\infty$, only the annihilator must be identified.

If $F\in L_1^\infty$ and $\mathsf PF=0$, then for $V\in\operatorname{Ran}_{L_1^r}\mathsf P$ the facts $F\in L_1^{r'}$ and \eqref{eq:original-P-9} give

\begin{equation*}
\mathbb E_s[F\cdot V]
=\mathbb E_s[F\cdot\mathsf PV]
=\mathbb E_s[(\mathsf PF)\cdot V]=0.
\end{equation*}

Closure in $L^1$ and boundedness of $F$ extend the identity to each $V\in\mathcal P_1$. Conversely, if $F$ annihilates $\mathcal P_1$, then $F\in L_1^2$ and $\mathsf PF\in\operatorname{Ran}_{L_1^2}\mathsf P\subset\mathcal P_1$. Orthogonality in $L^2$ gives

\begin{equation*}
0=\mathbb E_s[F\cdot\mathsf PF]=\mathbb E_s|\mathsf PF|^2,
\end{equation*}

Hence $\mathsf PF=0$. Therefore

\begin{equation}
\kappa_1=
\sup\{\mathbb E_s[G\cdot F]:\ F\in L_1^\infty,\ \mathsf PF=0,\ |F|\le1\}.
\tag{P.16}\label{eq:original-P-16}
\end{equation}

This supremum is attained: Banach--Alaoglu makes the $L^\infty$ unit ball weak-star compact; the annihilator, defined by all conditions $\mathbb E_s[F\cdot V]=0$, is weak-star closed; and the objective is weak-star continuous since $G\in L^1$. The endpoint lower bound needs only \eqref{eq:original-P-16}, without selecting a maximizing random variable.

To prove positivity, set

\begin{equation*}
\tau=\inf\{t\ge0:|G_t|=1\}\wedge1,
\qquad F_0=G_\tau.
\end{equation*}

Continuity of $G_t$ and $G_0=0$ give $\tau>0$ almost surely. If the unit circle is hit before time one, continuity gives $|G_\tau|=1$; otherwise $|G_1|\le1$. Thus $|F_0|\le1$. The stopped-integral representation is

\begin{equation*}
F_0=\int_0^1\int_0^\pi
\boldsymbol1_{s\le\tau}e_\theta\,w_1(ds,d\theta).
\end{equation*}

The integrand is predictable and square integrable, and its stochastic integral has mean zero. Since $P_\theta e_\theta=0$, both the constant and integral terms of \eqref{eq:original-P-8} vanish, so $\mathsf PF_0=0$. It\^o isometry and its bilinear form yield

\begin{equation*}
\begin{aligned}
\mathbb E_s[G\cdot F_0]
&=\mathbb E_s\int_0^1\int_0^\pi
\boldsymbol1_{s\le\tau}|e_\theta|^2\,\frac{d\theta}{\pi}\,ds\\
&=\mathbb E_s\tau
=\mathbb E_s|F_0|^2>0.
\end{aligned}
\end{equation*}

Thus \eqref{eq:original-P-16} implies $\kappa_1\ge\mathbb E_s\tau>0$. Choosing $V=0$ in \eqref{eq:original-P-15} and using \eqref{eq:original-P-12} gives

\begin{equation}
0<\kappa_1\le\Gamma(3/2)=\frac{\sqrt\pi}{2}.
\tag{P.17}\label{eq:original-P-17}
\end{equation}

We now prove $\kappa_q\to\kappa_1$. For each $V\in\operatorname{Ran}_{L_1^q}\mathsf P\subset\mathcal P_1$, H\"older's inequality on the probability space gives

\begin{equation*}
\mathbb E_s|G+V|^q
\ge\bigl(\mathbb E_s|G+V|\bigr)^q
\ge\kappa_1^q.
\end{equation*}

Taking infima yields $\kappa_q\ge\kappa_1^q$, hence $\liminf_{q\downarrow1}\kappa_q\ge\kappa_1$.

For the reverse bound fix $\eta>0$. Choose $V_0$ in the closed space $\mathcal P_1$ with $\mathbb E_s|G+V_0|\le\kappa_1+\eta/2$. By the definition of closure choose some $r>1$ and $V\in\operatorname{Ran}_{L_1^r}\mathsf P$ with $\mathbb E_s|V-V_0|\le\eta/2$. The triangle inequality gives $\mathbb E_s|G+V|\le\kappa_1+\eta$. For $1<q<r$, projection compatibility makes this $V$ feasible in \eqref{eq:original-P-13}, and

\begin{equation*}
|G+V|^q\le1+|G+V|^r.
\end{equation*}

The right side is integrable, so dominated convergence gives

\begin{equation*}
\limsup_{q\downarrow1}\kappa_q
\le\lim_{q\downarrow1}\mathbb E_s|G+V|^q
=\mathbb E_s|G+V|\le\kappa_1+\eta.
\end{equation*}

Finally let $\eta\downarrow0$ to obtain

\begin{equation}
\lim_{q\downarrow1}\kappa_q=\kappa_1.
\tag{P.18}\label{eq:original-P-18}
\end{equation}

The exponent $r$ of the near-optimal correction may depend on $\eta$; it need not be fixed at every precision.

\subsection{Periodic Beckmann duality and normalized endpoint energy}\label{orig:III:6}

Classical periodic Beckmann duality states, for a smooth real mean-zero source $g$,

\begin{equation}
\mathcal E_1(g)=\sup\left\{\int_{\mathbb T^2}g\phi\,dx:
\phi\in W^{1,\infty}(\mathbb T^2),\ |\nabla\phi|\le1\text{ a.e.}\right\}.
\tag{P.19}\label{eq:original-P-19}
\end{equation}

This interface applies to the unit flat torus, whose diameter is $\sqrt2/2$. To pass from probability measures to a general smooth mean-zero source, set $a=\int g_+\,dx=\int g_-\,dx$. If $a=0$ then $g=0$ and both flux energy and dual vanish. If $a>0$, set $\nu=g_+dx/a$ and $\eta=g_-dx/a$; then $gdx=a(\nu-\eta)$. Scaling the flux constraint and pairing in both directions extends the probability version to $g$. AGGT v1 Lemma~2.15 uses this formulation too. The $W^{1,\infty}$ gradient bound equals unit Lipschitz continuity for the torus quotient Euclidean metric: integration along shortest geodesics gives one implication, and local Euclidean Rademacher gives the other.

To justify the representative in this equivalence, one must not integrate an
almost-everywhere weak gradient along an arbitrary fixed line. Instead take
periodic nonnegative mollifications $\phi_\epsilon$ of $\phi$. Convolution
preserves $|\nabla\phi_\epsilon|\le1$. For these smooth functions, integration
along a shortest lifted segment proves the Lipschitz inequality for every
pair of points. After subtracting their common mean, they are uniformly
bounded and equicontinuous. Arzel\`a--Ascoli gives a uniformly convergent
subsequence; adding the mean back and using convergence of mollification
in $L^1$ identifies a representative of $\phi$. The Lipschitz inequality passes to
the limit for all pairs. On the convex square the same argument works first
on each smaller interior square; their compatible representatives extend
continuously to $Q$. This also justifies the gradient-to-Lipschitz step in
the boundary modification below.

Allowing vector Radon measure fluxes in the classical statement gives the same $L^1$ infimum as \eqref{eq:original-P-3}. Indeed, for any measure flux $H$ satisfying $\operatorname{div}H=g$, let $P_t^{\mathrm{heat}}$ be the periodic heat semigroup. The smooth flux $P_t^{\mathrm{heat}}H$ has divergence $P_t^{\mathrm{heat}}g$ and satisfies $\|P_t^{\mathrm{heat}}H\|_1\le|H|(\mathbb T^2)$. Add

\begin{equation*}
R_t=-\nabla(-\Delta)^{-1}(g-P_t^{\mathrm{heat}}g)
\end{equation*}

to restore divergence $g$. Since the first positive eigenvalue is $4\pi^2$,

\begin{equation*}
\|R_t\|_1\le\|R_t\|_2
\le\frac1{2\pi}\|g-P_t^{\mathrm{heat}}g\|_2\to0.
\end{equation*}

Thus the $L^1$-flux infimum is at most the measure-flux infimum; the reverse follows from inclusion. They coincide, without asserting existence of an $L^1$ minimizer.

By $\operatorname{div}J(g)=g$ the correct integration-by-parts sign is $\int g\phi=-\int J(g)\cdot\nabla\phi$. Symmetry under $\phi\mapsto-\phi$ shows that \eqref{eq:original-P-19} also equals the positive-pairing supremum. Hence

\begin{equation}
\frac{\mathcal E_1(\xi_\Lambda)}{\sqrt{m_\Lambda}}
=\sup_{\phi\in W^{1,\infty},\ |\nabla\phi|\le1}
\int Z_\Lambda\cdot\nabla\phi\,dx.
\tag{P.20}\label{eq:original-P-20}
\end{equation}

Set $e_{1,\Lambda}=\mathbb E_g\mathcal E_1(\xi_\Lambda)/\sqrt{m_\Lambda}$. It is finite because $J_\Lambda$ is an admissible flux. It is measurable because the mean-zero periodic Lipschitz tests form a compact set in the uniform norm, their support functional is continuous on $L^1$ sources, and the cutoff field depends continuously on finitely many Gaussian coefficients.

\subsection{Diagonalization for every real cutoff and the upper bound}\label{orig:III:7}

By \eqref{eq:original-P-18} choose $1<q_j<2$ so that $q_j$ strictly decreases to one and $|\kappa_{q_j}-\kappa_1|\le1/j$. For each fixed $q_j$, the limit \eqref{eq:original-P-14} provides a threshold above which every real $\Lambda$ satisfies

\begin{equation*}
\left|\mathbb E_g\int|U_{q_j,\Lambda}|^{q_j}\,dx-\kappa_{q_j}\right|
\le1/j.
\end{equation*}

Enlarge thresholds recursively to obtain strictly increasing $\Lambda_j\to\infty$. For $\Lambda_j\le\Lambda<\Lambda_{j+1}$ define

\begin{equation*}
q_\Lambda=q_j,\qquad p_\Lambda=\frac{q_\Lambda}{q_\Lambda-1},
\qquad U_\Lambda=\frac{H_{q_\Lambda,\Lambda}}{\sqrt{m_\Lambda}},
\qquad e_\Lambda=\mathbb E_g\int|U_\Lambda|^{q_\Lambda}\,dx.
\end{equation*}

On this entire interval $|e_\Lambda-\kappa_1|\le2/j$. Thus

\begin{equation}
q_\Lambda\to1,\qquad p_\Lambda\to\infty,
\qquad e_\Lambda\to\kappa_1.
\tag{P.21}\label{eq:original-P-21}
\end{equation}

No convergence uniform in $q$ is required in \eqref{eq:original-P-14}, and the two limits have not been interchanged.

The flux $U_\Lambda$ is feasible for $\xi_\Lambda/\sqrt{m_\Lambda}$. The homogeneity $\mathcal E_1(ag)=a\mathcal E_1(g)$ for $a>0$ follows by scaling admissible fluxes in both directions through $H\leftrightarrow aH$. Therefore

\begin{equation*}
\begin{aligned}
e_{1,\Lambda}
&\le\mathbb E_g\int|U_\Lambda|\,dx\\
&\le\left(\mathbb E_g\int|U_\Lambda|^{q_\Lambda}\,dx\right)^{1/q_\Lambda}
=e_\Lambda^{1/q_\Lambda}.
\end{aligned}
\end{equation*}

Here H\"older is applied to the mass-one product measure $\mathbb P_g\otimes dx$. Since $\kappa_1>0$ and $\log e_\Lambda\to\log\kappa_1$, we have $e_\Lambda^{1/q_\Lambda}\to\kappa_1$. Consequently

\begin{equation}
\limsup_{\Lambda\to\infty}e_{1,\Lambda}\le\kappa_1.
\tag{P.22}\label{eq:original-P-22}
\end{equation}

\subsection{The nonlinear dual field, exact pairing, and higher moments}\label{orig:III:8}

Define

\begin{equation*}
F_\Lambda=A_{q_\Lambda}(U_\Lambda)
=|U_\Lambda|^{q_\Lambda-2}U_\Lambda.
\end{equation*}

By \eqref{eq:original-P-6} and $A_q(cH)=c^{q-1}A_q(H)$ there is a periodic potential with

\begin{equation}
F_\Lambda=-m_\Lambda^{-(q_\Lambda-1)/2}\nabla u_{q_\Lambda,\Lambda}.
\tag{P.23}\label{eq:original-P-23}
\end{equation}

Thus $F_\Lambda$ is a spatially mean-zero periodic gradient in $L^{p_\Lambda}$. Its normalization exponent is $(q_\Lambda-1)/2$, not $1/2$.

Pointwise calculation gives

\begin{equation*}
U_\Lambda\cdot F_\Lambda=|U_\Lambda|^{q_\Lambda},
\qquad |F_\Lambda|^{p_\Lambda}=|U_\Lambda|^{(q_\Lambda-1)p_\Lambda}
=|U_\Lambda|^{q_\Lambda}.
\end{equation*}

Since $\operatorname{div}(U_\Lambda-Z_\Lambda)=0$, its pairing with every smooth periodic gradient vanishes. Approximate the potential of \eqref{eq:original-P-23} by smooth periodic potentials in $W^{1,p_\Lambda}$; H\"older passes the identity to the limit. Samplewise,

\begin{equation*}
\int Z_\Lambda\cdot F_\Lambda\,dx
=\int U_\Lambda\cdot F_\Lambda\,dx.
\end{equation*}

It follows that

\begin{equation}
\mathbb E_g\int Z_\Lambda\cdot F_\Lambda\,dx=e_\Lambda,
\qquad \mathbb E_g\int|F_\Lambda|^{p_\Lambda}\,dx=e_\Lambda.
\tag{P.24}\label{eq:original-P-24}
\end{equation}

Take an arbitrary deterministic sequence $\Lambda_n\to\infty$ and write $F_n=F_{\Lambda_n}$, $p_n=p_{\Lambda_n}$, and $e_n=e_{\Lambda_n}$. Fix $r>2$ and discard a finite prefix so that every $p_n>r$. On the product probability measure,

\begin{equation}
\mathbb E_g\int|F_n|^r\,dx
\le\left(\mathbb E_g\int|F_n|^{p_n}\,dx\right)^{r/p_n}
=e_n^{r/p_n}\to1.
\tag{P.25}\label{eq:original-P-25}
\end{equation}

These $r$th moments are uniformly bounded. Fix $\varepsilon>0$. On $|F_n|\le1+\varepsilon$ we have $(|F_n|-1)_+^r\le\varepsilon^r$; on the complement, since $r-p_n<0$,

\begin{equation*}
(|F_n|-1)_+^r\le|F_n|^r
=|F_n|^{p_n}|F_n|^{r-p_n}
\le(1+\varepsilon)^{r-p_n}|F_n|^{p_n}.
\end{equation*}

Integration yields

\begin{equation*}
\mathbb E_g\int\operatorname{dist}(F_n,\overline B_1)^r\,dx
\le\varepsilon^r+(1+\varepsilon)^{r-p_n}e_n.
\end{equation*}

First let $n\to\infty$ so the second term vanishes, then let $\varepsilon\downarrow0$ to obtain

\begin{equation}
\mathbb E_g\int\operatorname{dist}(F_n,\overline B_1)^r\,dx\to0.
\tag{P.26}\label{eq:original-P-26}
\end{equation}

Here $\overline B_1$ is the closed unit ball of $\mathbb R^2$. We still cannot use $F_n$ directly as a test in \eqref{eq:original-P-20}: \eqref{eq:original-P-26} is not an exact samplewise unit-gradient bound. Truncation must preserve the gradient structure.

First establish strong measurability of the potentials. For fixed $q\in(1,2)$ let $\alpha=q-1\in(0,1)$. The map $z\mapsto |z|^{\alpha-1}z$ satisfies

\begin{equation*}
|A_q(x)-A_q(y)|\le C_q|x-y|^\alpha.
\end{equation*}

There are two cases. If $|x-y|\ge\max(|x|,|y|)/2$, use $|A_q(x)|=|x|^\alpha$ directly. Otherwise the segment from the longer vector to the other stays away from zero, the derivative operator norm is at most $|z|^{\alpha-1}$ along it, and integrating the derivative gives the same bound. Since $\alpha p=q$,

\begin{equation*}
\|A_q(U)-A_q(V)\|_p\le C_q\|U-V\|_q^{q-1}.
\end{equation*}

Thus $A_q:L^q\to L^p$ is continuous. Strong measurability of the flux from Section~\ref{orig:III:2} gives strong measurability of $F_n$ in $L^{p_n}$, hence also in $L^r$.

For a mean-zero periodic gradient $F\in L^r$, its mean-zero potential $h$ satisfies the periodic Poincar\'e inequality

\begin{equation*}
\|h\|_{W^{1,r}}\le C_r\|F\|_r.
\end{equation*}

Recovery of the mean-zero potential from its gradient is therefore bounded and linear. The gradient space is closed too: if $F_j\to F$ in $L^r$, this bound makes the corresponding mean-zero potentials Cauchy in $W^{1,r}$; the limiting gradient is $F$. Choosing $\nabla h_n=F_n$ and $\int h_n=0$ makes $h_n$ a strongly measurable $W^{1,r}$-valued random variable.

\subsection{M\"uller's local truncation theorem and its deterministic periodic form}\label{orig:III:9}

\label{original-anchor:periodic-truncation}

\textbf{Cited theorem (M\"uller's local gradient truncation).} Let $O\subset\mathbb R^d$ be open, $K\subset\mathbb R^{m\times d}$ compact and convex, and $u_j\in W^{1,1}_{\mathrm{loc}}(O;\mathbb R^m)$. If $u_j\to u_0$ in $L^1_{\mathrm{loc}}$ and $\operatorname{dist}(Du_j,K)\to0$ in $L^1_{\mathrm{loc}}$, then there are increasing open sets $U_j\Subset O$ and functions $v_j\in W^{1,1}_{\mathrm{loc}}(O;\mathbb R^m)$ with

\begin{equation}
v_j=u_0\text{ on }O\setminus U_j,
\quad |\{u_j\ne v_j\}\cap U_j|\to0,
\quad \|\operatorname{dist}(Dv_j,K)\|_{L^\infty(O)}\to0.
\tag{P.27}\label{eq:original-P-27}
\end{equation}

If $|O|<\infty$, then $|O\setminus U_j|\to0$, so the modification set $|\{u_j\ne v_j\}|\to0$ is small on all of $O$. The original output space is $W^{1,1}_{\mathrm{loc}}$. Boundedness of $K$ and the global distance bound in \eqref{eq:original-P-27} imply that, for sufficiently large $j$, $\|Dv_j\|_\infty<\infty$, and hence $v_j$ admits a locally Lipschitz representative. The relevant truncation input is
\cite[Theorem~4]{Muller}; its local form is cited in
\cite[Lemma~2.18]{AGGT}. We use this foundational truncation theorem without reproducing its entire local construction.

We now prove the deterministic periodic consequence: for every $\eta>0$ there is $\delta(\eta)>0$ such that, whenever periodic $h\in W^{1,1}(\mathbb T^2)$ satisfies

\begin{equation*}
\int\operatorname{dist}(\nabla h,\overline B_1)\,dx\le\delta(\eta)
\end{equation*}

there is a periodic $\psi\in W^{1,\infty}(\mathbb T^2)$ satisfying

\begin{equation}
|\nabla\psi|\le1+\eta\text{ a.e.},
\qquad |\{h\ne\psi\}|\le\eta.
\tag{P.28}\label{eq:original-P-28}
\end{equation}

Otherwise, for some fixed $\eta_0>0$ one can choose periodic $u_j$ with $\int\operatorname{dist}(\nabla u_j,\overline B_1)\to0$ but with no modification satisfying \eqref{eq:original-P-28}. Subtracting the mean of $u_j$ leaves this contradiction property intact, since adding it back to any modification gives a modification of the original function.

Set $a_j=\Pi_{\overline B_1}(\nabla u_j)$ and $e_j=\nabla u_j-a_j$. The ball is closed and convex, so its nearest-point projection is defined, $|a_j|\le1$, and $\|e_j\|_1=\int\operatorname{dist}(\nabla u_j,\overline B_1)\to0$. Hence $\|\nabla u_j\|_1\le1+o(1)$; mean-zero normalization and periodic Poincar\'e give a uniform $W^{1,1}$ bound.

Rellich compactness yields a subsequence with $u_j\to u_0$ in $L^1(\mathbb T^2)$. Weak-star compactness in $L^\infty$ yields $a_j\rightharpoonup^*a$. Against smooth periodic tests, convergence in $L^1$ and the identity $e_j\to0$ imply $\nabla u_0=a$ distributionally. The closed convex unit ball is preserved under weak-star convergence, so $|a|\le1$ almost everywhere. Thus $u_0$ has a periodic unit Lipschitz representative.

Apply \eqref{eq:original-P-27} with $O=(0,1)^2$, $d=2$, $m=1$, and $K=\overline B_1$. Global convergence gives the required local convergence. We obtain $v_j=u_0$ in a boundary neighborhood, $\|\operatorname{dist}(\nabla v_j,\overline B_1)\|_\infty\to0$, and, by the finite-area clause, $|\{u_j\ne v_j\}|\to0$.

Opposite boundary values of $u_0$ match, so the opposite traces of $v_j$ match too. Periodic extension has no seam jump; matching normal derivatives is unnecessary. The global gradient bound and bounded $u_0$ near the boundary make $v_j$ a global $W^{1,\infty}$ function, whose periodic extension has the same gradient bound. Thus sufficiently large $j$ satisfies \eqref{eq:original-P-28}, contradicting the assumed failure and proving the consequence.

The function preserved near the boundary is the limit $u_0$, not the original $u_j$ or zero. This suffices for periodicity. Smallness of the modification set on the whole fundamental square is supplied exactly by the finite-area clause.

\subsection{\texorpdfstring{Random measurable truncation and $L^2$ convergence of gradients}{Random measurable truncation and L2 gradient convergence}}\label{orig:III:10}

\label{original-anchor:random-truncation}

We prove the random periodic form used in AGGT v1 Lemma~2.18. Fix $r>2$ and assume $h_n:\Omega_g\to W^{1,r}(\mathbb T^2)$ is strongly measurable with

\begin{equation}
\mathbb E_g\int\operatorname{dist}(\nabla h_n,\overline B_1)^r\,dx\to0,
\qquad C_r:=\sup_n\mathbb E_g\int|\nabla h_n|^r\,dx<\infty.
\tag{P.29}\label{eq:original-P-29}
\end{equation}

First make the selection in \eqref{eq:original-P-28} measurable. Sobolev embedding in dimension two says that for $r>2$ the separable Banach space $W^{1,r}(\mathbb T^2)$ embeds continuously in $C(\mathbb T^2)$. We therefore define equality sets using continuous representatives.

For fixed $\eta>0$ define a relation on the product of two $W^{1,r}$ spaces:

\begin{equation*}
\Gamma_\eta=\{(h,\psi):\|\nabla\psi\|_\infty\le1+\eta,
\ |\{h\ne\psi\}|\le\eta\}.
\end{equation*}

The first constraint is closed in $W^{1,r}$: if $\psi_j\to\psi$ in $W^{1,r}$, a subsequence of gradients converges almost everywhere and the bound passes to the limit. The second constraint is Borel, since

\begin{equation*}
|\{h\ne\psi\}|=\lim_{m\to\infty}\int\min\{1,m|h-\psi|\}\,dx,
\end{equation*}

The right side is a pointwise limit of continuous functions. The function $D(h)=\int\operatorname{dist}(\nabla h,\overline B_1)$ is continuous because distance is unit Lipschitz and $|D(h)-D(h')|\le\|\nabla h-\nabla h'\|_1$. Thus $E_\eta=\{h:D(h)\le\delta(\eta)\}$ is Borel, and by \eqref{eq:original-P-28} the relation has a nonempty section over every $h\in E_\eta$.

The Jankov--von Neumann selection theorem says that for Polish spaces $X,Y$ and a Borel relation $A\subset X\times Y$ there is a universally measurable selection $s$ on the first-coordinate projection satisfying $(x,s(x))\in A$. See A.~S.~Kechris~\cite{Kechris}, \emph{Classical Descriptive Set Theory} (1995), Theorem~18.1; AGGT v1 Lemma~2.18 uses this theorem. A Borel selection is not asserted. More precisely, Theorem~18.1 gives
a $\sigma(\Sigma^1_1)$-measurable uniformization of an analytic relation.
The Borel relation here is analytic. Analytic sets are universally measurable~\cite[Corollary~7.42.1]{BertsekasShreve},
and universally measurable sets form a sigma-algebra, so
$\sigma(\Sigma^1_1)$ is contained in the universal completion of the Borel
sigma-algebra. This is the bridge from the original theorem to the formulation
used here.

Apply it to $\Gamma_\eta$ restricted to $E_\eta$. Any strongly measurable $W^{1,r}$-valued random variable has a Borel probability law. A universally measurable selection is measurable for the completion of this law, so composition is measurable on the completed probability space. Separability gives strong measurability. If the original Fourier space is incomplete, complete it and choose an almost-everywhere equal version measurable for the original sigma-algebra. All subsequent conclusions are almost-everywhere statements and are unaffected.

Set

\begin{equation*}
D_n(\omega)=\int\operatorname{dist}(\nabla h_n(\omega),\overline B_1)\,dx,
\qquad b_n=\mathbb E_gD_n.
\end{equation*}

H\"older on the product probability space gives $b_n\le(\mathbb E_g\int\operatorname{dist}(\nabla h_n,\overline B_1)^r)^{1/r}\to0$. For $k\ge1$ set $\eta_k=1/k$ and $\delta_k=\delta(\eta_k)>0$. Choose strictly increasing integers $N_k$ so that for every $n\ge N_k$, $b_n\le\delta_k/k$.

For $N_k\le n<N_{k+1}$ put $\varepsilon_n=1/k$. On the good event $D_n\le\delta_k$ use the strongly measurable truncation selection; on the bad event set $\psi_n=0$. Finitely many initial indices may also be set to zero, with $\varepsilon_n=1$. This gives a nonincreasing deterministic sequence $\varepsilon_n\to0$ and strongly measurable periodic potentials satisfying $|\nabla\psi_n|\le1+\varepsilon_n$.

Markov's inequality yields

\begin{equation*}
\mathbb P_g(D_n>\delta_k)\le\frac{b_n}{\delta_k}\le\frac1k.
\end{equation*}

The modification set has area at most $1/k$ on the good event and at most the torus area one on the bad event. Writing $A_n=\{(\omega,x):\psi_n(\omega,x)\ne h_n(\omega,x)\}$ gives

\begin{equation}
\rho_n:=(\mathbb P_g\otimes dx)(A_n)\le\frac2k\to0.
\tag{P.30}\label{eq:original-P-30}
\end{equation}

By locality of Sobolev gradients, two Sobolev functions equal on a measurable set have equal weak gradients there almost everywhere. Apply this samplewise and then use Fubini: on product space, the gradient difference is supported in $A_n$. Hence

\begin{equation}
\begin{aligned}
\mathbb E_g\int|\nabla\psi_n-\nabla h_n|^2\,dx
&\le2\mathbb E_g\int_{A_n}|\nabla h_n|^2\,dx
+2(1+\varepsilon_n)^2\rho_n\\
&\le2C_r^{2/r}\rho_n^{1-2/r}
+2(1+\varepsilon_n)^2\rho_n\to0.
\end{aligned}
\tag{P.31}\label{eq:original-P-31}
\end{equation}

The second step uses H\"older exponents $r/2$ and $r/(r-2)$; $r>2$ ensures $1-2/r>0$. This shows how a small modification set and a uniform higher moment exclude $L^2$ concentration of gradients.

Let $\widehat\psi_n=\psi_n/(1+\varepsilon_n)$. It satisfies exactly $|\nabla\widehat\psi_n|\le1$, and

\begin{equation*}
|\nabla\widehat\psi_n-\nabla\psi_n|
=\frac{\varepsilon_n}{1+\varepsilon_n}|\nabla\psi_n|
\le\varepsilon_n.
\end{equation*}

Since product measure has mass one, the triangle inequality and \eqref{eq:original-P-31} give

\begin{equation}
\left(\mathbb E_g\int|\nabla\widehat\psi_n-\nabla h_n|^2\,dx\right)^{1/2}
\le\left(\mathbb E_g\int|\nabla\psi_n-\nabla h_n|^2\,dx\right)^{1/2}
+\varepsilon_n\to0.
\tag{P.32}\label{eq:original-P-32}
\end{equation}

This global normalization error is retained explicitly, rather than counted as an error only on the modification set.

\subsection{The final lower bound and the all-real-parameter endpoint theorem}\label{orig:III:11}

\label{original-anchor:periodic-endpoint-theorem}

By \eqref{eq:original-P-25}--\eqref{eq:original-P-26}, the potential $h_n$ of Section~\ref{orig:III:8} satisfies \eqref{eq:original-P-29}. Section~\ref{orig:III:10} gives a periodic unit Lipschitz potential $\widehat\psi_n$. Write $\widehat F_n=\nabla\widehat\psi_n$; then

\begin{equation*}
|\widehat F_n|\le1,
\qquad \mathbb E_g\int|\widehat F_n-F_n|^2\,dx\to0.
\end{equation*}

By \eqref{eq:original-P-5} and Cauchy--Schwarz on product space,

\begin{equation*}
\begin{aligned}
\mathbb E_g\left|\int Z_{\Lambda_n}\cdot(\widehat F_n-F_n)\,dx\right|
&\le\left(\mathbb E_g\int|Z_{\Lambda_n}|^2\,dx\right)^{1/2}
\left(\mathbb E_g\int|\widehat F_n-F_n|^2\,dx\right)^{1/2}\\
&\to0.
\end{aligned}
\end{equation*}

Samplewise, $\widehat\psi_n$ is admissible in \eqref{eq:original-P-20}, so

\begin{equation*}
\begin{aligned}
e_{1,\Lambda_n}
&\ge\mathbb E_g\int Z_{\Lambda_n}\cdot\widehat F_n\,dx\\
&=\mathbb E_g\int Z_{\Lambda_n}\cdot F_n\,dx+o(1)
=e_n+o(1)\to\kappa_1.
\end{aligned}
\end{equation*}

The random potential need not be independent of the noise: samplewise admissibility suffices to compare with the supremum. The $L^2$ error estimate likewise uses no independence.

The initial $\Lambda_n\to\infty$ was arbitrary. If the lower limit over all real parameters were below $\kappa_1$, some $\eta>0$ and sequence $\Lambda_n\to\infty$ would satisfy $e_{1,\Lambda_n}\le\kappa_1-\eta$, contradicting the preceding conclusion. Thus $\liminf_{\Lambda\to\infty}e_{1,\Lambda}\ge\kappa_1$. Combined with \eqref{eq:original-P-22} this yields \eqref{eq:original-P-1}; positivity and the upper bound follow from \eqref{eq:original-P-17}. This is the endpoint conclusion of AGGT v1 Theorem~2.19, with the finite-$q$-to-endpoint passage expanded here.

\subsection{Interface to the main proof, citation scope, and notation checks}\label{orig:III:12}

\label{original-anchor:main-interface}

The periodic Lipschitz energy used in the main proof is

\begin{equation*}
T(f)=\sup\left\{\int_{\mathbb T^2}uf\,dx:
\operatorname{Lip}_{\mathbb T^2}u\le1,\ \int u\,dx=0\right\}.
\end{equation*}

For mean-zero $f$, subtracting the test mean leaves the pairing unchanged; thus \eqref{eq:original-P-19} gives $T(f)=\mathcal E_1(f)$. If the main proof uses eigenvalue cutoff $L$, the actual periodic eigenvalues are $|k|^2$, so

\begin{equation*}
\Lambda=\sqrt L,
\qquad \xi_T^L=\xi_{\sqrt L},
\qquad m_T(L)=m_{\sqrt L},
\qquad F_T(L)=\mathbb E_gT(\xi_T^L).
\end{equation*}

The independent standard real Fourier coefficients $A_k,B_k$ correspond to our complex coefficients by $z_k=(A_k-iB_k)/\sqrt2$, because

\begin{equation*}
z_ke^{ik\cdot x}+\overline{z_k}e^{-ik\cdot x}
=\sqrt2A_k\cos(k\cdot x)+\sqrt2B_k\sin(k\cdot x).
\end{equation*}

The two representations therefore have the same law, with no additional $\sqrt2$. By \eqref{eq:original-P-1} the required output is

\begin{equation}
\boxed{\frac{F_T(L)}{\sqrt{m_T(L)}}\longrightarrow\kappa_1,
\qquad 0<\kappa_1\le\frac{\sqrt\pi}{2}.}
\tag{P.33}\label{eq:original-P-33}
\end{equation}

Only the periodic Gaussian mean limit is established here. The square
comparison, the empirical replacement, and the almost-sure upgrade are
separate steps. In particular, \eqref{eq:original-P-33} is not by itself
an almost-sure statement for one empirical sample.

For transparency, the cited statements used below are exactly the
following: \cite[Theorem~2.2]{AGGT} and its variational identification of
the constant; the bounded extensions and duality for the scale projection
in \cite[Lemma~2.4]{AGGT}; the local gradient truncation in
\cite[Theorem~4]{Muller}; classical periodic Beckmann duality; the
Gaussian martingale representation and It\^o isometry; Hahn--Banach
distance and Banach--Alaoglu; periodic Hodge decomposition, Poincar\'e,
Rellich compactness, Sobolev embedding, and gradient locality; and
Jankov--von Neumann selection. Their relevant statements and
applicability are specified where they are used, but their foundational
proofs are not presented as new research conclusions. Standard Sobolev
tools follow \cite{AdamsFournier}; gradient locality follows
\cite{EvansGariepy}; white noise and Wiener integration follow
\cite{Janson}. The martingale representation used is
\cite[eq.~(2.7)]{AGGT}, reproduced here as \eqref{eq:original-P-7}.

The cited truncation theorem \cite[Theorem~4]{Muller} outputs
$W^{1,1}_{\mathrm{loc}}$. The global gradient-distance bound above then
supplies the required locally Lipschitz representative, while preserving
the hypotheses and the exterior-boundary clause of that theorem.

Finally, \eqref{eq:original-P-33} serves as the periodic input for the
square Gaussian comparison. It gives
\[
F_T(L)/\sqrt{m_T(L)}\to\kappa_1,
\]
and the endpoint passage has been derived from the fixed
finite-exponent theory.

\section{Comparison of square and periodic Gaussian widths}\label{sec:comparison}

Changing from a torus to a square changes both the source law and the admissible test functions. We compare the sources using a common white noise, and compare the tests by localization followed by a boundary modification. The latter has three costs: localization, global rescaling, and modification in a boundary strip. All three are retained in the parameter choice. The comparison proves that the boundary contributes only a lower-order error.

\begin{proposition}[Boundary comparison and the Neumann limit]\label{prop:Neumann}
For every sufficiently large real $L$ there are $C_0,C_1<\infty$ independent of $L$ such that
\begin{equation}
F_T(L)-C_0\le F_N(L)\le F_T(L)+C_0+C_1m_N(L)^{1/3}.
\end{equation}
Consequently $F_N(L)/\sqrt{m_N(L)}\to\kappa_1$.
\end{proposition}
We prove every comparison step, including test localization, boundary modification, common-noise heat sources, and the errors of the two spectral cutoffs.

We start from the established square Neumann and periodic spectra, flux bounds, and spectral counting. The parameter $L$ is always an eigenvalue cutoff, and each estimate applies to every sufficiently large real $L$.

\subsection{Common white noise and internally independent coefficients}

Identify $L^2(Q,dx)$ and $L^2(\mathbb T^2,dx)$ as functions on the fundamental square; its null boundary leaves the Hilbert space unchanged. On this real Hilbert space take an isonormal Gaussian process $W$, a linear map satisfying

\begin{equation*}
\mathbb EW(a)=0,\qquad
\mathbb E[W(a)W(b)]=\langle a,b\rangle_{L^2(Q)}.
\end{equation*}

Explicitly, choose any real orthonormal basis $(e_j)$ and independent standard real Gaussian variables $(G_j)$, and define

\begin{equation*}
W(a)=\sum_jG_j\langle a,e_j\rangle.
\end{equation*}

The series converges in the $L^2$ norm of random variables because a tail's expected square is $\sum_{j>J}|\langle a,e_j\rangle|^2\to0$. Parseval yields the stated covariance.

On the Neumann basis set $g_k=W(\varphi_k)$. These variables are jointly Gaussian with covariance $\delta_{kl}$, hence independent, each with law $N(0,1)$. The same argument in the periodic real basis gives independent standard real coefficients. The two families may be correlated with each other; only each model's prescribed internal law matters. This constructs all coefficients at once and preserves joint laws across every cutoff within each model.

Let

\begin{equation*}
\xi_N^L=\sum_{0<\lambda_k\le L}g_k\varphi_k,
\qquad
m_N(L)=\sum_{0<\lambda_k\le L}\lambda_k^{-1}.
\end{equation*}

Similarly, the periodic field $\xi_T^L$ retains all real eigenbasis modes with positive eigenvalue at most $L$; denote its inverse Laplacian trace by $m_T(L)$.

\subsection{Three test classes, finiteness, and measurability}

Define

\begin{equation*}
\begin{aligned}
\mathcal K&=\{u\in C(Q):\operatorname{Lip}_Q(u)\le1,\ \int_Q u=0\},\\
\mathcal K_T&=\{u\in C(\mathbb T^2):\operatorname{Lip}_{\mathbb T^2}(u)\le1,\ \int_Q u=0\},\\
\mathcal K_0&=\{w\in C(Q):\operatorname{Lip}_Q(w)\le1,\ w|_{\partial Q}=0\}.
\end{aligned}
\end{equation*}

There is no mean-zero constraint in $\mathcal K_0$. For a mean-zero source $f$ write

\begin{equation*}
\mathcal S(f)=\sup_{u\in\mathcal K}\langle f,u\rangle,
\quad
\mathcal T(f)=\sup_{u\in\mathcal K_T}\langle f,u\rangle,
\quad
\mathcal Z(f)=\sup_{w\in\mathcal K_0}\langle f,w\rangle.
\end{equation*}

For $u\in\mathcal K$, mean-zero normalization and area one of $Q$ give

\begin{equation*}
|u(x)|=\left|\int_Q(u(x)-u(y))\,dy\right|
\le\int_Q|x-y|\,dy\le\sqrt2.
\end{equation*}

The periodic class has the same bound, since quotient distance is at most square Euclidean distance. For $w\in\mathcal K_0$ choose a nearest boundary point $y$; then

\begin{equation*}
|w(x)|=|w(x)-w(y)|\le\operatorname{dist}(x,\partial Q)\le\frac12.
\end{equation*}

All three classes are uniformly bounded, equicontinuous, and closed in the uniform norm, hence compact by Arzel\`a--Ascoli. At finite cutoff, the pairing is a continuous function of finite Gaussian coefficients, maximized over a compact class. Its supremum is finite and continuous.

For any test class $A$ and $L^1$ source $f,h$, each test satisfies

\begin{equation*}
\langle f,u\rangle\le\langle h,u\rangle
+\|u\|_\infty\|f-h\|_1.
\end{equation*}

Take the supremum and interchange $f,h$ to obtain

\begin{equation*}
\left|\sup_{u\in A}\langle f,u\rangle-
\sup_{u\in A}\langle h,u\rangle\right|
\le\left(\sup_{u\in A}\|u\|_\infty\right)\|f-h\|_1.
\end{equation*}

All three support functions are continuous in the $L^1$ source norm. The heat sources' $L^1$ measurability follows from the mean-square spectral expansions below, giving measurable support functions. Alternatively, a countable dense subset of the compact test class has the same supremum under continuous pairing; no random maximizing test need be selected.

\subsection{Inclusions of zero-boundary, periodic, and square tests}

For $w\in\mathcal K_0$ extend it to $\mathbb R^2$ by integer translates. Values on adjacent seams are zero, so the extension is continuous. Grid lines split any finite straight segment into finitely many pieces. Within a square, variation is at most length; on grid lines the function is zero. Summing gives

\begin{equation*}
|\widetilde w(x)-\widetilde w(y)|\le|x-y|.
\end{equation*}

The extension is $\mathbb Z^2$-periodic, so for each $z\in\mathbb Z^2$,

\begin{equation*}
|\widetilde w(x)-\widetilde w(y)|
=|\widetilde w(x)-\widetilde w(y+z)|\le|x-y-z|.
\end{equation*}

Minimizing over $z$ gives the unit Lipschitz bound for quotient distance. Subtracting the spatial mean places the test in $\mathcal K_T$, while $\int f=0$ preserves its pairing. Thus $\mathcal Z(f)\le\mathcal T(f)$.

Conversely, $d_T(x,y)\le|x-y|$, so a periodic test restricted to $Q$ belongs to $\mathcal K$. Hence

\begin{equation*}
\mathcal Z(f)\le\mathcal T(f)\le\mathcal S(f).
\end{equation*}

This is only inclusion of test classes; it does not pair a periodic flux with arbitrary square tests by boundary-free integration by parts.

\subsection{Finite-dimensional Gaussian exponential moments and finite suprema}

We first prove the concentration estimate used here. Let $G\sim N(0,I_d)$ and let $X:\mathbb R^d\to\mathbb R$ be $R$-Lipschitz. We claim

\begin{equation*}
\mathbb E e^{a(X(G)-\mathbb EX(G))}\le e^{a^2R^2/2}
\qquad(a\in\mathbb R).
\end{equation*}

Take standard Brownian motion $B_t$ with covariance $tI_d$. The Euclidean heat flow here has generator $\Delta/2$ and is denoted $P_t^{\mathrm B}$; distinguish it from the spatial-source generator $\Delta$. Set

\begin{equation*}
v(t,x)=P_{1-t}^{\mathrm B}X(x),\qquad 0\le t<1.
\end{equation*}

Heat convolution preserves the $R$ Lipschitz constant: using the same Gaussian displacement $Z$,

\begin{equation*}
|\mathbb EX(x+Z)-\mathbb EX(y+Z)|\le R|x-y|.
\end{equation*}

At positive time heat convolution is smooth, so $|\nabla_xv|\le R$ and $\partial_tv+\frac12\Delta v=0$. Brownian independent increments give

\begin{equation*}
M_t=v(t,B_t)=\mathbb E[X(B_1)\mid\mathcal F_t].
\end{equation*}

The drift in It\^o's formula cancels, yielding

\begin{equation*}
M_t-M_0=\int_0^t\nabla v(s,B_s)\cdot dB_s,
\qquad
\langle M\rangle_t=\int_0^t|\nabla v(s,B_s)|^2ds\le R^2t.
\end{equation*}

The bounded integrand satisfies the exponential-martingale integrability condition. Hence

\begin{equation*}
\mathbb E\exp\left(a(M_t-M_0)-\frac{a^2}2\langle M\rangle_t\right)=1.
\end{equation*}

Thus $\mathbb Ee^{a(M_t-M_0)}\le e^{a^2R^2t/2}$. Lipschitz growth makes $X(B_1)$ integrable; as $t\uparrow1$ the conditional-expectation martingale converges to $X(B_1)$ and $M_0=\mathbb EX(G)$. Fatou for the nonnegative exponentials yields the endpoint estimate. The inputs are It\^o's formula and the exponential-martingale theorem for bounded integrands; $X$ need not itself be differentiable.

Now suppose $X_1,\ldots,X_N$ are $R$-Lipschitz functions of the same Gaussian vector. For $a>0$,

\begin{equation*}
\begin{aligned}
\exp\left(a\max_i(X_i-\mathbb EX_i)\right)
&\le\sum_i\exp\left(a(X_i-\mathbb EX_i)\right),\\
\mathbb E\max_i(X_i-\mathbb EX_i)
&\le\frac1a\log\mathbb E e^{a\max_i(X_i-\mathbb EX_i)}\\
&\le\frac{\log N}{a}+\frac{aR^2}2.
\end{aligned}
\end{equation*}

The first line is the deterministic finite-sum bound, the second is Jensen, and the third uses the proved exponential moment. Set $a=\sqrt{2\log N}/R$ to obtain

\begin{equation*}
\mathbb E\max_iX_i\le\max_i\mathbb EX_i+R\sqrt{2\log N}.
\end{equation*}

The cases $N=1$ or $R=0$ are immediate. No independence between distinct $X_i(G)$ is used.

\subsection{Construction of a uniform grid cover}

Take $0<\varepsilon<1$ and set $M=\lceil16/\varepsilon\rceil$ and $h=1/M$. At the $(M+1)^2$ grid vertices round function values to the nearest integer multiple of $h$, using a fixed tie rule.

At the root vertex values lie in $[-\sqrt2,\sqrt2]$, so there are at most $C/h$ rounded values. Choose a nearest-neighbor spanning tree of the grid. Adjacent original values differ by at most $h$; each rounding contributes at most $h/2$, so rounded differences are at most $2h$. Given the parent's value, the child has the five relative possibilities $-2h,-h,0,h,2h$. The number of nonempty rounding classes is at most

\begin{equation*}
N_\varepsilon\le\frac Ch\,5^{(M+1)^2-1}.
\end{equation*}

Since $M\le17/\varepsilon$,

\begin{equation*}
\log N_\varepsilon\le\log(C/h)+((M+1)^2-1)\log5
\le C\varepsilon^{-2}.
\end{equation*}

Functions $u,v$ in the same class differ at vertices by at most $h$. Every $x\in Q$ is within $\sqrt2h$ of a vertex $x_0$. Thus

\begin{equation*}
|u(x)-v(x)|\le|u(x)-u(x_0)|+|u(x_0)-v(x_0)|+|v(x_0)-v(x)|
\le(1+2\sqrt2)h<\varepsilon.
\end{equation*}

In each nonempty class choose a representative $u_i$ actually belonging to $\mathcal K$. The sets

\begin{equation*}
A_i=\mathcal K\cap\overline B_\infty(u_i,\varepsilon)
\end{equation*}

are nonempty, compact, and cover $\mathcal K$. These are original tests rather than arbitrary grid interpolants, so interpolation cannot violate the Lipschitz constraint.

\subsection{Small-amplitude localization of the expected Gaussian width}

Write

\begin{equation*}
H_L(A)=\sup_{u\in A}\langle\xi_N^L,u\rangle,
\qquad h_L(A)=\mathbb EH_L(A).
\end{equation*}

For nonempty compact $A$, linear pairing gives samplewise

\begin{equation*}
\begin{aligned}
H_L((A-A)/2)
&=\frac12\sup_{u,v\in A}
\left(\langle\xi_N^L,u\rangle-\langle\xi_N^L,v\rangle\right)\\
&=\frac12\left(H_L(A)+H_L(-A)\right).
\end{aligned}
\end{equation*}

The variables $\xi_N^L$ and $-\xi_N^L$ have the same law, so $h_L(-A)=h_L(A)$ and consequently $h_L((A-A)/2)=h_L(A)$.

For $u,v\in A_i$, the quantity $a=(u-v)/2$ satisfies

\begin{equation*}
\int_Qa=0,\qquad\operatorname{Lip}(a)\le1,
\qquad\|a\|_\infty\le\varepsilon,
\qquad\operatorname{osc}(a)\le2\varepsilon.
\end{equation*}

Define

\begin{equation*}
\mathcal K_\varepsilon=\{u\in\mathcal K:\operatorname{osc}(u)\le2\varepsilon\},
\qquad F_\varepsilon(L)=\mathbb EH_L(\mathcal K_\varepsilon).
\end{equation*}

Then $h_L(A_i)=h_L((A_i-A_i)/2)\le F_\varepsilon(L)$.

View the finite cutoff coefficients as a Euclidean vector $g$. For $g,g'$,

\begin{equation*}
|H_L(A_i;g)-H_L(A_i;g')|
\le|g-g'|\sup_{u\in A_i}
\left(\sum_{\lambda_k\le L}|\langle u,\varphi_k\rangle|^2\right)^{1/2}
\le\sqrt2|g-g'|.
\end{equation*}

The last step uses Bessel and $\|u\|_2\le\|u\|_\infty\le\sqrt2$. Apply the preceding finite-supremum estimate to $H_L(A_i)$ and use $H_L(\mathcal K)=\max_iH_L(A_i)$ to obtain

\begin{equation*}
\begin{aligned}
F_N(L)&:=\mathbb E\mathcal S(\xi_N^L)
\le F_\varepsilon(L)+\sqrt2\sqrt{2\log N_\varepsilon}
\le F_\varepsilon(L)+C\varepsilon^{-1},\\
0&\le F_N(L)-F_\varepsilon(L)\le C\varepsilon^{-1}.
\end{aligned}
\end{equation*}

This estimates expected width; it does not assert that a samplewise optimal test can be replaced without loss by a small-amplitude test.

\subsection{All errors in the zero-boundary modification}

Write $F_{0,N}(L)=\mathbb E\mathcal Z(\xi_N^L)$, where $\mathcal Z$ is the zero-boundary support function already defined.

Let $m=m_N(L)$ and $J_L=\nabla(-\Delta_N)^{-1}\xi_N^L$. The flux bound and gradient orthogonality give

\begin{equation*}
\mathbb E\|J_L\|_2^2=m,
\qquad F_N(L)\le\mathbb E\|J_L\|_2\le\sqrt m.
\end{equation*}

The explicit cosine formula for $\varphi_k$ gives pointwise $|\nabla\varphi_k|^2\le4\lambda_k$, hence

\begin{equation*}
\mathbb E|J_L(x)|^2
=\sum_{\lambda_k\le L}\frac{|\nabla\varphi_k(x)|^2}{\lambda_k^2}
\le4m,
\qquad \mathbb E|J_L(x)|\le2\sqrt m.
\end{equation*}

Take $0<\delta<1/2$ and define

\begin{equation*}
B_\delta=\{x:\operatorname{dist}(x,\partial Q)<\delta\},
\qquad\chi_\delta(x)=\min\{1,\operatorname{dist}(x,\partial Q)/\delta\}.
\end{equation*}

We have $|B_\delta|=1-(1-2\delta)^2\le4\delta$, $\chi_\delta=1$ on the complement, and $|\nabla\chi_\delta|\le1/\delta$ almost everywhere.

For $u\in\mathcal K_\varepsilon$ set

\begin{equation*}
v=u-\frac{\max_Q u+\min_Q u}{2},
\qquad s=\frac1{1+\varepsilon/\delta},
\qquad w=s\chi_\delta v.
\end{equation*}

Then $\|v\|_\infty\le\varepsilon$ and

\begin{equation*}
|\nabla w|\le s\left(|\nabla v|+|v||\nabla\chi_\delta|\right)
\le s(1+\varepsilon/\delta)=1.
\end{equation*}

On the convex square this gradient bound gives a unit Lipschitz continuous representative. Since $w$ vanishes on the boundary, $w\in\mathcal K_0$.

Expand the gradient difference term by term:

\begin{equation*}
\nabla(v-w)=(1-s)\nabla u
+s(1-\chi_\delta)\nabla u-sv\nabla\chi_\delta.
\end{equation*}

The source is mean zero, so subtracting constants leaves pairings unchanged. Neumann flux integration by parts bounds the first term by $(1-s)\mathcal S(\xi_N^L)$. The remaining two terms contribute only on $B_\delta$, with total absolute value at most

\begin{equation*}
s(1+\varepsilon/\delta)\int_{B_\delta}|J_L|=
\int_{B_\delta}|J_L|.
\end{equation*}

Thus samplewise, simultaneously for all $u$,

\begin{equation*}
|\langle\xi_N^L,u-w\rangle|
\le(1-s)\mathcal S(\xi_N^L)+\int_{B_\delta}|J_L|.
\end{equation*}

The first term is a global rescaling error and cannot be restricted to the boundary strip. Since $\langle\xi_N^L,w\rangle\le\mathcal Z(\xi_N^L)$, taking the supremum over $u$ and expectation gives

\begin{equation*}
\begin{aligned}
F_\varepsilon(L)-F_{0,N}(L)
&\le(1-s)F_N(L)+\int_{B_\delta}\mathbb E|J_L(x)|dx\\
&\le\frac\varepsilon\delta\sqrt m+2|B_\delta|\sqrt m\\
&\le\left(\frac\varepsilon\delta+8\delta\right)\sqrt m,
\end{aligned}
\end{equation*}

We used Tonelli and $1-s=\varepsilon/(\delta+\varepsilon)\le\varepsilon/\delta$ here.

Combine this with small-amplitude localization. For $m>64$ choose

\begin{equation*}
\varepsilon=m^{-1/3},\qquad\delta=m^{-1/6}.
\end{equation*}

The three losses are respectively

\begin{equation*}
\varepsilon^{-1}=m^{1/3},\quad
\frac\varepsilon\delta\sqrt m=m^{1/3},\quad
\delta\sqrt m=m^{1/3}.
\end{equation*}

Subtracting the mean of $w\in\mathcal K_0$ places it in $\mathcal K$. Finally,

\begin{equation*}
0\le F_N(L)-F_{0,N}(L)\le C m_N(L)^{1/3}.
\end{equation*}

\subsection{Definition of heat sources and their zero mean}

For $D=N,T$ take the spatial heat kernel $p_D$ with generator $\Delta$ and define

\begin{equation*}
\eta_D(t,x)=W(p_D(t,x,\cdot)-1).
\end{equation*}

If $\psi_j$ is the model's positive-eigenvalue real basis, then

\begin{equation*}
\eta_D(t)=\sum_{\lambda_j>0}e^{-t\lambda_j}W(\psi_j)\psi_j.
\end{equation*}

Spectral counting $N_D(a)\le Ca$ and exponential decay make every finite-order Sobolev second moment summable for fixed $t>0$. Finite spectral projections converge in the corresponding mean-square Sobolev spaces; Sobolev embedding gives a smooth measurable version. The constant mode has been removed. Alternatively, mass one and self-adjointness give directly

\begin{equation*}
\int_Q\eta_D(t,x)dx
=W\left(\int_Qp_D(t,x,\cdot)dx-1\right)=0.
\end{equation*}

The Bochner integral passes through the continuous linear isometry $W:L^2(Q)\to L^2(\mathbb P)$, justifying this interchange with the random integral.

\subsection{One-dimensional heat-kernel images and uniform tail sums}

Set $G_t(z)=(4\pi t)^{-1/2}e^{-z^2/(4t)}$. The one-dimensional Neumann and periodic kernels are

\begin{equation*}
a_t(x,y)=\sum_{j\in\mathbb Z}[G_t(x-y+2j)+G_t(x+y+2j)],
\qquad
b_t(x,y)=\sum_{j\in\mathbb Z}G_t(x-y+j).
\end{equation*}

The first comes from Gaussian displacement followed by even, two-periodic folding; the second from reducing displacement modulo one. Both kernels have mass one, with cosine and Fourier heat spectra respectively. At fixed $t>0$ the exponential image tails are summable, justifying these series and the $L^2$ estimates below.

If an image term $z=z(y)$ satisfies $|z'|=1$ and $|z(y)|\ge D$, then

\begin{equation*}
\begin{aligned}
\int_0^1G_t(z(y))^2dy
&=(4\pi t)^{-1}\int_0^1e^{-z(y)^2/(2t)}dy\\
&\le(4\pi t)^{-1}e^{-D^2/(4t)}
\int_{\mathbb R}e^{-z^2/(4t)}dz\\
&\le Ct^{-1/2}e^{-D^2/(4t)}.
\end{aligned}
\end{equation*}

Its $L^2_y$ norm is therefore at most $Ct^{-1/4}e^{-D^2/(8t)}$.

All direct even translations in $a_t-b_t$ cancel exactly, leaving

\begin{equation*}
a_t-b_t=\sum_{j\in\mathbb Z}G_t(x+y+2j)
-\sum_{j\in2\mathbb Z+1}G_t(x-y+j).
\end{equation*}

For $x,y\in[0,1]$, the four remaining image branches have the following distances, where $r\in\mathbb N_0$:

\begin{longtable}{p{0.3033\linewidth}p{0.3033\linewidth}p{0.3033\linewidth}}
\toprule
Image term & Range as $y$ varies & Nearest distance to zero\\
$x+y+2r$ & $[x+2r,x+1+2r]$ & $x+2r$\\
$x+y-2r-2$ & $[x-2r-2,x-2r-1]$ & $1-x+2r$\\
$x-y+2r+1$ & $[x+2r,x+2r+1]$ & $x+2r$\\
$x-y-2r-1$ & $[x-2r-2,x-2r-1]$ & $1-x+2r$\\\bottomrule
\end{longtable}

Since $(a+2r)^2\ge a^2+4r^2$, for $0<t\le1$,

\begin{equation*}
\sum_{r\ge0}e^{-(a+2r)^2/(8t)}
\le e^{-a^2/(8t)}\sum_{r\ge0}e^{-r^2/(2t)}
\le C e^{-a^2/(8t)}.
\end{equation*}

Apply the $L^2$ triangle inequality to finite sums, then pass to the infinite series using summability of norm tails. This gives

\begin{equation*}
\|a_t(x,\cdot)-b_t(x,\cdot)\|_2
\le Ct^{-1/4}[e^{-x^2/(8t)}+e^{-(1-x)^2/(8t)}]
\le Ct^{-1/4}e^{-d(x)^2/(8t)},
\end{equation*}

Here $d(x)=\min(x,1-x)$, and the final $C$ absorbs both terms.

For each kernel, the nearest direct image has norm at most $Ct^{-1/4}$; other direct even translations are at distance at least $2|j|-1$ from zero. Reflections and odd translations obey the table bounds. Summing separately yields

\begin{equation*}
\|a_t(x,\cdot)\|_2+\|b_t(x,\cdot)\|_2\le Ct^{-1/4}.
\end{equation*}

The endpoints $x=0,1$ are included. The proof does not require remaining images to be at strictly positive distance from zero.

\subsection{Two-dimensional corners and the expected source error}

The two-dimensional kernels are coordinate products. Writing $a_i=a_t(x_i,y_i)$ and $b_i=b_t(x_i,y_i)$ gives the exact identity

\begin{equation*}
a_1a_2-b_1b_2=(a_1-b_1)a_2+b_1(a_2-b_2).
\end{equation*}

The $L^2$ norm of a tensor product is the product of the factor norms. Thus

\begin{equation*}
\|p_N(t,x,\cdot)-p_T(t,x,\cdot)\|_2
\le Ct^{-1/2}\left[e^{-d(x_1)^2/(8t)}+e^{-d(x_2)^2/(8t)}\right].
\end{equation*}

Expanding $a_i=b_i+\delta_i$ shows that the corner term $\delta_1\delta_2$ is already included in $\delta_1a_2$; no additional term is omitted.

Common white noise cancels the constant mode exactly, and isometry gives

\begin{equation*}
\mathbb E|\eta_N(t,x)-\eta_T(t,x)|^2
=\|p_N(t,x,\cdot)-p_T(t,x,\cdot)\|_2^2.
\end{equation*}

Its mean absolute value is therefore bounded by this row norm. Tonelli and

\begin{equation*}
\int_0^1e^{-d(x)^2/(8t)}dx
=2\int_0^{1/2}e^{-r^2/(8t)}dr\le C\sqrt t,
\end{equation*}

give

\begin{equation*}
\mathbb E\|\eta_N(t)-\eta_T(t)\|_1
\le Ct^{-1/2}(2C\sqrt t)\le C,
\qquad0<t\le1.
\end{equation*}

This proves $\sup_t\mathbb E\|\eta_N-\eta_T\|_1\le C$ without needing a bound for $\mathbb E\sup_t$.

\subsection{All spectral errors between sharp and heat cutoffs}

For either model define $c_j(L)=\boldsymbol1_{\{\lambda_j\le L\}}-e^{-\lambda_j/L}$. On low frequencies use $1-e^{-a}\le a$:

\begin{equation*}
\sum_{\lambda_j\le L}\frac{c_j(L)^2}{\lambda_j}
\le L^{-2}\sum_{\lambda_j\le L}\lambda_j
\le\frac{N_D(L)}L\le C.
\end{equation*}

On high frequencies use half-open bands $(2^rL,2^{r+1}L]$, with $r\ge0$. Their contribution is at most

\begin{equation*}
\frac{e^{-2^{r+1}}}{2^rL}N_D(2^{r+1}L)
\le C e^{-2^{r+1}}.
\end{equation*}

Summing over $r$ gives

\begin{equation*}
\sum_{\lambda_j>0}\frac{c_j(L)^2}{\lambda_j}\le C.
\end{equation*}

A spectral point equal to $L$ belongs to the low-frequency part; every other positive eigenvalue lies in exactly one high-frequency band. Thus this estimate covers all sufficiently large real $L$ directly.

Let $h_D^L=\xi_D^L-\eta_D(1/L)$. Gradient orthogonality and standard Gaussian variance give

\begin{equation*}
\mathbb E\|J_D(h_D^L)\|_2^2
=\sum_{\lambda_j>0}\frac{c_j(L)^2}{\lambda_j}\le C.
\end{equation*}

For fixed $L$, the source and flux heat tails converge in their mean-square $L^2$ spaces. Uniform $\|u\|_2$ bounds control source pairings for every test, and $\|\nabla u\|_2\le1$ controls flux pairings. Finite-spectral integration by parts and support-function estimates therefore pass to the infinite heat expansion. Neither random optimizers nor a subsequence restriction of the final parameter limit is needed.

On the Neumann geometry the flux bound applies to $\mathcal S$ and $\mathcal Z$; on the periodic geometry only to $\mathcal T$. Hence

\begin{equation*}
\begin{aligned}
|F_N(L)-\mathbb E\mathcal S(\eta_N(1/L))|&\le C,\\
|F_{0,N}(L)-\mathbb E\mathcal Z(\eta_N(1/L))|&\le C,\\
|F_T(L)-\mathbb E\mathcal T(\eta_T(1/L))|&\le C.
\end{aligned}
\end{equation*}

\subsection{Combining the two comparison directions line by line}

Source stability in $L^1$ for the support functions and the proved heat-source error give

\begin{equation*}
|\mathbb E\mathcal S(\eta_N(t))-\mathbb E\mathcal S(\eta_T(t))|\le C,
\qquad
|\mathbb E\mathcal Z(\eta_N(t))-\mathbb E\mathcal Z(\eta_T(t))|\le C.
\end{equation*}

Set $t=1/L$. The first direction uses, in order, periodic cutoff stability, test inclusion, the source error, and Neumann cutoff stability:

\begin{equation*}
\begin{aligned}
F_T(L)&\le\mathbb E\mathcal T(\eta_T(t))+C\\
&\le\mathbb E\mathcal S(\eta_T(t))+C\\
&\le\mathbb E\mathcal S(\eta_N(t))+C\\
&\le F_N(L)+C.
\end{aligned}
\end{equation*}

The second direction first uses the zero-boundary loss and then source comparison:

\begin{equation*}
\begin{aligned}
F_N(L)&\le F_{0,N}(L)+Cm_N(L)^{1/3}\\
&\le\mathbb E\mathcal Z(\eta_N(t))+Cm_N(L)^{1/3}+C\\
&\le\mathbb E\mathcal Z(\eta_T(t))+Cm_N(L)^{1/3}+C\\
&\le\mathbb E\mathcal T(\eta_T(t))+Cm_N(L)^{1/3}+C\\
&\le F_T(L)+Cm_N(L)^{1/3}+C.
\end{aligned}
\end{equation*}

Since $m_N(L)\to\infty$, the final fixed constant can be absorbed. This proves

\begin{equation*}
-C\le F_N(L)-F_T(L)\le Cm_N(L)^{1/3}.
\end{equation*}

No periodic limit theorem has yet been used in this comparison.

\subsection{Real and complex coefficients and the precise periodic-limit interface}

For a nonzero physical frequency pair $\{k,-k\}$ the real orthonormal basis is $\sqrt2\cos(k\cdot x)$ and $\sqrt2\sin(k\cdot x)$. With standard real coefficients $A_k,B_k$ set

\begin{equation*}
z_k=\frac{A_k-iB_k}{\sqrt2},\qquad z_{-k}=\overline{z_k}.
\end{equation*}

Termwise expansion gives

\begin{equation*}
z_ke^{ik\cdot x}+\overline{z_k}e^{-ik\cdot x}
=\sqrt2A_k\cos(k\cdot x)+\sqrt2B_k\sin(k\cdot x).
\end{equation*}

The complex coefficient's real and imaginary parts each have variance $1/2$, as in the periodic endpoint theorem. Here $k\in2\pi\mathbb Z^2$ is a physical frequency with eigenvalue $|k|^2$, so the external wave-number cutoff is

\begin{equation*}
\Lambda=\sqrt L,\qquad m_\Lambda=m_T(L),\qquad
\xi_\Lambda\overset{\mathrm{law}}=\xi_T^L.
\end{equation*}

Section~\ref{sec:periodic} proved $F_T(L)/\sqrt{m_T(L)}\to\kappa_1$. Both spectral traces are $\log L/(4\pi)+O(1)$, hence $m_T(L)/m_N(L)\to1$. The full denominator conversion is

\begin{equation*}
\frac{F_N(L)}{\sqrt{m_N(L)}}
=\frac{F_T(L)}{\sqrt{m_T(L)}}
\sqrt{\frac{m_T(L)}{m_N(L)}}
+\frac{F_N(L)-F_T(L)}{\sqrt{m_N(L)}}.
\end{equation*}

The last term has absolute value $O(m_N(L)^{-1/6})\to0$. Therefore, for every real $L\to\infty$,

\begin{equation*}
\frac{F_N(L)}{\sqrt{m_N(L)}}\longrightarrow\kappa_1\in(0,\infty).
\end{equation*}

This gives a comparison rate between finite-cutoff widths, not a convergence rate for the periodic endpoint itself.

\section{Empirical mean reduction and the almost-sure limit for all sample sizes}\label{sec:empirical}

We first identify the mean and then prove the pathwise limit. Gaussian replacement takes place in flux coordinates, whose dimension grows with the spectral cutoff. The coarse cutoff leaves room for Gaussian replacement, while the two-scale heat argument keeps the transport cost of reaching that cutoff below the leading scale. The resulting mean asymptotic does not itself give almost-sure convergence: the concentration bound is summable on geometric subsequences, but not over all integers for small deviations. A martingale decomposition and the monotone compensator fill this gap.

\begin{proposition}[Empirical mean reduction and the probability upgrade]\label{prop:empirical}
For each fixed real $B>2$ take $L_n=n/(\log n)^B$. The actual empirical distance satisfies the introductory error bound~\eqref{eq:intro-reduction}, and its normalized mean converges to $c_*$. Along the same infinite iid sequence, the normalized distance converges almost surely to this constant over all integers.
\end{proposition}
We first prove the mean reduction, then independently prove the conditional interface from a positive finite mean limit to an all-integer almost-sure limit, and finally instantiate it.

There are two tasks, both strictly on $Q=[0,1]^2$ with Euclidean distance and $\mu=dx$. First reduce the normalized mean of empirical $W_1$ to the truncated Neumann Gaussian energy. Second, assuming a positive finite normalized mean limit, prove the same almost-sure limit for the actual distance at every integer sample size of the original iid sequence. Section~\ref{sec:comparison} established the required square--periodic Gaussian comparison; here we derive the empirical reduction before inserting that Gaussian conclusion.

\subsection{Objects, conventions, and foundational results}\label{orig:V:1}

We use compact-space KR and $d_{BL}=W_1$ with the same Lipschitz dual, empirical source, and normalization. The original probability space, metric, and sample sequence remain fixed.

Let $Q=[0,1]^2$ have Euclidean distance and let $\mu=dx$ be area measure of mass one. The sequence $X_1,X_2,\ldots$ is a fixed infinite iid sequence with law $\mu$. Define

\begin{equation}
\mu_n=\frac1n\sum_{i=1}^n\delta_{X_i},\qquad
W_n=W_1^Q(\mu_n,\mu),\qquad
r_n=\sqrt{\frac{\log n}{n}},\qquad
\Xi_n=\frac{W_n}{r_n},\quad n\ge2.
\tag{1.1}\label{eq:original-1-1}
\end{equation}

All logarithms are natural. The centered empirical source is

\begin{equation}
Z_n=\sqrt n(\mu_n-\mu)
=\frac1{\sqrt n}\sum_{i=1}^n(\delta_{X_i}-\mu).
\tag{1.2}\label{eq:original-1-2}
\end{equation}

The symbol $Z_n$ denotes a signed measure, not the martingale below or a scalar standardized fluctuation.

For a mean-zero signed measure, or a source with the Sobolev regularity specified below, define the Lipschitz dual functional

\begin{equation}
\mathcal S(f)=\sup_{u\in K}\langle f,u\rangle,\qquad
K=\left\{u\in C(Q):\operatorname{Lip}_Q(u)\le1,\ \int_Q u\,dx=0\right\}.
\tag{1.3}\label{eq:original-1-3}
\end{equation}

Mean normalization leaves pairings with zero-mass sources unchanged. The class $K$ is compact in the uniform norm because it is equicontinuous and

\begin{equation*}
|u(x)|=\left|\int_Q(u(x)-u(y))\,dy\right|\le\sqrt2.
\end{equation*}

Fix a countable dense subset of $K$. Every finite-measure pairing has the same supremum over this subset, so the random functionals used are measurable. For finite-dimensional sources, pairings are finite-dimensional linear functions and the supremum is a continuous finite-dimensional function.

We use compact-space KR, Rademacher, and Sobolev weak-gradient theory. Tonelli handles nonnegative series and integrals; dominated convergence handles controlled limits; conditional expectations use the tower property. The first Borel--Cantelli lemma says that if $\sum_j\Pr(E_j)<\infty$, then $\Pr(E_j\text{occurs infinitely often})=0$, without independence of the events. Conditional Hoeffding, bounded differences, and maximal Azuma are proved fully below; Zhai's quantitative CLT is stated precisely where used.

Since $\operatorname{diam}(Q)=\sqrt2<2$, subtracting the midrange from any unit Lipschitz function gives uniform norm at most $\sqrt2/2<1$. This leaves the pairing with a probability difference unchanged. Hence

\begin{equation}
d_{BL}(\nu,\eta)=\mathcal S(\nu-\eta)=W_1^Q(\nu,\eta).
\tag{1.4}\label{eq:original-1-4}
\end{equation}

\subsection{Neumann spectrum, flux, and counting estimates}\label{orig:V:2}

We compute the cosine eigenfunctions, flux orthogonality, dimension, and $H^{-1}$ dual bound, obtaining derivative traces by lattice counting. The flux sign is uniformly $-\operatorname{div}J_N=f$; changing the overall sign does not affect norm bounds. The heat-to-sharp error is proved separately below.

For $k=(k_1,k_2)\in\mathbb N_0^2\setminus\{0\}$ set

\begin{equation}
s(k)=\#\{j:k_j>0\},\qquad
\varphi_k(x)=2^{s(k)/2}\cos(\pi k_1x_1)\cos(\pi k_2x_2),\qquad
\lambda_k=\pi^2|k|^2.
\tag{2.1}\label{eq:original-2-1}
\end{equation}

The constant mode is $\varphi_0=1$. One-dimensional cosine orthogonality and product integration give zero means for nonzero $\varphi_k$ and

\begin{equation}
\int_Q\varphi_k\varphi_l\,dx=\delta_{kl},\qquad
|\varphi_k(x)|\le2.
\tag{2.2}\label{eq:original-2-2}
\end{equation}

They satisfy $-\Delta\varphi_k=\lambda_k\varphi_k$ and the Neumann condition on every side. Integration by parts gives

\begin{equation}
\int_Q\nabla\varphi_k\cdot\nabla\varphi_l\,dx
=\lambda_l\int_Q\varphi_k\varphi_l\,dx
=\lambda_k\delta_{kl}.
\tag{2.3}\label{eq:original-2-3}
\end{equation}

Fix a real cutoff $L$ and set

\begin{equation}
\Pi_L f=\sum_{0<\lambda_k\le L}\langle f,\varphi_k\rangle\varphi_k,
\qquad m_N(L)=\sum_{0<\lambda_k\le L}\frac1{\lambda_k},
\qquad \xi_N^L=\sum_{0<\lambda_k\le L}g_k\varphi_k,
\tag{2.4}\label{eq:original-2-4}
\end{equation}

Here $g_k$ are independent standard real Gaussians. The subscript $N$ means Neumann, not sample size. Define

\begin{equation}
F_N(L)=\mathbb E\mathcal S(\xi_N^L).
\tag{2.5}\label{eq:original-2-5}
\end{equation}

Use the positive-gradient convention:

\begin{equation}
J_N(f)=\nabla(-\Delta_N)^{-1}f
=\sum_{k\ne0}\frac{\langle f,\varphi_k\rangle}{\lambda_k}\nabla\varphi_k,
\qquad -\operatorname{div}J_N(f)=f.
\tag{2.6}\label{eq:original-2-6}
\end{equation}

By \eqref{eq:original-2-3},

\begin{equation}
\|J_N(f)\|_2^2=\sum_{k\ne0}\frac{|\langle f,\varphi_k\rangle|^2}{\lambda_k}
=:\|f\|_{H_N^{-1}}^2.
\tag{2.7}\label{eq:original-2-7}
\end{equation}

For a finite-spectral source the normal flux vanishes on the boundary, so

\begin{equation*}
\langle f,u\rangle=\int_Q J_N(f)\cdot\nabla u\,dx.
\end{equation*}

For the mean-zero $L^2$ sources used here, finite spectral approximation gives $L^2$ convergence of both source and flux. Section~\ref{sec:preliminaries}, equations \eqref{eq:original-P15} and \eqref{eq:original-P18}, proved passage of the pairings to the limit. A test $u\in K$ belongs to $H^1$; from $|\nabla u|\le1$ almost everywhere and $|Q|=1$ we obtain $\|\nabla u\|_2\le1$. Thus

\begin{equation}
\mathcal S(f)\le\|J_N(f)\|_2,
\qquad |\mathcal S(f)-\mathcal S(h)|\le\|J_N(f-h)\|_2.
\tag{2.8}\label{eq:original-2-8}
\end{equation}

The second inequality follows from subadditivity and symmetry of $\mathcal S$: $|\mathcal S(f)-\mathcal S(h)|\le\mathcal S(f-h)$. We do not pair an undefined arbitrary distribution with a Lipschitz function.

Let $N_N(a)$ count modes in $0<\lambda_k\le a$. It counts nonnegative-quadrant lattice points in the disk of radius $R=\sqrt a/\pi$. Only a fixed-width boundary annulus and the two axes differ from the area comparison. The annulus has area $O(R+1)$ and the axis count is $O(R+1)$, so

\begin{equation}
N_N(a)=\frac14\pi R^2+O(R+1)
=\frac a{4\pi}+O(\sqrt a+1),\qquad N_N(a)\le Ca.
\tag{2.9}\label{eq:original-2-9}
\end{equation}

The last bound holds for every $a>0$: below the first positive eigenvalue $\pi^2$ its left side is zero, and enlarging the constant covers the rest.

Take $a_0\in(0,\pi^2)$. Stieltjes integration by parts gives

\begin{equation*}
m_N(L)=\int_{(a_0,L]}\frac1a\,dN_N(a)
=\frac{N_N(L)}L+\int_{a_0}^L\frac{N_N(a)}{a^2}\,da.
\end{equation*}

Insert \eqref{eq:original-2-9}. The main integral is $(4\pi)^{-1}\log L$, and the error at infinity is controlled by $a^{-3/2}+a^{-2}$. Thus

\begin{equation}
m_N(L)=\frac{\log L}{4\pi}+O(1).
\tag{2.10}\label{eq:original-2-10}
\end{equation}

We also need, uniformly in $0<u\le1$,

\begin{equation}
\sum_{k\ne0}\lambda_k e^{-2u\lambda_k}\le Cu^{-2}.
\tag{2.11}\label{eq:original-2-11}
\end{equation}

Split the spectrum into bands $[j/u,(j+1)/u)$. Band $j$ contains at most $C(j+1)/u$ points, and each term is at most $((j+1)/u)e^{-2j}$. The total is therefore at most

\begin{equation*}
\frac C{u^2}\sum_{j\ge0}(j+1)^2e^{-2j}<\infty\cdot u^{-2}.
\end{equation*}

\subsection{Empirical spectral covariance: a termwise calculation}\label{orig:V:3}

Write $a_{n,k}=\langle Z_n,\varphi_k\rangle$. By $\mu\varphi_k=0$,

\begin{equation}
a_{n,k}=\frac1{\sqrt n}\sum_{i=1}^n\varphi_k(X_i).
\tag{3.1}\label{eq:original-3-1}
\end{equation}

Its expectation is

\begin{equation*}
\mathbb E a_{n,k}=\frac1{\sqrt n}\sum_{i=1}^n\int_Q\varphi_k\,d\mu=0.
\end{equation*}

The covariance of two coefficients is

\begin{equation*}
\mathbb E[a_{n,k}a_{n,l}]
=\frac1n\sum_{i,j=1}^n\mathbb E[\varphi_k(X_i)\varphi_l(X_j)].
\end{equation*}

For $i\ne j$ independence gives $(\mu\varphi_k)(\mu\varphi_l)=0$; for $i=j$ the term is $\mu(\varphi_k\varphi_l)=\delta_{kl}$. There are $n$ diagonal terms, so

\begin{equation}
\mathbb E[a_{n,k}a_{n,l}]=\delta_{kl},\qquad
\mathbb E|a_{n,k}|^2=1.
\tag{3.2}\label{eq:original-3-2}
\end{equation}

Empirical coefficients at different frequencies are generally dependent. Multiplier second-moment calculations need only \eqref{eq:original-3-2}, not independence. Correspondingly, for $b_{n,k}=\langle\mu_n-\mu,\varphi_k\rangle=a_{n,k}/\sqrt n$,

\begin{equation}
\mathbb E|b_{n,k}|^2=\frac1n.
\tag{3.3}\label{eq:original-3-3}
\end{equation}

\subsection{Two-scale heat smoothing: fine displacement and the integral remainder}\label{orig:V:4}

Fix $B>2$ and set

\begin{equation}
t_n=\frac{(\log n)^B}{n},\qquad L_n=t_n^{-1}=\frac n{(\log n)^B}.
\tag{4.1}\label{eq:original-4-1}
\end{equation}

The Neumann heat semigroup is $P_t^N=e^{t\Delta_N}$ with multiplier $e^{-t\lambda_k}$. Thus the integer-index cutoff is $|k|\le\sqrt L/\pi$, not $|k|\le\sqrt L$.

\subsubsection{Heat contraction and particle displacement}\label{orig:V:4.1}

Let $\rho:\mathbb R\to[0,1]$ be the even, two-periodic triangular folding map. It is globally $1$-Lipschitz and the identity on $[0,1]$. If $B_t$ has covariance $2tI_2$, then

\begin{equation}
Y_t(x)=(\rho(x_1+B_t^1),\rho(x_2+B_t^2))
\tag{4.2}\label{eq:original-4-2}
\end{equation}

has the single-time reflected heat law on $P_t^N$, as checked by Gaussian images or cosine expansion. Under the same Gaussian increment,

\begin{equation*}
|Y_t(x)-Y_t(y)|\le|x-y|.
\end{equation*}

Thus a unit Lipschitz $u$ satisfies $|P_t^Nu(x)-P_t^Nu(y)|\le|x-y|$. The symmetric heat kernel has row mass one, so $P_t^N$ is self-adjoint and $P_t^N\mu=\mu$. Hence

\begin{equation}
\mathcal S(P_t^Nf)
=\sup_{u\in K}\langle f,P_t^Nu\rangle\le\mathcal S(f).
\tag{4.3}\label{eq:original-4-3}
\end{equation}

If mean normalization of $P_t^Nu$ needs adjustment, subtracting a constant preserves the pairing; in fact its mean is already zero.

Since $\rho(x_j)=x_j$,

\begin{equation*}
|Y_t(x)-x|\le|B_t|.
\end{equation*}

The two-dimensional Gaussian radius has density $(r/(2t))e^{-r^2/(4t)}$. Therefore

\begin{equation}
\mathbb E|B_t|=\int_0^\infty\frac{r^2}{2t}e^{-r^2/(4t)}\,dr
=4\sqrt t\int_0^\infty v^2e^{-v^2}\,dv
=\sqrt{\pi t}.
\tag{4.4}\label{eq:original-4-4}
\end{equation}

We used $r=2\sqrt t\,v$ and $\int_0^\infty v^2e^{-v^2}dv=\sqrt\pi/4$, the latter following from integration by parts and Gaussian integration. Conditional on the empirical sample, let $x$ have law $\mu_n$ and use independent heat displacement. Then $(x,Y_t(x))$ couples $(\mu_n,P_t^N\mu_n)$, giving

\begin{equation}
W_1^Q(\mu_n,P_t^N\mu_n)\le\sqrt{\pi t},\qquad 0<t\le1.
\tag{4.5}\label{eq:original-4-5}
\end{equation}

\subsubsection{Flux integration over the intermediate time interval}\label{orig:V:4.2}

Fix $0<s<t\le1$ and set

\begin{equation}
g_u=(-\Delta_N)P_u^N(\mu_n-\mu)
=\sum_{k\ne0}\lambda_ke^{-u\lambda_k}b_{n,k}\varphi_k.
\tag{4.6}\label{eq:original-4-6}
\end{equation}

By \eqref{eq:original-2-7}, \eqref{eq:original-3-3}, and Tonelli,

\begin{equation}
\mathbb E\|g_u\|_{H_N^{-1}}^2
=\mathbb E\sum_{k\ne0}\lambda_ke^{-2u\lambda_k}|b_{n,k}|^2
=\frac1n\sum_{k\ne0}\lambda_ke^{-2u\lambda_k}
\le\frac C{nu^2}.
\tag{4.7}\label{eq:original-4-7}
\end{equation}

Interchanging nonnegative terms needs no prior integrability; the finite bound on the right then proves integrability.

First integrate at finite cutoff:

\begin{equation*}
\int_s^t\lambda_ke^{-u\lambda_k}\,du=e^{-s\lambda_k}-e^{-t\lambda_k}.
\end{equation*}

Thus $\Pi_M(P_s^N\mu_n-P_t^N\mu_n)=\int_s^t\Pi_Mg_u\,du$ at finite spectrum. To let $M\to\infty$, note that for fixed sample size $|b_{n,k}|\le2$. On the closed interval $[s,t]$ the flux tail satisfies

\begin{equation*}
\sup_{u\in[s,t]}\sum_{\lambda_k>M}\lambda_ke^{-2u\lambda_k}|b_{n,k}|^2
\le4\sum_{\lambda_k>M}\lambda_ke^{-2s\lambda_k}\longrightarrow0.
\end{equation*}

We must also control the source's $L^2$ tail to stay within the integration-by-parts limiting scope of Section~\ref{sec:preliminaries}. Spectral orthogonality and the same coefficient bound give

\begin{equation*}
\sup_{u\in[s,t]}\|g_u-\Pi_Mg_u\|_2^2
\le4\sum_{\lambda_k>M}\lambda_k^2e^{-2s\lambda_k}
\longrightarrow0.
\end{equation*}

For fixed $s>0$ both tails are summable: the first is integer order one in Section~\ref{sec:preliminaries}, equation \eqref{eq:original-P28}, and the second is integer order two. Source and flux thus converge uniformly in their respective $L^2$ spaces, allowing passage of each test pairing. In particular, finite projections of $g_u$ converge uniformly in $H_N^{-1}$, as do their integrals in this norm. Hence

\begin{equation}
P_s^N\mu_n-P_t^N\mu_n=\int_s^t g_u\,du\quad\text{in }H_N^{-1}.
\tag{4.8}\label{eq:original-4-8}
\end{equation}

For every $u\in K$ the pairing is linear and continuous; alternatively use the triangle inequality for the $L^2$ flux integral directly. By \eqref{eq:original-2-8},

\begin{equation*}
W_1^Q(P_s^N\mu_n,P_t^N\mu_n)
\le\int_s^t\|g_u\|_{H_N^{-1}}\,du.
\end{equation*}

Tonelli, Cauchy--Schwarz, and \eqref{eq:original-4-7} give

\begin{equation}
\mathbb E W_1^Q(P_s^N\mu_n,P_t^N\mu_n)
\le\int_s^t\left(\mathbb E\|g_u\|_{H_N^{-1}}^2\right)^{1/2}\,du
\le\frac C{\sqrt n}\int_s^t\frac{du}{u}
=\frac C{\sqrt n}\log\frac ts.
\tag{4.9}\label{eq:original-4-9}
\end{equation}

\subsubsection{Choosing scales and comparison with empirical transport}\label{orig:V:4.3}

The triangle inequality, \eqref{eq:original-4-5}, and \eqref{eq:original-4-9} give

\begin{equation}
\mathbb E W_1^Q(\mu_n,P_t^N\mu_n)
\le\sqrt{\pi s}+\frac C{\sqrt n}\log\frac ts.
\tag{4.10}\label{eq:original-4-10}
\end{equation}

For sufficiently large $n$ choose $s=1/n$ and $t=t_n$. Then $t_n>s$ and $t_n\le1$, and

\begin{equation*}
\sqrt{\pi s}=\frac{\sqrt\pi}{\sqrt n},\qquad
\log(t_n/s)=\log((\log n)^B)=B\log\log n.
\end{equation*}

Thus

\begin{equation}
\mathbb E W_1^Q(\mu_n,P_{t_n}^N\mu_n)
\le\frac{C_B(1+\log\log n)}{\sqrt n}=o(r_n),
\tag{4.11}\label{eq:original-4-11}
\end{equation}

because $(1+\log\log n)/\sqrt{\log n}\to0$. Applying \eqref{eq:original-4-5} directly to $t=t_n$ would give ratio $\sqrt\pi(\log n)^{(B-1)/2}$, which does not vanish. This explains the need for two scales.

Samplewise, \eqref{eq:original-4-3} and subadditivity give

\begin{equation*}
0\le W_n-\mathcal S(P_{t_n}^N(\mu_n-\mu))
\le W_1^Q(\mu_n,P_{t_n}^N\mu_n).
\end{equation*}

Using $\mu_n-\mu=n^{-1/2}Z_n$ and positive homogeneity, then taking expectation, yields

\begin{equation}
\mathbb E W_n=\frac1{\sqrt n}\mathbb E\mathcal S(P_{t_n}^NZ_n)+\varepsilon_n^{(s)},
\quad 0\le\varepsilon_n^{(s)}\le\frac{C_B(1+\log\log n)}{\sqrt n}.
\tag{4.12}\label{eq:original-4-12}
\end{equation}

\subsection{A bounded error between heat smoothing and spectral truncation}\label{orig:V:5}

The precise conclusion needed is the mean-energy error $O(1)=o(\sqrt{m_N(L_n)})$. A bounded absolute error suffices to vanish after leading-order normalization; the absolute error itself need not tend to zero.

For sufficiently large real $L$ set $t=1/L$ and write

\begin{equation*}
H_{n,L}=P_{1/L}^NZ_n-\Pi_LZ_n,
\qquad a_L(\lambda)=e^{-\lambda/L}-\mathbf1_{\lambda\le L}.
\end{equation*}

Gradient orthogonality and \eqref{eq:original-3-2} give

\begin{equation}
\mathbb E\|J_N(H_{n,L})\|_2^2
=\sum_{k\ne0}\frac{a_L(\lambda_k)^2}{\lambda_k}.
\tag{5.1}\label{eq:original-5-1}
\end{equation}

First prove the identity at finite spectrum, then apply monotone convergence/Tonelli to nonnegative series. For fixed $L$ the heat tails are exponentially summable.

Split the right side into three parts. For $\lambda\le L/2$ we have $1-e^{-\lambda/L}\le\lambda/L$, so

\begin{equation*}
\sum_{\lambda_k\le L/2}\frac{a_L(\lambda_k)^2}{\lambda_k}
\le L^{-2}\sum_{\lambda_k\le L/2}\lambda_k
\le L^{-2}\frac L2 N_N(L/2)\le C.
\end{equation*}

For $L/2<\lambda\le L$ we have $|a_L|\le1$, so

\begin{equation*}
\sum_{L/2<\lambda_k\le L}\frac{a_L(\lambda_k)^2}{\lambda_k}
\le\frac2L N_N(L)\le C.
\end{equation*}

For $\lambda>L$ the multiplier is $e^{-\lambda/L}$. Splitting into bands $2^jL<\lambda\le2^{j+1}L$ gives

\begin{equation*}
\sum_{\lambda_k>L}\frac{e^{-2\lambda_k/L}}{\lambda_k}
\le\sum_{j\ge0}\frac{e^{-2^{j+1}}}{2^jL}N_N(2^{j+1}L)
\le2C\sum_{j\ge0}e^{-2^{j+1}}\le C'.
\end{equation*}

Thus \eqref{eq:original-5-1} has a uniform constant bound for all $n,L$. Apply \eqref{eq:original-2-8} and Cauchy--Schwarz:

\begin{equation}
\left|\mathbb E\mathcal S(P_{1/L}^NZ_n)-\mathbb E\mathcal S(\Pi_LZ_n)\right|
\le\mathbb E\|J_N(H_{n,L})\|_2
\le\left(\mathbb E\|J_N(H_{n,L})\|_2^2\right)^{1/2}\le C.
\tag{5.2}\label{eq:original-5-2}
\end{equation}

The conclusion is $O(1)$, not $o(1)$. At $L=L_n$, the required statement is $O(1)/\sqrt{\log n}\to0$, which has been proved.

\subsection{Zhai's theorem and Gaussian replacement in growing dimension}\label{orig:V:6}

\subsubsection{Precise statement of the external theorem}\label{orig:V:6.1}

We use the quantitative central limit theorem of
\cite[Theorem~1.1]{Zhai}.

\textbf{Theorem.} Let $Y_1,\ldots,Y_n$ be independent random vectors in $\mathbb R^d$, each centered with common covariance matrix $\Sigma$, and satisfying $|Y_i|\le\beta$ almost surely. Let $S=n^{-1/2}\sum_iY_i$ and let $G$ be a centered Gaussian with covariance $\Sigma$. Then

\begin{equation}
W_2(\mathcal L(S),\mathcal L(G))
\le C\frac{\sqrt d\,\beta\log n}{\sqrt n},\qquad n\ge2.
\tag{6.1}\label{eq:original-6-1}
\end{equation}

Here $C$ is universal and independent of $n,d,\beta,\Sigma$, and $W_2$ is Euclidean quadratic Wasserstein distance. The theorem's original bound is $5\sqrt d\,\beta(1+\log n)/\sqrt n$; for $n\ge2$ absorb $1+\log n$ into a universal constant times $\log n$. Identity covariance is unnecessary, and no $\Sigma^{-1}$, condition number, or uniform smallest-eigenvalue bound appears. The same bound applies separately for each $n$, so row parameters $d,\beta,\Sigma$ may vary.

\subsubsection{The finite-dimensional flux Hilbert space and the hypotheses}\label{orig:V:6.2}

Let

\begin{equation*}
\mathcal H_L=\operatorname{span}\{\nabla\varphi_k:0<\lambda_k\le L\}
\subset L^2(Q;\mathbb R^2),\qquad
e_k=\frac{\nabla\varphi_k}{\sqrt{\lambda_k}}.
\end{equation*}

By \eqref{eq:original-2-3}, the family $\{e_k\}$ is orthonormal, so the real dimension is

\begin{equation}
d_L=N_N(L)\le CL.
\tag{6.2}\label{eq:original-6-2}
\end{equation}

For $x\in Q$ define the truncated atomic flux

\begin{equation}
V_x=J_N(\Pi_L(\delta_x-\mu))
=\sum_{0<\lambda_k\le L}\frac{\varphi_k(x)}{\sqrt{\lambda_k}}e_k.
\tag{6.3}\label{eq:original-6-3}
\end{equation}

Since each $\mu\varphi_k=0$, we have $\mathbb EV_X=0$. The covariance of its basis coordinates is

\begin{equation*}
\mathbb E\left[\frac{\varphi_k(X)}{\sqrt{\lambda_k}}
\frac{\varphi_l(X)}{\sqrt{\lambda_l}}\right]
=\frac{\delta_{kl}}{\sqrt{\lambda_k\lambda_l}},
\end{equation*}

Therefore

\begin{equation}
\Sigma_L=\operatorname{diag}(\lambda_k^{-1}:0<\lambda_k\le L).
\tag{6.4}\label{eq:original-6-4}
\end{equation}

The flux norm at each $x$ satisfies

\begin{equation}
\|V_x\|_2^2=\sum_{0<\lambda_k\le L}\frac{\varphi_k(x)^2}{\lambda_k}
\le4m_N(L),\qquad \beta_L=2\sqrt{m_N(L)}.
\tag{6.5}\label{eq:original-6-5}
\end{equation}

Moreover,

\begin{equation}
J_N(\Pi_LZ_n)=\frac1{\sqrt n}\sum_{i=1}^nV_{X_i},\qquad
J_N(\xi_N^L)=\sum_{0<\lambda_k\le L}\frac{g_k}{\sqrt{\lambda_k}}e_k.
\tag{6.6}\label{eq:original-6-6}
\end{equation}

The second is the centered Gaussian in $\mathcal H_L$ with covariance $\Sigma_L$. The first has independent summands because each depends only on its corresponding $X_i$. Identify $\mathcal H_L$ isometrically with $\mathbb R^{d_L}$ through basis coordinates. All hypotheses of \eqref{eq:original-6-1} have been checked, giving

\begin{equation}
W_2\bigl(\mathcal L(J_N(\Pi_LZ_n)),\mathcal L(J_N(\xi_N^L))\bigr)
\le C\sqrt{L m_N(L)}\frac{\log n}{\sqrt n}.
\tag{6.7}\label{eq:original-6-7}
\end{equation}

\subsubsection{From a flux coupling estimate to a mean-energy comparison}\label{orig:V:6.3}

The linear map from truncated sources to fluxes is bijective onto $\mathcal H_L$: each coordinate $f_k$ maps to $f_k/\sqrt{\lambda_k}$ and every $\lambda_k>0$. Define

\begin{equation*}
\Phi_L(J_N(f))=\mathcal S(f).
\end{equation*}

By \eqref{eq:original-2-8} the functional $\Phi_L$ is $1$-Lipschitz on $\mathcal H_L$. For any coupling $(U,V)$ of the two flux laws,

\begin{equation*}
|\mathbb E\Phi_L(U)-\mathbb E\Phi_L(V)|
\le\mathbb E|\Phi_L(U)-\Phi_L(V)|
\le\mathbb E\|U-V\|_2
\le(\mathbb E\|U-V\|_2^2)^{1/2}.
\end{equation*}

Taking the infimum over couplings gives the following, without requiring a common optimal coupling for every $n$:

\begin{equation}
|\mathbb E\mathcal S(\Pi_LZ_n)-F_N(L)|
\le C\sqrt{L m_N(L)}\frac{\log n}{\sqrt n}.
\tag{6.8}\label{eq:original-6-8}
\end{equation}

\subsubsection{\texorpdfstring{The error when dimension grows with $n$}{The error in growing dimension}}\label{orig:V:6.4}

Take $L=L_n=n/(\log n)^B$. By \eqref{eq:original-2-10} we have $m_N(L_n)=O(\log n)$, so

\begin{equation}
\sqrt{L_n m_N(L_n)}\frac{\log n}{\sqrt n}
\le C_B\frac{\sqrt n\,(\log n)^{(1-B)/2}\log n}{\sqrt n}
=C_B(\log n)^{(3-B)/2}.
\tag{6.9}\label{eq:original-6-9}
\end{equation}

Dividing by the required scale $\sqrt{\log n}$ gives

\begin{equation}
C_B(\log n)^{1-B/2}\longrightarrow0\quad(B>2).
\tag{6.10}\label{eq:original-6-10}
\end{equation}

For $2<B\le3$ the unnormalized bound need not vanish; this does not prevent \eqref{eq:original-6-10}. For each $n$ apply the same quantitative theorem to one row of finite-dimensional vectors, which is the valid triangular-array application. Rows need not be independent and couplings need not be compatible. This is only a mean comparison, not a deduction of empirical almost-sure convergence directly from the CLT.

\subsection{Combining errors and checking the constant}\label{orig:V:7}

Combining \eqref{eq:original-4-12}, \eqref{eq:original-5-2}, and \eqref{eq:original-6-8} gives

\begin{equation}
\left|\mathbb EW_n-\frac1{\sqrt n}F_N(L_n)\right|
\le\frac{C_B}{\sqrt n}\left[1+\log\log n+(\log n)^{(3-B)/2}\right].
\tag{7.1}\label{eq:original-7-1}
\end{equation}

Divide by $r_n=\sqrt{\log n}/\sqrt n$. The three errors are

\begin{equation*}
\frac1{\sqrt{\log n}}\to0,\qquad
\frac{\log\log n}{\sqrt{\log n}}\to0,\qquad
(\log n)^{1-B/2}\to0.
\end{equation*}

Hence

\begin{equation}
\mathbb E\Xi_n=\frac{F_N(L_n)}{\sqrt{\log n}}+o(1).
\tag{7.2}\label{eq:original-7-2}
\end{equation}

Also,

\begin{equation*}
\log L_n=\log n-B\log\log n,
\quad m_N(L_n)=\frac{\log n-B\log\log n}{4\pi}+O(1),
\end{equation*}

therefore

\begin{equation}
\frac{m_N(L_n)}{\log n}
=\frac1{4\pi}-\frac B{4\pi}\frac{\log\log n}{\log n}
+O((\log n)^{-1})\longrightarrow\frac1{4\pi}.
\tag{7.3}\label{eq:original-7-3}
\end{equation}

Furthermore, \eqref{eq:original-2-8}, Gaussian second moments, and Jensen give

\begin{equation}
0\le F_N(L)\le\mathbb E\|J_N(\xi_N^L)\|_2
\le\left(\mathbb E\|J_N(\xi_N^L)\|_2^2\right)^{1/2}
=\sqrt{m_N(L)}.
\tag{7.4}\label{eq:original-7-4}
\end{equation}

Thus the normalization-factor error times $F_N/\sqrt{m_N}\le1$ still vanishes. By \eqref{eq:original-7-2}--\eqref{eq:original-7-4},

\begin{equation}
\boxed{\mathbb E\Xi_n
=\frac1{2\sqrt\pi}\frac{F_N(L_n)}{\sqrt{m_N(L_n)}}+o(1).}
\tag{7.5}\label{eq:original-7-5}
\end{equation}

This is the reduction of empirical means to Neumann width. Section~\ref{sec:comparison} established

\begin{equation}
\frac{F_N(L)}{\sqrt{m_N(L)}}\longrightarrow\kappa_1\in(0,\infty)
\quad\text{as all real }L\to\infty,
\tag{7.6}\label{eq:original-7-6}
\end{equation}

Inserting this limit into \eqref{eq:original-7-5} gives

\begin{equation}
\mathbb E\Xi_n\longrightarrow c_*=\frac{\kappa_1}{2\sqrt\pi}.
\tag{7.7}\label{eq:original-7-7}
\end{equation}

The dual-constant and stopped-martingale arguments in Section~\ref{sec:periodic} give $0<\kappa_1\le\sqrt\pi/2$; Section~\ref{sec:comparison} transfers the periodic limit to the square, so $0<c_*\le1/4$. The single-sample $Z_n$ has identity spectral covariance. An unnormalized difference of two independent empirical samples would have twice the covariance and cannot use the same constant normalization.

\subsection{Conditional Hoeffding moments via variance under exponential tilting}\label{orig:V:8}

Let a real $Y$ have mean zero and take values in a bounded interval of length $c$. Define $K(s)=\log\mathbb Ee^{sY}$. Boundedness permits differentiation under the integral:

\begin{equation*}
K'(s)=\frac{\mathbb E[Ye^{sY}]}{\mathbb Ee^{sY}},\qquad
K''(s)=\frac{\mathbb E[Y^2e^{sY}]}{\mathbb Ee^{sY}}-(K'(s))^2.
\end{equation*}

The quantity $K''(s)$ is variance under the exponentially tilted law. If $m$ is the interval midpoint, then under any supported law

\begin{equation*}
\operatorname{Var}(Y)=\inf_a\mathbb E(Y-a)^2
\le\mathbb E(Y-m)^2\le c^2/4.
\end{equation*}

Thus $K''(s)\le c^2/4$ for every real $s$. Since $K(0)=0$ and $K'(0)=\mathbb EY=0$, the integral Taylor formula gives

\begin{equation*}
K(s)=s^2\int_0^1(1-v)K''(vs)\,dv\le s^2c^2/8.
\end{equation*}

Consequently

\begin{equation}
\mathbb Ee^{sY}\le e^{s^2c^2/8}.
\tag{8.1}\label{eq:original-8-1}
\end{equation}

If a variable is centered conditional on a sigma-algebra and its conditional support lies in an interval of length at most $c$, the same calculation under each conditional law gives the conditional form of \eqref{eq:original-8-1}. The interval may shift with the conditioning value; only its length enters the bound.

\subsection{Bounded differences: two-sided tails from the Doob martingale}\label{orig:V:9}

Let $F=F(Y_1,\ldots,Y_n)$ be bounded and measurable with independent inputs. Replacing coordinate $i$ changes $F$ by at most $c_i$. Our empirical distance is bounded and satisfies these conditions. Define

\begin{equation*}
T_i=\mathbb E[F\mid Y_1,\ldots,Y_i],\quad T_0=\mathbb EF,
\qquad D_i=T_i-T_{i-1}.
\end{equation*}

Fix the first $i-1$ inputs. As a function of $Y_i$, the quantity $T_i$ integrates out the remaining independent coordinates. For each fixed remaining tuple, replacing $Y_i$ changes it by at most $c_i$, and integration preserves this bound. Hence $D_i$ has conditional range length at most $c_i$ and conditional mean zero.

By \eqref{eq:original-8-1},

\begin{equation*}
\mathbb E[e^{sD_i}\mid Y_1,\ldots,Y_{i-1}]\le e^{s^2c_i^2/8}.
\end{equation*}

Repeated application of the tower property gives

\begin{equation*}
\mathbb Ee^{s(F-\mathbb EF)}
=\mathbb Ee^{s\sum_iD_i}\le\exp\left(\frac{s^2}8\sum_i c_i^2\right).
\end{equation*}

Write $A=\sum_i c_i^2>0$. For $s>0$, Markov gives upper-tail bound $\le e^{-st+s^2A/8}$. Its exponent is minimized at $s=4t/A$, with value $-2t^2/A$. Repeat for $-F$ to obtain

\begin{equation}
\Pr(|F-\mathbb EF|\ge t)\le2e^{-2t^2/A}.
\tag{9.1}\label{eq:original-9-1}
\end{equation}

If $A=0$, no input affects $F$ and the variable is deterministic; handle this case separately.

\subsection{Proof of maximal Azuma}\label{orig:V:10}

Let the real martingale $M_0=0$ have increments $D_k=M_k-M_{k-1}$ with $|D_k|\le a$ almost surely. Their conditional support lies in $[-a,a]$ with range length $2a$. By \eqref{eq:original-8-1},

\begin{equation*}
\mathbb E[e^{sD_k}\mid\mathcal F_{k-1}]\le e^{s^2a^2/2}.
\end{equation*}

Thus

\begin{equation*}
L_k=\exp(sM_k-s^2a^2k/2)
\end{equation*}

is a nonnegative supermartingale with $L_0=1$. Fix terminal time $N$ and let $\tau$ be the first time $M_k\ge t$, or $N$ if none occurs. This is bounded stopping. Iterating the conditional supermartingale inequality, or bounded optional stopping, gives $\mathbb EL_\tau\le1$.

On the crossing event $M_\tau\ge t$ and $\tau\le N$, so

\begin{equation*}
L_\tau\ge e^{st-s^2a^2N/2}.
\end{equation*}

The upper crossing probability is at most $e^{-st+s^2a^2N/2}$. Choose $s=t/(Na^2)$ and repeat the argument for $-M$ to obtain

\begin{equation}
\Pr\left(\max_{k\le N}|M_k|\ge t\right)
\le2\exp\left(-\frac{t^2}{2Na^2}\right).
\tag{10.1}\label{eq:original-10-1}
\end{equation}

If $a=0$ the martingale is identically zero. Increment independence was never required; zero conditional mean and conditional boundedness suffice.

\subsection{The empirical-distance submartingale and Doob compensator}\label{orig:V:11}

This subsection and the next three consider only the $Q=[0,1]^2$, Euclidean metric, $\mu=dx$, and same infinite iid sequence fixed in Section~\ref{orig:V:1}. Fix $D=\operatorname{diam}(Q)=\sqrt2$, keeping $W_n=W_1^Q(\mu_n,\mu)$, $r_n=\sqrt{\log n/n}$, and $\Xi_n=W_n/r_n$. Assume a known positive finite mean limit

\begin{equation}
\mathbb E\Xi_n\longrightarrow C\in(0,\infty).
\tag{11.1}\label{eq:original-11-1}
\end{equation}

Set $h_n=\sqrt{n\log n}$. We prove almost-sure convergence conditional on this positive finite mean limit. Its existence comes from Section~\ref{sec:comparison} and subsection~\ref{orig:V:7} here and cannot be inferred backwards from the conditional conclusion.

For zero-mass signed measures on $Q$, the KR norm is the supremum of pairings over the compact class $K$ of Section~\ref{orig:V:1}. Its fixed countable dense subset ensures measurability. Set

\begin{equation}
S_n=nW_n=\left\|\sum_{i=1}^n(\delta_{X_i}-\mu)\right\|_{\rm KR},\quad S_0=0,
\qquad d_n=\mathbb ES_n,
\quad\mathcal F_n=\sigma(X_1,\ldots,X_n).
\tag{11.2}\label{eq:original-11-2}
\end{equation}

For each fixed test $u$, the new atomic pairing $u(X_n)-\mu u$ has zero conditional mean. Conditional expectation of the supremum is at least each conditional expected pairing. Taking the supremum over countably many tests gives the submartingale inequality

\begin{equation}
\mathbb E[S_n\mid\mathcal F_{n-1}]\ge S_{n-1}.
\tag{11.3}\label{eq:original-11-3}
\end{equation}

The unique coupling of $\delta_x$ and $\mu$ is $\delta_x\otimes\mu$, so

\begin{equation*}
\|\delta_x-\mu\|_{\rm KR}=W_1^Q(\delta_x,\mu)=\int_Q|x-y|\,dy\le D=\sqrt2.
\end{equation*}

The reverse norm triangle inequality gives

\begin{equation}
|S_n-S_{n-1}|\le\|\delta_{X_n}-\mu\|_{\rm KR}\le D.
\tag{11.4}\label{eq:original-11-4}
\end{equation}

Define

\begin{equation}
B_n=\mathbb E[S_n-S_{n-1}\mid\mathcal F_{n-1}],\qquad
A_n=\sum_{i=1}^nB_i,\qquad M_n=S_n-A_n.
\tag{11.5}\label{eq:original-11-5}
\end{equation}

Equations \eqref{eq:original-11-3}--\eqref{eq:original-11-4} give $0\le B_n\le D$, so $A_n$ is pathwise nondecreasing. The tower property yields

\begin{equation*}
\mathbb EA_n=\sum_i\mathbb E(S_i-S_{i-1})=\mathbb ES_n=d_n.
\end{equation*}

Also,

\begin{equation*}
M_n-M_{n-1}=S_n-S_{n-1}-B_n
\end{equation*}

The increment is conditionally centered and bounded in absolute value by $2D=2\sqrt2$. Thus $M_n$ is a martingale with bounded increments. The submartingale property of $S_n$ is not samplewise monotonicity; only $A_n$ is sandwiched below.

\subsection{The martingale remainder for all sample sizes}\label{orig:V:12}

In \eqref{eq:original-10-1} take $a=2D=2\sqrt2$, $N=2^{j+1}$, and $t=\varepsilon\sqrt{2^jj}$ to obtain

\begin{equation}
\Pr\left(\max_{m\le2^{j+1}}|M_m|\ge\varepsilon\sqrt{2^jj}\right)
\le2\exp\left(-\frac{\varepsilon^2 2^jj}{2\cdot2^{j+1}\cdot4D^2}\right)
=2e^{-\varepsilon^2j/32}.
\tag{12.1}\label{eq:original-12-1}
\end{equation}

For each fixed $\varepsilon>0$ this is a summable geometric series in $j$. First Borel--Cantelli implies only finitely many occurrences, without event independence. Take $\varepsilon=1/m$ and intersect the countably many full-measure events. On the same event, the maximum of \eqref{eq:original-12-1} divided by $\sqrt{2^jj}$ tends to zero.

If $2^j\le n<2^{j+1}$, then

\begin{equation*}
h_n=\sqrt{n\log n}\ge\sqrt{2^j j\log2}.
\end{equation*}

Consequently

\begin{equation}
\frac{|M_n|}{h_n}
\le\frac1{\sqrt{\log2}}\frac{\max_{m\le2^{j+1}}|M_m|}{\sqrt{2^jj}}
\longrightarrow0.
\tag{12.2}\label{eq:original-12-2}
\end{equation}

This directly covers every integer $n$, with neither a martingale LIL nor a conditional-variance growth assumption.

\subsection{Concentration on a fixed geometric subsequence and countable events}\label{orig:V:13}

Replacing $X_i\in Q$ by $X_i'\in Q$ couples the empirical laws by moving only mass $1/n$ at cost at most $D/n=\sqrt2/n$. The reverse triangle inequality for $W_1$ shows that $W_n$ changes by at most $\sqrt2/n$. In \eqref{eq:original-9-1} take each $c_i=\sqrt2/n$; then $\sum_i c_i^2=n(2/n^2)=2/n$, giving

\begin{equation*}
\Pr(|W_n-\mathbb EW_n|\ge t)
\le2\exp(-nt^2),\qquad D^2=2.
\end{equation*}

Set $t=\varepsilon r_n$. Since $nr_n^2=\log n$,

\begin{equation}
\Pr(|\Xi_n-\mathbb E\Xi_n|\ge\varepsilon)
\le2\exp(-\varepsilon^2\log n).
\tag{13.1}\label{eq:original-13-1}
\end{equation}

This probability need not be summable over all $n$, so it does not directly prove full-sequence almost-sure vanishing of the deviation.

Fix $b>1$ and set $n_k=\lfloor b^k\rfloor$. Eventually $n_k\ge b^k/2$, hence

\begin{equation*}
\Pr(|\Xi_{n_k}-\mathbb E\Xi_{n_k}|\ge\varepsilon)
\le2\exp[-\varepsilon^2(k\log b-\log2)].
\end{equation*}

For $k$ this is a summable geometric series. Apply Borel--Cantelli and intersect over $\varepsilon=1/m$ to obtain a full-measure event $\Omega_b$ on which

\begin{equation}
\frac{S_{n_k}-d_{n_k}}{h_{n_k}}
=\Xi_{n_k}-\mathbb E\Xi_{n_k}\longrightarrow0.
\tag{13.2}\label{eq:original-13-2}
\end{equation}

Only finitely many initial indices can repeat, which has no effect. Combining \eqref{eq:original-12-2}, \eqref{eq:original-11-1}, and $d_n/h_n=\mathbb E\Xi_n$ gives

\begin{equation}
\frac{A_{n_k}}{h_{n_k}}
=\frac{S_{n_k}-d_{n_k}}{h_{n_k}}+\frac{d_{n_k}}{h_{n_k}}-\frac{M_{n_k}}{h_{n_k}}
\longrightarrow C.
\tag{13.3}\label{eq:original-13-3}
\end{equation}

\subsection{\texorpdfstring{From geometric subsequences to all $n$: sandwiching the compensator}{From geometric subsequences to all sample sizes}}\label{orig:V:14}

We keep the same square, iid sequence, and $C\in(0,\infty)$, expand the sandwich for $A_n/h_n$, and finally take the common full-measure event for countably many geometric ratios.

When $n_k\le n<n_{k+1}$, both $A_n$ and $h_n$ are nonnegative and nondecreasing. Thus

\begin{equation}
\frac{A_{n_k}}{h_{n_{k+1}}}\le\frac{A_n}{h_n}\le\frac{A_{n_{k+1}}}{h_{n_k}}.
\tag{14.1}\label{eq:original-14-1}
\end{equation}

The relative floor error vanishes, so

\begin{equation*}
\frac{n_{k+1}}{n_k}\to b,\quad
\frac{\log n_{k+1}}{\log n_k}\to1,\quad
\frac{h_{n_{k+1}}}{h_{n_k}}\to\sqrt b.
\end{equation*}

By \eqref{eq:original-13-3} the left side of \eqref{eq:original-14-1} tends to $C/\sqrt b$ and the right side to $C\sqrt b$. For this fixed $b$, almost surely,

\begin{equation}
\frac C{\sqrt b}\le\liminf_n\frac{A_n}{h_n}
\le\limsup_n\frac{A_n}{h_n}\le C\sqrt b.
\tag{14.2}\label{eq:original-14-2}
\end{equation}

Now take only the countable sequence $b_j=1+1/j$. Intersect the full-measure events for each $\Omega_{b_j}$ and for \eqref{eq:original-12-2}. The intersection still has probability one, and \eqref{eq:original-14-2} holds there for all $j$ simultaneously. Let $j\to\infty$ to get

\begin{equation*}
\frac{A_n}{h_n}\longrightarrow C.
\end{equation*}

Finally, $S_n=A_n+M_n$ and \eqref{eq:original-12-2} give

\begin{equation}
\boxed{\Xi_n=\frac{S_n}{h_n}\longrightarrow C\quad\text{a.s.}}
\tag{14.3}\label{eq:original-14-3}
\end{equation}

This uses the same infinite iid sequence of Section~\ref{orig:V:1} and covers all integer sample sizes. Combining the preceding Gaussian limit with \eqref{eq:original-7-7} establishes the positive finite mean limit; apply this result with $C=\kappa_1/(2\sqrt\pi)$. The final combination follows in the last section.

\section{Final combination of the main theorem}\label{sec:conclusion}

We now combine the proved limits on their actual domains and with their actual normalizations. The constant comes from the product of the normalized Gaussian width and the square root of the Neumann trace coefficient. The probability argument then applies to the original sample sequence. No compatible Gaussian approximation couplings across sample sizes are required; the pathwise argument uses the original infinite sample sequence throughout.

\begin{proof}[Final combination for Theorem~\ref{thm:main}]

Apply the proved statements in order, retaining the original probability space, area law, and sample sequence, with no additional analytic-limit assumptions.

Section~\ref{sec:preliminaries} gives, for every probability pair in the original problem,

\begin{equation*}
d_{BL}(\nu,\eta)=W_1^Q(\nu,\eta).
\end{equation*}

Section~\ref{sec:periodic} proves the periodic endpoint limit and positive finiteness of $\kappa_1$. Section~\ref{sec:comparison} transfers it to the square by common white noise and real--complex coefficient calculations, yielding the limit for every real cutoff

\begin{equation*}
\frac{F_N(L)}{\sqrt{m_N(L)}}\longrightarrow\kappa_1.
\end{equation*}

In Section~\ref{sec:empirical} fix any $B>2$ and choose $L_n=n/(\log n)^B$. Then $L_n\to\infty$ and

\begin{equation*}
\frac{m_N(L_n)}{\log n}\longrightarrow\frac1{4\pi}.
\end{equation*}

Two-scale smoothing, spectral truncation, and Gaussian replacement in flux space give the combined error $o(\sqrt{\log n/n})$. Hence

\begin{equation*}
\sqrt{\frac n{\log n}}\,\mathbb EW_1^Q(\mu_n,\mu)
=\frac{F_N(L_n)}{\sqrt{m_N(L_n)}}
\sqrt{\frac{m_N(L_n)}{\log n}}+o(1)
\longrightarrow\frac{\kappa_1}{2\sqrt\pi}
=:c_*.
\end{equation*}

By $\kappa_1\in(0,\infty)$ the constant $c_*$ is deterministic, positive, and finite, meeting every hypothesis of the almost-sure upgrade in Section~\ref{sec:empirical}. Along the original iid sequence, the submartingale and nondecreasing compensator, maximal Azuma, and countable geometric subsequences yield

\begin{equation*}
\sqrt{\frac n{\log n}}\,W_1^Q(\mu_n,\mu)
\longrightarrow c_*\qquad\text{almost surely, along all integers }n.
\end{equation*}

For each $n$ use the metric identity of Section~\ref{sec:preliminaries} to obtain the requested conclusion

\begin{equation*}
\boxed{
\sqrt{\frac n{\log n}}\,d_{BL}(\mu_n,\mu)
\longrightarrow\frac{\kappa_1}{2\sqrt\pi}\in(0,\infty)
\quad\text{almost surely, along all integers }n.}
\end{equation*}

The original existence question therefore has an affirmative answer. The constant is identified rigorously through periodic Gaussian energy, without using numerical approximations from the literature to define or prove it.

\end{proof}
\section{Foundational interfaces, integration, and measurability details}\label{app:analytic}
This section expands the interfaces actually used above. Space norms, source coefficients, and probability laws are unchanged. The new estimates can replace the corresponding foundational inputs without changing the main constant.

\subsection{Compact-space KR: finite-grid strong duality and its limit}
Let $(M,d)$ be compact and $\nu,\eta$ be Borel probability measures. Fix $\delta>0$ and a finite $\delta$-net $\{x_1,\ldots,x_m\}$. Break ties by the smallest index and take the Borel partition $A_i\subset B(x_i,\delta)$. Set $p_i=\nu(A_i)$ and $q_i=\eta(A_i)$. The finite transport problem is
\begin{equation}
\min_{\substack{\pi_{ij}\ge0\\\sum_j\pi_{ij}=p_i\\\sum_i\pi_{ij}=q_j}}
\sum_{i,j}d(x_i,x_j)\pi_{ij}
=\max_{a_i+b_j\le d(x_i,x_j)}\left(\sum_i p_i a_i+\sum_jq_jb_j\right).
\end{equation}
This is finite linear-programming strong duality. The primal is nonempty since $\pi_{ij}=p_iq_j$ is feasible, and its feasible set is closed and bounded, so a minimum exists. Convex-cone separation and the Farkas alternative produce dual variables with the same value. This finite-dimensional separation is the finite-dimensional Hahn--Banach theorem, not infinite-dimensional transport duality.

For feasible $(a,b)$ put $u_i=\min_j(d(x_i,x_j)-b_j)$. Then $a_i\le u_i$, and choosing $j=i$ gives $u_i\le-b_i$. The triangle inequality yields $u_i\le d(x_i,x_l)+u_l$; interchange $i,l$ to obtain $|u_i-u_l|\le d(x_i,x_l)$. Hence
\begin{equation}
\sum_i p_i a_i+\sum_jq_jb_j\le\sum_i(p_i-q_i)u_i.
\end{equation}
Conversely, any unit Lipschitz grid function $u$ makes $(a_i,b_j)=(u_i,-u_j)$ feasible. The finite transport value therefore equals the grid Lipschitz dual. The function
$\widetilde u(x)=\min_i(u_i+d(x,x_i))$ is a unit Lipschitz extension with $\widetilde u(x_i)=u_i$.

Pushing an original coupling onto the grid by the partition changes cost by at most $2\delta$. Conversely, lift each $\pi_{ij}$ of a finite coupling $\pi$ to
$\pi_{ij}(\nu|_{A_i}/p_i)\otimes(\eta|_{A_j}/q_j)$. If $p_iq_j=0$ then necessarily $\pi_{ij}=0$, and omit that term. The resulting measure has the original marginals and adds at most $2\delta$ to cost. Original and grid transport values therefore differ by at most $2\delta$. For unit Lipschitz $u$ their pairings differ by at most $2\delta$, since each marginal integral error is at most $\delta$. Let $\delta\downarrow0$ to obtain KR equality. For a zero-mass finite signed measure, the positive and negative parts have equal mass; normalize by that mass, apply the probability version, and scale back. If the mass is zero, both parts vanish and the result is immediate.

\subsection{Cosine completeness and identification of ordinary weak gradients}
For $u\in H^1((0,1)^2)$ reflect evenly in both coordinates to $(-1,1)^2$ and extend two-periodically. Reflection has Jacobian of absolute value one and transforms weak derivatives by the reflection signs. Thus on a fundamental periodic cell,
\begin{equation}
\|\widetilde u\|_{L^2((-1,1)^2)}^2=4\|u\|_2^2,
\qquad\|\nabla\widetilde u\|_{L^2((-1,1)^2)}^2=4\|\nabla u\|_2^2.
\end{equation}
Verify this first for smooth functions, then use $H^1$ density of smooth functions on rectangles. Across reflection seams the traces coincide and the piecewise integration boundary terms cancel. Opposite cell traces also match, so the weak gradient has no additional seam measure.

The two-periodic complex Fourier basis is $e^{i\pi k\cdot x}$. Parseval for the function and each weak derivative gives derivative coefficients $i\pi k_j\widehat {\widetilde u}(k)$. Even symmetry excludes sine modes in both coordinates; the remaining normalized cosine tensor basis is exactly $\varphi_k$. Dividing cell integrals by four gives
\begin{equation}
\|u\|_2^2=\sum_{k\in\mathbb N_0^2}|u_k|^2,
\qquad\|\nabla u\|_2^2=\sum_{k\ne0}\pi^2|k|^2|u_k|^2.
\end{equation}
Both summable tails vanish, so finite cosine projections converge in $H^1$. This proves the spectral form of the ordinary Sobolev norm and shows that the coefficient-defined Poisson potential has its ordinary weak gradient, not a separately defined formal gradient. For every mean-zero $L^2$ source, $L^2$ continuity on both sides passes finite-spectral Poisson identities to all $H^1$ tests, yielding weak zero normal flux.

\subsection{Gaussian rotation concentration without Brownian differentiation}
\begin{lemma}\label{lem:rotation}
If $A\sim N(0,I_d)$ and $F:\mathbb R^d\to\mathbb R$ is $R$-Lipschitz, then for every $s\in\mathbb R$,
\begin{equation}\label{eq:rotation-mgf}
\E\exp\{s(F(A)-\E F(A))\}\le\exp\left(\frac{\pi^2R^2s^2}{8}\right).
\end{equation}
The constant is independent of dimension.
\end{lemma}
\begin{proof}
First let $F$ be smooth. Take independent standard Gaussians $A,B$ and, for $0\le\theta\le\pi/2$, set
\begin{equation}
Z_\theta=\cos\theta A+\sin\theta B,
\qquad V_\theta=-\sin\theta A+\cos\theta B.
\end{equation}
This is an orthogonal transformation in $\mathbb R^{2d}$, so the joint law of $(Z_\theta,V_\theta)$ remains the product of independent standard Gaussians. The chain rule and fundamental theorem of calculus give
\begin{equation}
F(B)-F(A)=\int_0^{\pi/2}DF(Z_\theta)[V_\theta]\dd\theta.
\end{equation}
Normalize $\dd\theta$ to probability measure $2\dd\theta/\pi$ and apply Jensen to the exponential:
\begin{equation}
e^{s(F(B)-F(A))}\le\frac2\pi\int_0^{\pi/2}
\exp\{(s\pi/2)DF(Z_\theta)[V_\theta]\}\dd\theta.
\end{equation}
Conditional on $Z_\theta=z$, the linear functional $DF(z)[V_\theta]$ is centered Gaussian with variance $\|DF(z)\|^2\le R^2$. Its exponential moment is therefore at most
$\exp((s\pi/2)^2R^2/2)$. Tonelli for the nonnegative integrand gives this same bound for the exponential moment of the difference. Fix $B$, apply Jensen over $A$, and integrate $B$ to obtain the centered bound~\eqref{eq:rotation-mgf}. Linear growth $|F(x)|\le|F(0)|+R|x|$ ensures finiteness of all exponential moments used.

For general Lipschitz $F$ take a nonnegative smooth mollifier supported in the unit ball and put $F_\varepsilon=F*\rho_\varepsilon$. Integration gives $|F_\varepsilon-F|\le R\varepsilon$; the function $F_\varepsilon$ remains $R$-Lipschitz, so $\|DF_\varepsilon\|\le R$. The centered functions differ by at most $2R\varepsilon$. For $0<\varepsilon\le1$ their exponentials are dominated by $C_s\exp(|s|R|A|)$, integrable because
$c|x|\le a|x|^2+c^2/(4a)$ and Gaussian square-exponential integrability hold for $0<a<1/2$. Dominated convergence passes the smooth conclusion to $F$.
\end{proof}

The rotation argument gives a dimension-independent constant. The optimal constant $R^2/2$ from the exponential-martingale argument above gives the same localization order. We use Lemma~\ref{lem:rotation} below to check the implementation's constants explicitly.

\subsection{Explicit grid-covering constants with actual test centers}
For integer $m\ge1$ encode at vertices $(i/m,j/m)$ the integers $\lfloor m u(i/m,j/m)\rfloor$. Since $u\in\mathcal K$ satisfies $\|u\|_\infty\le\sqrt2$, the root has at most $4m+3$ codes. Connect vertices by a row-by-row nearest-neighbor tree. Parent and child original values differ by at most $1/m$, so their integer codes differ by at most two. Counting the seven encompassing values $-3,\ldots,3$ suffices. Nonempty code classes number at most
\begin{equation}
(4m+3)7^{(m+1)^2}\le7^{m+(m+1)^2}\le7^{5m^2}.
\end{equation}
For the first inequality $4m+3\le7^m$ follows from equality at $m=1$ and induction; the second is equivalent to $4m^2-3m-1\ge0$.

Two functions in a code class differ at each vertex by less than $1/m$. Every point is within $\sqrt2/m$ of a vertex, so their uniform distance is less than $(1+2\sqrt2)/m<5/m$. Choose an \emph{original test} in each nonempty class, rather than an arbitrary interpolant. For $0<\varepsilon\le1$ take $m=\lceil5/\varepsilon\rceil$, so $5/m\le\varepsilon$ and $m\varepsilon\le6$. The covering number $N_\varepsilon$ satisfies
\begin{equation}\label{eq:explicit-entropy}
\log N_\varepsilon\le5m^2\log7\le\frac{180\log7}{\varepsilon^2}.
\end{equation}
To avoid extrema on nonclosed code classes, the actual implementation intersects closed uniform-norm balls about these representatives with $\mathcal K$. The intersections are nonempty compact sets and still cover $\mathcal K$.

The local support functions have Lipschitz constant $R=\sqrt2$ in finite Gaussian coordinates. For finitely many such functions $Y_i$, Lemma~\ref{lem:rotation}, the finite-sum bound, and Jensen give, for any $a>0$,
\begin{equation}
\E\max_iY_i\le\max_i\E Y_i+\frac{\log N_\varepsilon}{a}+\frac{\pi^2a}{4}.
\end{equation}
Choose $a=\varepsilon^{-1}$. By~\eqref{eq:explicit-entropy} and $\varepsilon\le1$,
\begin{equation}
F_N(L)\le F_\varepsilon(L)+\frac{180\log7+\pi^2/4}{\varepsilon}.
\end{equation}
Add the boundary error $(\varepsilon/\delta+8\delta)\sqrt m$ from the main proof and choose $\varepsilon=m^{-1/3}$ and $\delta=m^{-1/6}$. For $m>64$ all boundary-parameter conditions hold, giving
\begin{equation}
F_N(L)\le F_{0,N}(L)+C_1m_N(L)^{1/3},
\qquad C_1=180\log7+\pi^2/4+9.
\end{equation}
This constant includes both global rescaling and boundary-strip losses.

\subsection{Explicit heat-trace and sharp-cutoff error constants}
Let $\Theta(s)=\sum_{j\in\mathbb Z}e^{-\pi s j^2}$. Gaussian Fourier transform and Poisson summation give
$\Theta(s)=s^{-1/2}\Theta(s^{-1})$. Since Neumann indices form the nonnegative quadrant, for $0<t\le1$ the heat trace including the constant mode satisfies
\begin{equation}
\sum_{k\in\mathbb N_0^2}e^{-t\pi^2|k|^2}
\le\Theta(\pi t)^2
\le\frac{\Theta(\pi^{-1})^2}{\pi t}=\frac{C_{\rm heat}}t,
\qquad C_{\rm heat}=\frac{\Theta(\pi^{-1})^2}{\pi}.
\end{equation}
The last step uses $1/(\pi t)\ge1/\pi$ and monotonic decrease of each nonnegative theta term with its parameter. Gaussian tails make the fixed-parameter theta series absolutely convergent, so the constant is finite and positive. Also $u\lambda e^{-u\lambda}\le e^{-1}$, hence
\begin{equation}
\sum_{k\ne0}\lambda_ke^{-2u\lambda_k}
\le\frac{C_{\rm heat}}{eu^2}\le\frac{eC_{\rm heat}}{u^2}.
\end{equation}
The last weaker estimate agrees with the implementation's weighted heat-trace bound.

For $L\ge1$ set $x=\lambda/L$. When $0<x\le1$, $(1-e^{-x})^2/x\le x\le e e^{-x}$; when $x>1$, $e^{-2x}/x\le e^{-x}$. Consequently
\begin{equation}
\sum_{k\ne0}\frac{(\mathbf1_{\{\lambda_k\le L\}}-e^{-\lambda_k/L})^2}{\lambda_k}
\le\frac eL\sum_{k\ne0}e^{-\lambda_k/L}\le eC_{\rm heat}.
\end{equation}
The expected difference of actual empirical heat and sharp support functions is at most $\sqrt{eC_{\rm heat}}$, an explicit version of $O(1)$ in the main proof.

\subsection{Continuous-function-valued empirical heat derivatives and the actual integral identity}
Fix $n\ge1$, the sample $\omega$, and $0<s\le t$. The original normalized coefficients
$a_{n,k}=n^{-1/2}\sum_{i=1}^n\varphi_k(X_i)$ satisfy $|a_{n,k}|\le2\sqrt n$. Set
\begin{equation}
D_n(u)=\sum_{k\ne0}\lambda_ke^{-u\lambda_k}a_{n,k}\varphi_k\in C(Q),\qquad u\in[s,t].
\end{equation}
Each term has $C(Q)$ norm at most $4\sqrt n\lambda_ke^{-s\lambda_k}$. This envelope is summable: splitting into bands $j\le\lambda<j+1$ bounds the sum by $C\sqrt n\sum_{j\ge0}(j+1)^2e^{-sj}$. The Weierstrass test gives uniform convergence in $C(Q)$ norm on $[s,t]$. Finite terms are time-continuous, so $D_n$ is a continuous $C(Q)$-valued function and has an actual Bochner integral.

The empirical heat source $P_u^NZ_n$ is likewise constructed from a uniformly summable series with coefficients $e^{-u\lambda_k}a_{n,k}$. It equals the finite heat-kernel average
\begin{equation}
\frac1{\sqrt n}\sum_{i=1}^n\bigl(p_N(u,\cdot,X_i)-1\bigr).
\end{equation}
Integrating the kernel modewise gives the same coefficients. Cosine completeness identifies the two in $L^2$, so their continuous representatives agree everywhere on the square. The coefficient functional $f\mapsto\int f\varphi_k$ is continuous on $C(Q)$ and commutes with the Bochner integral. Thus
\begin{equation}
\left\langle\int_s^tD_n(u)\dd u,\varphi_k\right\rangle
=(e^{-s\lambda_k}-e^{-t\lambda_k})a_{n,k}.
\end{equation}
The constant coefficient vanishes. Completeness again gives
\begin{equation}\label{eq:strong-heat-integral}
P_s^NZ_n-P_t^NZ_n=\int_s^tD_n(u)\dd u
\end{equation}
at least in $L^2(Q)$, and the uniform representations in fact give equality in the continuous-function space. Inverse flux is a continuous linear map from $L^2_0$ to the gradient Hilbert space, with norm at most $1/\pi$, so it commutes with this integral. The $H_N^{-1}$ identity, flux triangle inequality, and all Lipschitz pairings in the main proof therefore involve the same actual source.

\subsection{Joint measurability, Fubini, and integrability of actual heat displacement}
Finite terms $(u,\omega)\mapsto\lambda_ke^{-u\lambda_k}a_{n,k}(\omega)\varphi_k$ are strongly measurable because $a_{n,k}$ is a finite sum of measurable sample evaluations. Norm limits of the summable series are strongly measurable, as are their continuous linear inverse fluxes. On $[s,t]\times\Omega$ the deterministic envelope gives an integrable constant bound; both source and flux are product-integrable. Tonelli and empirical covariance give the nonnegative spectral energy sum
\begin{equation}
\E\|J_N(D_n(u))\|_2^2=\sum_{k\ne0}\lambda_ke^{-2u\lambda_k}
\le\frac{eC_{\rm heat}}{u^2}.
\end{equation}
Probability Cauchy--Schwarz, Fubini, and the integral norm bound give
\begin{equation}
\E\left\|\int_s^tJ_N(D_n(u))\dd u\right\|_2
\le\sqrt{eC_{\rm heat}}\log(t/s).
\end{equation}
Only after division by $\sqrt n$ do we recover the transport scale for the two smoothed probability measures. For $s=1/n$ and $t=(\log n)^B/n$ we may therefore choose
\begin{equation}
C_B=\sqrt\pi+\sqrt{eC_{\rm heat}}\,|B|
\end{equation}
to control $C_B(1+\log\log n)/\sqrt n$ in equation $(4.11)$ of the main proof.

We must also check that expectations involve measurable transport quantities. For finite prefixes $x=(x_i)$ and $y=(y_i)$, empirical distances differ by at most $n^{-1}\sum_i|x_i-y_i|$, using the atomwise coupling. Heat contraction gives the same bound after smoothing. The variation of $x\mapsto W_1^Q(\mu_n(x),P_t^N\mu_n(x))$ is thus at most $2n^{-1}\sum_i|x_i-y_i|$, making it continuous. Composition with the original measurable prefix is measurable, bounded by diameter $\sqrt2$, and hence integrable. The same applies to distances between two heat-smoothed measures.

\subsection{Finite-dimensional Gaussian laws, quadratic transport, and varying covariance}
The Euclidean norm of coordinates $y_k=\varphi_k(X)/\sqrt{\lambda_k}$ equals the actual flux's $L^2$ norm by gradient orthogonality. Covariance is the nonnegative diagonal matrix $\Sigma_L=(\lambda_k^{-1})$. For each fixed $L$ the Gaussian vector $(g_k/\sqrt{\lambda_k})$ has precisely this law. Mapping coordinates back to the source gives the same Neumann cutoff field, not merely a field with nearby covariance.

Zhai's source theorem~\cite{Zhai} gives the explicit bound $5\sqrt d\,\beta(1+\log n)/\sqrt n$. For $n\ge2$, the inequality $1+\log n\le(1+1/\log2)\log n$ yields the universal-constant form used above. If covariance degenerates, each centered zero-variance direction vanishes almost surely. Restrict to the positive-variance support and embed isometrically back; dimension can only decrease, so the same bound holds. No $\lambda_{\min}(\Sigma_L)^{-1}$ occurs, allowing every varying $L_n$ separately without compatible couplings across rows.

For unit Lipschitz $\Phi$, each coupling gives
$|\E\Phi(U)-\E\Phi(V)|\le(\E|U-V|^2)^{1/2}$. Taking the infimum gives control by $W_2$. This mean inequality remains valid even without selecting an optimal coupling on one probability space. Actual coordinates have finite second moments, and linear growth of $\Phi$ ensures integrability.

\subsection{A second verified spectral-band reduction}
The logarithmic flux-integral route has been proved in the main text. Another checkable bound uses fine heat smoothing at $1/n$, heat-to-sharp error at cutoff $n$, and nested spectral bands. For $1\le L\le n$ empirical second moments give
\begin{equation}
\E\|J_N(\Pi_nZ_n-\Pi_LZ_n)\|_2^2=m_N(n)-m_N(L).
\end{equation}
Flux stability and the triangle inequality therefore yield
\begin{equation}
\left|\sqrt n\,\E W_n-\E\mathcal S(\Pi_LZ_n)\right|
\le\sqrt\pi+\sqrt{eC_{\rm heat}}+\sqrt{m_N(n)-m_N(L)}.
\end{equation}
Take $L=L_n$. The two trace asymptotics give $m_N(n)-m_N(L_n)=B\log\log n/(4\pi)+O(1)$. Division by $\sqrt{\log n}$ makes this error vanish. Gaussian replacement then gives the same mean limit. This independent check does not replace the actual flux integral and nonnegative smoothing remainder proved above.

\subsection{Conditional expectations on arbitrary original probability spaces}
The probability upgrade uses only independence and a common law for measurable $X_i$. Every empirical quantity is a Borel function $F_n:Q^n\to\mathbb R$ of a finite prefix; measurability and boundedness of the distance were proved above. Independence makes the joint prefix law exactly $\mu^{\otimes n}$. For $0\le i\le n$ define
\begin{equation}
T_i(x_1,\ldots,x_i)=\int_{Q^{n-i}}F_n(x_1,\ldots,x_n)
\dd\mu(x_{i+1})\cdots\dd\mu(x_n).
\end{equation}
Fubini gives Borel measurability. Compose with the original prefix. For every bounded $\mathcal F_i$-measurable test, independence and iterated integration prove the defining conditional-expectation identity. This explicit function is thus a version of conditional expectation on the original space; no regular conditional distribution on $\Omega$ is required.

The coordinate oscillation bound survives integration over the same remaining sample. Prove conditional Hoeffding first by $\mu$ integration for fixed $x_1,\ldots,x_{i-1}$, then compose with the original prefix. This justifies every conditional-moment, bounded-differences, and maximal-Azuma application above. Conditional inequalities hold first for individual tests; a countable union of null sets gives common versions. The bounds $B_n\ge0$, $B_n\le\sqrt2$, and the martingale-increment bound can likewise hold on one full-measure event, making the compensator pathwise nondecreasing. After Borel--Cantelli we intersect only over countably many $\varepsilon$ and geometric ratios, without asserting simultaneous full-measure statements for uncountable parameters.

If needed, completion adds only null sets and preserves all integrable expectations and independence. The final limit event is expressed by liminf and limsup of measurable real functions and belongs to the original $\mathcal A$, so its probability-one conclusion returns to the original space. Universally measurable periodic truncation selections apply only to the auxiliary Gaussian law and impose no standard-Borel assumption on the original probability space.

\clearpage
\bibliographystyle{plain}
\bibliography{references}

\end{document}